\documentclass{article}

\usepackage{microtype}
\usepackage{graphicx}
\usepackage{booktabs} % for professional tables
\usepackage{bbding}
\usepackage{dsfont}
\usepackage{subcaption}
\usepackage{algorithm}
\usepackage{algorithmic}
\newcommand{\alglinelabel}{%
  \addtocounter{ALC@line}{-1}% Reduce line counter by 1
  \refstepcounter{ALC@line}% Increment line counter with reference capability
  \label% Regular \label
}
\usepackage{hyperref}

\usepackage{arXiv}

\usepackage{amsmath}
\usepackage{amssymb}
\usepackage{mathtools}
\usepackage{amsthm}

\usepackage[capitalize,noabbrev]{cleveref}

\theoremstyle{plain}
\newtheorem{theorem}{Theorem}[section]

\newtheorem{lemma}[theorem]{Lemma}
\newtheorem{corollary}[theorem]{Corollary}
\theoremstyle{definition}

\newtheorem{assumption}[theorem]{Assumption}
\theoremstyle{remark}
\newtheorem{remark}[theorem]{Remark}

\usepackage[textsize=tiny]{todonotes}

\def\grad{\nabla}

\def\ba{\mathbf{a}}

\def\bv{\mathbf{v}}
\def\bw{\mathbf{w}}
\def\bx{\mathbf{x}}  %{\mbox{\boldmath $\lambda$}}
\def\by{\mathbf{y}}
\def\bz{\mathbf{z}}

\def\cA{\mathcal{A}}

\def\cD{\mathcal{D}}
\def\cE{\mathcal{E}}

\def\cG{\mathcal{G}}

\def\cK{\mathcal{K}}
\def\cL{\mathcal{L}}

\def\cO{\mathcal{O}}
\def\cP{\mathcal{P}}

\def\cR{\mathcal{R}}

\def\cT{\mathcal{T}}

\def\cX{\mathcal{X}}
\def\cY{\mathcal{Y}}

\def\smskip{\smallskip}

\def\texitem#1{\par\smskip\noindent\hangindent 25pt
               \hbox to 25pt {\hss #1 ~}\ignorespaces}

\def\norm#1{\|#1\|}

\newcommand{\BEAS}{\begin{eqnarray*}}
\newcommand{\EEAS}{\end{eqnarray*}}
\newcommand{\BEA}{\begin{eqnarray}}
\newcommand{\EEA}{\end{eqnarray}}
\newcommand{\BEQ}{\begin{eqnarray}}
\newcommand{\EEQ}{\end{eqnarray}}
\newcommand{\BIT}{\begin{itemize}}
\newcommand{\EIT}{\end{itemize}}
\newcommand{\BNUM}{\begin{enumerate}}
\newcommand{\ENUM}{\end{enumerate}}

\newcommand{\BA}{\begin{array}}
\newcommand{\EA}{\end{array}}

\newcommand{\reals}{\mathbb{R}}
\newcommand{\integers}{\mathbb{Z}}

\def\fprod#1{\left\langle#1\right\rangle}
\def\prox#1{\mathbf{prox}_{#1}}
\DeclareMathOperator*{\argmax}{\mathbf{argmax}}
\DeclareMathOperator*{\argmin}{\mathbf{argmin}}

\newcommand{\dom}{\mathop{\bf dom}}

\usepackage{xcolor}

\newif\ifpagenumbering
\pagenumberingtrue

\pagenumberingfalse
\newsavebox{\theorembox}
\newsavebox{\lemmabox}
\newsavebox{\defnbox}
\newsavebox{\assbox}
\savebox{\theorembox}{\noindent\bf Theorem}
\savebox{\lemmabox}{\noindent\bf Lemma}
\savebox{\defnbox}{\noindent\bf Definition}
\newtheorem{defn}{\usebox{\defnbox}}

\usepackage{soul}
\usepackage[normalem]{ulem}
\DeclareUnicodeCharacter{2212}{~}
\usepackage{hyperref}
\usepackage{cleveref}
\crefname{assumption}{Assumption}{Assumptions}
\crefname{theorem}{Theorem}{Theorems}
\crefname{lemma}{Lemma}{Lemmas}

\usepackage{relsize}

\def\xz#1{\textcolor{teal}{#1}}
\def\sa#1{\textcolor{purple}{#1}}

\def\xqs#1{\textcolor{cyan}{#1}}
\def\mg#1{\textcolor{magenta}{#1}}

\def\bby{\Bar{\mathbf{y}}}

\def\shortalgline#1{\texttt{line}~\red{#1}}
\def\algline#1{\texttt{line}~\red{#1} of \agdap}

\definecolor{darkgreen}{rgb}{0.2,0.8,0.1}
\def\mg#1{\textcolor{blue}{#1}}

\newcommand{\gda}{\texttt{GDA}{}}
\newcommand{\agda}{\texttt{AGDA}{}}
\newcommand{\smagda}{\texttt{Sm-AGDA}{}}
\newcommand{\tiada}{\texttt{TiAda}{}}
\newcommand{\neada}{\texttt{NeAda}{}}
\newcommand{\sgdab}{\texttt{SGDA-B}{}}

\def\xz#1{\textcolor{black}{#1}}
\def\sa#1{\textcolor{black}{#1}}

\def\xqs#1{\textcolor{black}{#1}}
\def\mg#1{\textcolor{black}{#1}}
\def\xzr#1{\textcolor{blue}{#1}}

\def\na#1{\textcolor{black}{#1}}

\def\rv#1{\textcolor{black}{#1}}

\newcommand{\agdap}{\texttt{AGDA+}{}}

\def\algline#1{\texttt{line}~{#1} of \agdap}

\def\shortalgline#1{\texttt{line}~{#1}}
\def\cond#1{$\texttt{BacktrackCond}(#1)$}
\def\N{\mathbb{N}}
\def\ms{\texttt{max\_solver}}
\def\msf#1{$\texttt{max\_solver}=\texttt{#1}$}

\usepackage{pifont}

\usepackage{breqn}
\usepackage{todonotes}
\usepackage{etoolbox}
\makeatletter
\patchcmd{\@addmarginpar}{\ifodd\c@page}{\ifodd\c@page\@tempcnta\m@ne}{}{}
\makeatother
\reversemarginpar
\usepackage{enumitem}
\usepackage{array}
\newcolumntype{P}[1]{>{\centering\arraybackslash}p{#1}}
\usepackage{graphicx}

\makeatletter
\newcommand{\thickhline}{%
    \noalign {\ifnum 0=`}\fi \hrule height 1pt
    \futurelet \reserved@a \@xhline
}
\newcolumntype{"}{@{\hskip\tabcolsep\vrule width 1pt\hskip\tabcolsep}}
\makeatother

\title{\sa{\agdap{}: proximal alternating gradient descent ascent method with \na{a nonmonotone adaptive step-size search %backtracking 
for 
nonconvex minimax problems}}}

\author{
Xuan Zhang\thanks{The student authors contributed equally.}\\
Department of Industrial and Manufacturing Engineering\\
Pennsylvania State University
\\
University Park, PA,USA.\\
\texttt{xxz358@psu.edu}
\And
Qiushui Xu\footnotemark[1]\\
Department of Industrial and Manufacturing Engineering\\
Pennsylvania State University
\\
University Park, PA,USA.\\
\texttt{qjx5019@psu.edu}
\And
Necdet Serhat Aybat\\
Department of Industrial and Manufacturing Engineering\\
Pennsylvania State University
\\
University Park, PA,USA.\\
\texttt{nsa10@psu.edu}
}
\date{}
\makeatother

\begin{document}
\maketitle  
\begin{abstract}
%\xtodo{Shall we add equal contribution under Qiu Shui and my name? If it is ok, can Qiushui implement this change?}
We consider double-regularized \na{nonconvex}-strongly concave~(NCSC) minimax problems of the form $(P):\min_{\mathbf{x}\in\mathcal{X}}\max_{\mathbf{y}\in\mathcal{Y}} g(\mathbf{x})+f(\mathbf{x}, \mathbf{y})-h(\mathbf{y})$, where $g,h$ are closed convex, $f$ is \na{$L$-smooth} in $(\mathbf{x},\mathbf{y})$ %with Lipschitz constant $L$ 
and strongly concave in $\mathbf{y}$. We propose a 
proximal alternating gradient descent ascent method {\texttt{AGDA+}} that can adaptively choose \na{nonmonotone} primal-dual stepsizes to compute an approximate stationary point for $(P)$ without requiring the knowledge of the \na{global Lipschitz constant $L$ {and the concavity modulus $\mu$}}. %This algorithm 
\na{Using a nonmonotone step-size search (backtracking) scheme, \agdap{}} stands out by its ability to exploit the local Lipschitz structure and eliminates the need for precise tuning of hyper-parameters.
%For  deterministic scenarios,
{\texttt{AGDA+}} achieves the optimal iteration complexity of {$\mathcal{O}(\epsilon^{-2})$} %with the best $\mathcal{O}(1)$ constant among the existing methods agnostic to $L$.
and \na{it is the first step-size search method for NCSC minimax problems that require only $\cO(1)$ calls to $\grad f$ per backtracking iteration.}  
The numerical experiments demonstrate its robustness and efficiency. 
%in both deterministic and stochastic scenarios. 
% To the best of our knowledge, this is the first single-loop backtracking method for solving %nonconvex-strongly concave 
% WCSC minimax problems.
\end{abstract}
\section{Introduction}
\label{sec:intro}
\sa{Many machine learning~(ML) applications, such as training deep neural networks~(DNN), use adaptive gradient methods, e.g., AdaGrad~\cite{duchi2011adaptive} and Adam~\cite{KingBa15}, as they are robust to hyper-parameter choices and exhibit reasonable convergence behavior in practice even for \na{non-convex problems}. %settings. 
As also discussed in~\cite{li2022tiada}, a common characteristic of the majority of these methods is that they use \textit{diminishing} step sizes inversely proportional to the cumulative sum of gradient norms from previous iterations, \na{and} their step sizes eventually go below $2/L$ (this threshold is also required by the analysis of the gradient descent \na{method for smooth minimization problems}). %and this 
\na{This choice of monotonic stepsizes} guarantees the convergence of the method while being agnostic to the parameter $L$, which is usually \textit{unknown} and hard to estimate for complex ML models such as DNN.} \sa{While adaptive gradient methods for minimization problems have been actively studied, this is not the case for %adaptive 
primal-dual methods for solving non-convex minimax problems; indeed, even the deterministic setting is theoretically understudied, and there is a need for %single-loop 
simple primal-dual methods \na{that are} agnostic to the usually \textit{unknown global Lipschitz constant $L$} {and \textit{the concavity modulus} $\mu$,} and \na{that have} convergence guarantees for adaptive step sizes which do not tend to $0$ --as this would make them exhibit faster convergence behavior in practice while still being easy to tune. \na{Moreover, for better practical performance, it is highly desirable to design a method with a \textit{nonmonotone} step-size search mechanism in order to exploit the fluctuations of the local Lipschitz constants over the problem domain.}}

\sa{In this paper, we study the following class of \sa{\emph{double regularized nonconvex-strongly concave minimax}} problems:
\begin{equation}
\begin{aligned}\label{eq:main-problem}
\min_{\mathbf{x} \in \mathcal{X}} \max _{\by \in \mathcal{Y}} \mathcal{L}(\mathbf{x}, \by) \triangleq %\sum_{i \in \mathcal{N}} g_i\left(x_i\right)
g(\bx)+f(\mathbf{x}, \by)-h(\by),
\end{aligned}
\end{equation}
where $\mathcal{X}$ and $\mathcal{Y}$ are finite dimensional Euclidean spaces, $f:\cX\times\cY\to \mathbb{R}$ is differentiable \sa{such that $\grad f$ is Lipschitz with constant $L$,} $f$ is possibly non-convex in $\bx$ and strongly concave in $\by$ with modulus \sa{$\mu>0$},
%\nsa{did we get results for $\mu=0$?} 
$h:\cY\to\reals\cup\{+\infty\}$ and $g:\cX\to\reals\cup\{+\infty\}$  are closed convex functions. %If $\mu>0$, then 
\sa{The minimax problem in~\eqref{eq:main-problem} belongs to the class of \textit{weakly convex-strongly concave} (WCSC) minimax problems.}}
%, whereas for $\mu=0$, it is called \textit{weakly convex-merely concave}~(WCMC).
\sa{Nonconvex optimization problems in the form of \eqref{eq:main-problem} arise in many applications such as constrained optimization of weakly-convex objectives based on Lagrangian duality~\cite{li2021augmented}, Generative Adversarial Networks (GAN)~\cite{goodfellow2014generative},}
distributionally robust optimization (DRO)~\cite{namkoong2016stochastic}, and reinforcement learning~\cite{NEURIPS2021_d994e372}. 

\rv{For the minimax problem in~\eqref{eq:main-problem},} \sa{we are interested in designing a \textit{proximal gradient \na{descent ascent}}~(GDA) type method, i.e., given $(\bx_0,\by_0)\in\dom g\times\dom h$, for $t\geq 0$, let
%$\{(\bx_t,\by_t)\}_{t\in\integers_+}$ is constructed via
\begin{subequations}
\label{eq:gda}
    \begin{align}
        \bx_{t+1}&=\prox{\tau_t g}\big(\bx_t-\tau_t\grad_{\bx} f(\bx_t,\by_t)\big), \label{eq:x-update}\\
        \by_{t+1}&=\prox{\sigma_t h}\big(\by_t+\sigma_t\grad_y f(\bx_{t+1},\by_t)\big), \label{eq:y-update}
    \end{align}
\end{subequations}
such that the primal-dual step size sequence $\{(\tau_t,\sigma_t)\}_{t\in\integers_+}$ will be \textit{chosen on the fly} without requiring to know the \textit{global} Lipschitz constant $L$ and/or the concavity modulus $\mu$, where $\prox{r}(\cdot)$ denotes the proximal mapping of a given proper, closed, convex function $r:\cX\to\reals\cup\{+\infty\}$, i.e., \na{using the Euclidean norm $\norm{\cdot}$, we define %the proximal map
$\prox{r}$ as follows:}
\begin{equation}
\label{eq:prox}
\prox{r}:\bx\mapsto\argmin_{\bw\in\cX}\{r(\bw)+\frac{1}{2}\norm{\bw-\bx}^2\}.
\end{equation}}%
%\nsa{Use the first paragraph of \tiada{} to motivate adaptive gradient methods.}
% \todo[inline]{NSA: From \tiada{} paper ``To sum up,
% adaptive stepsizes designed for minimization, are not time-scale adaptive for minimax optimization and thus
% not parameter-agnostic."}
% \xtodo{I found this has been written in the later paragraph, do I still need to add this?}
\paragraph{Preliminaries: Assumptions and Definitions}
\label{sec:preli}
\begin{assumption}\label{ASPT:lipshiz gradient}
Let $g:\mathcal{X}\rightarrow \reals\cup\{+\infty\}$
and 
$h:\mathcal{Y}\rightarrow \reals\cup\{+\infty\}$ 
be proper, closed, convex functions.
Moreover, let 
$f:\cX\times\cY\to\reals$  
such that $f$ is differentiable on an open set containing $\dom g \times \dom h$ such that $\grad f$ is $L$-Lipschitz on $\dom g \times \dom h$, and that $f(\bx,\cdot)$ is strongly concave with modulus \sa{$\mu> 0$ uniformly for \rv{all} $\bx\in \dom h\subset \mathcal{Y}$.}
\end{assumption}

\begin{assumption}\label{aspt:bounded-Y}
\rv{Suppose} $\cD_y \triangleq \sup_{\by_1,\by_2\in \dom g} \norm{\by_1 - \by_2} < \infty$.
%\nsa{do we use it anywhere? --except for the definition of $\cD_y$}
\end{assumption}

% \begin{assumption}
%  $\cL(x,y)\geq \underline{\cL}$   for all $x,y\in \cX\times\cY$.
% \end{assumption}
% \qstodo{need to discuss if we need this assumption}

\begin{defn}
\label{def:opt-response}
    Define $\na{{\by}^{*}}:\cX\to\cY$ such that ${\by}^{*}(\bx)\triangleq\argmax_{\by\in\cY}f(\bx,\by)-h(\by)$ for all $\bx\in\dom g$.
\end{defn}
\begin{defn}
\label{def:max-function}
Let $F:\dom g\to \reals\cup\{+\infty\}$ denote the primal function, i.e., $F(\bx)\triangleq g(\bx)+\Phi(\bx)$, where the max function \rv{$\Phi$ is defined as} $\Phi(\bx)\triangleq \max_{\by\in\cY} f(\bx,\by)-h(\by)$ for all $\bx\in\rv{\dom g}$. 
\end{defn}
%\nsa{We removed lower boundedness assumption on $\cL$}
\begin{assumption}\label{aspt:primal_lb}
\sa{Suppose $-\infty<F^*\triangleq \rv{\inf_{\bx\in\cX}}F(\bx)$.}
\end{assumption}
\begin{defn}
%[1st order SP for min. problem]
\label{def:1st order SP for min. problem}
Given $\epsilon\geq 0$, \sa{$\bx_\epsilon\in\dom g$} is an $\epsilon$-stationary point in metric (\textbf{M1}) 
%of a differentiable function $\Phi$ 
if $\norm{\rv{\cG^\tau}(\bx_\epsilon)}\leq \epsilon$ \na{for some $\tau>0$}, where the gradient map $\cG^\tau(\bx) \triangleq [\bx- \prox{\tau g} (\bx-\tau \grad \Phi(\bx))]/\tau$ for any $\bx\in\rv{\dom g\subset}\cX$. %and $\Phi(\cdot) = \max_y f(\cdot,\by)-h(\by)$.
\end{defn}
\begin{defn}\label{def:prox}
Given $\tau,\sigma>0$, 
%and $(\bx,\by)\in\dom f\times \dom g$, 
define $G^{\na{\tau}}_x:\cX\times\cY\to\cX$ and $G^{\na{\sigma}}_y:\cX\times\cY\to\cY$ such that for any $(\bx,\by)\in\cX\times \cY$,
{\small
\begin{align*}
    G^{\na{\tau}}_x(\bx,\by)&\triangleq [\bx- \prox{\tau g} (\bx-\tau \grad_x f(\bx,\by))]/\tau,\\
    G^{\na{\sigma}}_y(\bx,\by)&\triangleq [\prox{\sigma h} (\by+\sigma \grad_y f(\bx,\by))-\by]/\sigma,
\end{align*}}%
where $\prox{\tau g}(\cdot)$ and $\prox{\sigma h}(\cdot)$ are defined as in~\eqref{eq:prox}. 
%$\prox{\tau g}(\bx')\triangleq \argmin_{\bx\in\cX}\tau g(\bx)+\frac{1}{2}\norm{\bx-\bx'}^2$ for all $\bx'\in\cX$, and $\prox{\sigma h}$ is defined similarly. 
%For notational simplicity, we also
Moreover, define $G^{\na{\tau,\sigma}}(\bx,\by) \triangleq [G^\tau_x(\bx,\by)^\top G^\sigma_y(\bx,\by)^\top]^\top$. 
\end{defn}
\begin{defn}
\label{def:eps-stationarity}
Given $\epsilon\geq 0$, \sa{$(\bx_\epsilon,\by_\epsilon)\in\dom g\times\dom h$} is an $\epsilon$-stationary in metric (\textbf{M2}) if {$\norm{G^{\tau,\sigma}(\bx_\epsilon,\by_\epsilon)}\leq \epsilon$} for some $\tau,\sigma>0$.
\end{defn}
% \begin{remark}
% We should note that 1st order SP for min. problem and 1st order SP are suitable for \cref{eq:main-problem} when $h$ and $g$ are closed convex functions; and 1st order NE for min. problem and 1st order NE are suitable for \cref{eq:main-problem} when $h,g \equiv 0$.
% \end{remark}
%\subsection{Relation among different stationarity notions}
%\xtodo{Where should we put this paragraph}
\begin{remark}
\sa{Metrics (\textbf{M1}) and (\textbf{M2}) are equivalent under \cref{ASPT:lipshiz gradient}; indeed, %for the case 
when $g=h=0$, 
%Proposition 2.1 in~
\cite[Proposition 2.1]{yang2022faster} implies that %the results in 
each metric can be translated to the other requiring at most $\cO(\kappa\log(1/\epsilon))$ additional gradient calls, which is negligible when compared to $\cO(\epsilon^{-2})$ complexity of computing an $\epsilon$-stationary point in either metric.} 
% More precisely, given $x'$ such that $\norm{\grad \Phi(x')}\leq \epsilon$, one can compute $y'$ such that $\norm{\grad f(x',y')}=\cO(\epsilon)$ and it requires additional $\cO(\kappa \log(\frac{1}{\epsilon}))$ gradient calls.  Conversely, given $(x'', y'')$ such that $\norm{\grad_x f(x'',y'')}\leq \epsilon$ and $\norm{\grad_y f(x'',y'')}\leq \epsilon/\kappa$, one can compute $x'$ such that $\norm{\grad \Phi(x')}=\cO(\epsilon)$ and it requires additional $\cO(\kappa \log(\kappa))$ gradient calls.
\end{remark}
\paragraph{Related work.}
{%For the case 
When the \textit{global} Lipschitz constant ($L$) \rv{and the concavity modulus ($\mu$) are} \textit{known}, there are many theoretically efficient methods in the literature, e.g., see~\cite{lin2020near,ostrovskii2021efficient,kong2021accelerated} and references therein. Indeed, when $L$ {and $\mu$ are}  known, both Lin \emph{et al.}~\cite{lin2020near} and Ostrovskii \emph{et al.}~\cite{ostrovskii2021efficient} show an iteration complexity of $\tilde\cO(L\kappa^{1/2}\epsilon^{-2})$ for finding an $\epsilon$-stationary point in terms of metric (\textbf{M2})
 for deterministic WCSC problems --both algorithms 
have \textit{three nested loops} and 
%impose restrictions on 
require the closed convex functions $g$ and $h$
%, i.e., $g,h$ can only 
be the indicator functions of some closed convex sets. 
%Furthermore, both \cite{lin2020near} and \cite{ostrovskii2021efficient} establish an iteration complexity of $\tilde\cO(\epsilon^{-2.5})$ for deterministic WCMC problems, which is later improved to $\cO(\epsilon^{-2.5})$ in~\cite{kong2021accelerated}. 
{Given the substantial volume of recent research %in this field 
for addressing \textit{deterministic} WCSC minimax problems, we have restricted the scope of the discussion in \cref{table:related_work} to primarily include \textit{single-loop} methods \rv{that are} closely related to the proposed \agdap{} method in this paper.}}\footnote{We recently become aware of a new work~\cite{yang2024two} citing our online arXiv preprint~\cite{zhang2024agda+}. The authors of \cite{yang2024two} propose a GDA-type method with backtracking employing \textit{monotonically decreasing} stepsizes. \cite[Remark 7]{yang2024two} claims that the proposed method achieves $\cO(L\kappa^3/\epsilon^2)$ complexity; however, this bound is incorrect. Indeed, the proposed method 
%takes $l_{11}^1,l_{12}^1,l_{22}^1$ and $\mu^1$ as the initial estimates of the block Lipschitz constants $L_{11},L_{12},L_{22}$ and the concavity modulus $\mu$. In the following we consider two cases: $\mu$ unknown, and $\mu$ known. For the case $\mu$ is unknown, we treat $l_{11}^1,l_{12}^1,l_{22}^1$ and $\mu^1$ as $\cO(1)$ constants; on the other hand, when $\mu$ is known, we treat $l_{11}^1,l_{12}^1,l_{22}^1$ as $\cO(1)$ constants, and set $\mu_1=\mu$. The algorithm searches for $L$ and $\mu$ simultaneously, and it 
is guaranteed to generate an $\epsilon$-stationary point in terms of metric \textbf{(M2)} within $\frac{F_0-\underline{F}}{d_1\epsilon^2}$ iterations, where $F_0-\underline{F}=\Omega(\cL(\bx_0,\by_0)-\cL^*+L\kappa^3\cD_y^2)$ and $\cL^*=\inf_{\bx\in X, \by\in Y}\cL(\bx,\by)$. According to 
%the definition of $d_1$ in 
\cite[Theorem 6]{yang2024two}, it holds that $d_1\leq \frac{\beta_{\min}}{2\beta_{\max}^2}$, where $\beta_{\max}=\Omega(L\kappa^3)$ and $\beta_{\min}=\cO(1)$ when $\mu$ is unknown, and $\beta_{\min}=\cO(1/\mu^3)$ when $\mu$ is known, i.e., $\mu^1=\mu$. Therefore, we can conclude that for the case $\mu$ is unknown,
    $(F_0-\underline{F})/(d_1\epsilon^2)=\Omega(L^2\kappa^6(\cL(x_0,y_0)-\cL^*+L\kappa^3\cD_y^2)\epsilon^{-2})=\Omega(L^3\kappa^9\epsilon^{-2})$;
and when $\mu$ is known, the bound %improves to
becomes
    $(F_0-\underline{F})/(d_1\epsilon^2)=\Omega(L^5\kappa^3(\cL(x_0,y_0)-\cL^*+L\kappa^3\cD_y^2)\epsilon^{-2})=\Omega(L^6\kappa^6\epsilon^{-2})$.
%\cite[Theorem 6]{yang2024two} implies $\cO(L^3\kappa^9/\epsilon^2)$ and $\cO(L^6\kappa^6/\epsilon^2)$ complexity bounds when $\mu$ is unknown and known, respectively.
}
%Indeed, although there are many \sa{algorithms for \eqref{eq:main-problem} and its special cases}, 
To the best of 
our knowledge, there are \textit{very few} first-order methods for \eqref{eq:main-problem} that are \textit{agnostic} to the Lipschitz constant of $\grad f$ \rv{and/or concavity modulus $\mu$}, %and weak-convexity modulus of $f(\cdot,y)$, 
and use adaptive step sizes for solving minimax problems in the nonconvex setting, e.g.,~\cite{xu2024stochastic,li2022tiada,yang2022nest,lee2021fast,pethick2023escaping}.
%\qstodo{do we need to change the table? since some of them are agnostic to both $L$ and $\mu$ while \texttt{SGDA-B} is agnostic to $L$ only.}

{\renewcommand{\arraystretch}{1.2}
    \small %
    \begin{table*}[h]
    \captionsetup{font=scriptsize}
        \centering
        \resizebox{0.99\textwidth}{!}{
        \begin{tabular}{|l"c|c|c|P{2.5cm}|P{2.5cm}|P{2.75cm}|P{2cm}|c|c|}
            \thickhline
            \textbf{Work} & \textbf{Agnostic} & $g$ & $h$ & \textbf{Step-size search (backtracking)} & \textbf{Nonmonotone step-size search} & \textbf{\# of $\grad f$ calls per backtracking iter.} &\textbf{\# of calls to max-solver} &\textbf{Metric}  & \textbf{Complexity} \\
            \thickhline
            \gda{} \cite{lin2020gradient} & \XSolidBrush & \XSolidBrush & $\mathds{1}_{Y}$ &  \XSolidBrush & N/A & N/A & -- & M1 & $\cO\left((L\kappa^2 B_0+\kappa L^2\cD_y^2)  \epsilon^{-2}\right)$  \\
            \hline
            \agda{}  \cite{boct2020alternating} & \XSolidBrush & \CheckmarkBold & \CheckmarkBold  & \XSolidBrush & N/A & N/A & -- & M1 & $\cO\left((L\kappa^2 B_0+\kappa L^2\cD_y^2)  \epsilon^{-2}\right)$ \\
            \hline
            \smagda{} \cite{yang2022faster} & \XSolidBrush & \XSolidBrush & \XSolidBrush & \XSolidBrush & N/A & N/A & -- & M2 & $\cO(\kappa L (B_0+G_0)\epsilon^{-2})$ \\
            \hline
            \texttt{AltGDAm}~\cite{chen2022accelerated} &\XSolidBrush & \CheckmarkBold & \CheckmarkBold  & \XSolidBrush & N/A & N/A & -- & M1 & $\cO(\na{L}\kappa^{\frac{11}{6}}\na{B_0}\epsilon^{-2})$ \\
            \hline
            $\texttt{\na{HiBSA}}^{\diamond}$  \cite{lu2020hybrid} & \XSolidBrush & \CheckmarkBold & \CheckmarkBold & \XSolidBrush & N/A & N/A & -- & M2 & \na{$\cO(L\kappa^5(\cL(\bx_0,\by_0)-\cL^*+L\kappa^3\cD_y^2) \epsilon^{-2})$} \\
            \hline
            $\texttt{AGP}^{\diamond}$  \cite{xu2023unified} & \XSolidBrush & $\mathds{1}_{\sa{X}}$ & $\mathds{1}_{Y}$ & \XSolidBrush & N/A & N/A & -- & M2 & $\na{\cO(L\kappa^5(\cL(\bx_0,\by_0)-\cL^*+L\kappa\cD_y^2) \epsilon^{-2})}$ \\
            \thickhline
            \thickhline
            $\neada{}^{\diamond}$   \cite{yang2022nest}  & $L,\mu$ & \XSolidBrush  & $\mathds{1}_{Y}$ & \XSolidBrush & N/A & N/A & $\cO(\kappa^4L^4/\epsilon^2)$ & M2 & $\cO(\rv{\kappa^{a}}\log(\kappa)(L^4\kappa^4+B_0^2) \epsilon^{-2})$ \\
            \hline
            $\tiada{}^{\diamond}$  \cite{li2022tiada}  & $L,\mu$ & \XSolidBrush  & $\mathds{1}_{Y}$ & \XSolidBrush & N/A & N/A & -- & M2 & $\cO(\kappa^{10} \epsilon^{-2})$ \\
            \hline
            {\texttt{SGDA-B} \cite{xu2024stochastic}}  & $L$ & \CheckmarkBold & \CheckmarkBold & \CheckmarkBold & \XSolidBrush & $\cO(L\kappa^2\epsilon^{-2})$ & -- & M2 & $\cO({L\kappa^2\log(\kappa)(B_0+L\cD_y^2)\epsilon^{-2}})$ \\
            \hline
            {\agdap{} (\msf{True})}  & $L$ & \CheckmarkBold & \CheckmarkBold & \CheckmarkBold & \CheckmarkBold & $\cO(1)$ & $\cO(\log(\kappa))$ & M2 & $\cO({L\kappa^4(B_0+L\log(\kappa)\cD_y^2)\epsilon^{-2}})$ \\
            \hline
            {\agdap{} (\msf{False})}  & $L$ & \CheckmarkBold & \CheckmarkBold & \CheckmarkBold & \CheckmarkBold & $\cO(1)$ & -- & M2 & $\cO({L\kappa^4(B_0+ L\kappa\cD_y^2)\epsilon^{-2}})$ \\
            \hline
            {\agdap{} (\msf{True})}  & $L,\mu$ & \CheckmarkBold & \CheckmarkBold & \CheckmarkBold & \CheckmarkBold & $\cO(1)$ & $\cO(\log(\cR))$ & M2 & $\cO({\bar L\bar{\kappa}^4(B_0+\rv{\bar L}\log(\cR)\cD_y^2)\epsilon^{-2}})$ \\
            \hline
            {\agdap{} (\msf{False})}  & $L,\mu$ & \CheckmarkBold & \CheckmarkBold & \CheckmarkBold & \CheckmarkBold & $\cO(1)$ & -- & M2 & $\cO({\bar L\bar{\kappa}^4(B_0+ \rv{\tilde\mu\cR^3}\cD_y^2)\epsilon^{-2}})$ \\
            \thickhline
        \end{tabular}
        }
        \caption{\na{Comparison of related work. The column \rv{“\textbf{Agnostic}” indicates the parameters, $L$ and/or $\mu$, the method does \textit{not} need to know.}} 
        %and/or the weak convexity modulus. 
        In the columns related to \rv{$g$ and $h$}, we indicate whether the method can handle closed convex functions in the minimax formulation of \eqref{eq:main-problem} \rv{through prox operations}. In the ``\textbf{Step-size search}" column, we specify whether the method uses a step-size seach (backtracking) mechanism; while \tiada{} and \neada{} are agnostic to $L$ \rv{and $\mu$}, they adopt (monotonically decreasing) AdaGrad based adaptive step sizes, i.e., their step-size selection is not based on a step-size search scheme. In the ``\textbf{Nonmonotone step-size search}" column, we state whether the method can exploit local Lipschtz structure through employing a \textit{non-monotone} step size sequence. In the ``\textbf{\# of $\grad f$ calls per backtracking iter.}" column, we listed the number of $\grad f$ calls required to decide whether to accept or reject a candidate step-size pair $(\tau,\sigma)$ for each backtracking iteration. If the method does not call a maximization solver to inexactly solve $\max_{\by}\cL(\bar\bx,\by)$ for a given $\bar\bx$, we use ``--" to indicate it \rv{within the \textbf{\# of calls to max-solver} column}; otherwise, we state the number of calls required for the maximization solver. ``\textbf{Metric}" column indicates the metric used for defining $\epsilon$-stationarity, i.e., {(\textbf{M1}): $\bx_\epsilon$ such that $\norm{\cG^\tau(\bx_\epsilon)}\leq \epsilon$ for some $\tau>0$, (\textbf{M2}): $(\bx_\epsilon,\by_\epsilon)$ such that \rv{$\norm{G^{\tau,\sigma}(\bx_\epsilon,\by_\epsilon)}\leq \epsilon$ for some $\tau,\sigma>0$.} The definitions of $\cG^\tau$ and $G^{\tau,\sigma}$ are provided %at the beginning of  \textbf{Preliminaries: Assumptions and Definitions} of 
        in~\cref{sec:preli}.} \na{In the ``\textbf{Complexity}" column we give the gradient complexity of the method to compute an $\epsilon$-stationary point, where $B_0=F(x_0)-F^*$ denotes the primal suboptimality of the initial point with $F(\cdot)=\max_{\by\in\cY}\cL(\cdot,\by)$, $\cD_y$ is the diameter of the dual domain, $\cL^*=\inf\{\cL(\bx,\by):\ \bx\in\cX,~\by\in\cY\}$, and $G_0=\sup_{\by\in\cY} \hat\cL(\bx_0,\by;\bz_0)-\inf_{\bx\in\cX}\hat\cL(\bx,\by_0;\bz_0)$ for $\hat\cL(\bx,\by;\bz_0)=\cL(\bx,\by)+L\norm{\bx-\bz_0}^2$ denoting the primal-dual gap of the initial point with respect to a regularized coupling function.} \rv{For \agdap{}, for the case $\mu$ is not known, $\bar\kappa\triangleq \bar L/\bar\mu$ where $\bar L  \triangleq \rv{\cR\;\tilde l}$ and $\bar \mu \triangleq \rv{\tilde\mu/\cR}$ with $\rv{\cR \triangleq \max\{\tilde\mu/\mu,~L/\tilde l,~1\}}$ --here, $\tilde\mu\geq \mu$ and $\tilde l\leq L$ denote the initial estimates of the concavity modulus $\mu$ and the local Lipschitz constant at $\bz_0$, respectively; the discussion in Remark~\ref{remark:mu-discussion-with-max-solver} implies that whenever the initial parameters satisfy $\tilde\mu=\Theta(\sqrt{L\mu})$ and $\tilde l=\Theta(\sqrt{L\mu})$, we have $\cR=\Theta(\sqrt{\kappa})$, and the term $\tilde\mu \cR^3$ in the bound for \msf{false} is $\cO(L\kappa)$.}\\ \na{\textbf{Table Notes.} $\diamond$ The %complexity 
        bounds 
        %of these methods 
        with their explicit dependence on $\kappa$ and $L$ are derived in \cref{sec:complexity-others}.} \rv{For \neada{}, when $\mu$ is unknown $a=1$, and when $\mu$ is known, $a=0.5$.}}
    \label{table:related_work}
    \end{table*}
}%

For a special case of \eqref{eq:main-problem} with %$h(\cdot)=0$
$h(\cdot)=\mathbf{1}_{Y}(\cdot)$, i.e., the indicator function of a closed convex set $Y$, \sa{and \textit{bilinear} coupling of the form} \na{$f(\bx,\by)=s(\bx)-r(\by)+\fprod{A\bx,\by}$ such that $s$ and $r$ are  smooth functions with $r$ being strongly convex,} a triple-loop method is proposed in~\cite{dvurechensky2017gradient}. 
\na{In the outer loop, primal iterate $\bx_t$ is updated using the step sizes satisfying a backtracking condition based on the primal function, \na{i.e., $F(\cdot)=g(\cdot)+\max_{\by\in Y}f(\cdot,\by)$}; one searches for an admissible primal step size using the middle loop, which repeats until the backtracking condition holds, and for each middle-loop iteration, the inner maximization problem $\max_{\by\in\cY} f(\bx_t,\by)$ is solved inexactly using the inner loop. The gradient complexity of the proposed method is $\tilde\cO(L\kappa^{b}\epsilon^{-2})$ for computing an $\epsilon$-stationary point in terms of metric \textbf{(M1)}, \rv{where $b=2$ if $\mu$ is unknown, and $b=1.5$ if $\mu$ is known}.}
% the inner maximization problem \sa{$\max_{y\in Y} f(x,y)$} is solved inexactly %\sa{over a closed convex set $Y\subset \cY$} 
% and the backtracking steps are implemented for the primal function; \sa{the gradient complexity of the proposed method is \na{$\tilde\cO(L\kappa^{1.5}\epsilon^{-2})$} for computing an $\epsilon$-stationary point in terms of metric (\textbf{M1}).} 

For the more general case of WCSC minimax problems when the coupling function $f$ is \textit{not} bilinear, \na{\sgdab{}~\cite{xu2024stochastic}, \neada~\cite{yang2022nest} and \tiada~\cite{li2022tiada} algorithms are recently proposed. Except for \sgdab{}, which is discussed in detail in~\cref{sec:goals}, %both of which
the other two, i.e., \neada{} and \tiada{}, are based on AdaGrad stepsizes~\cite{duchi2011adaptive} and {these methods do not use backtracking/stepsize search}.} \neada{}, similar to \cite{dvurechensky2017gradient}, is a %double
multi-loop method for solving \eqref{eq:main-problem} with $g(\cdot)=0$ and $h(\cdot)=\mathbf{1}_{Y}(\cdot)$; %, i.e., the indicator function of a closed convex set $Y$; 
in the inner loop, a strongly concave problem is solved inexactly using AdaGrad method, of which solution is then used to compute an inexact gradient for the primal objective function, followed by an AdaGrad update for the primal variable $\bx$. Moreover, %it
\neada{} can compute $(\bx_\epsilon,\by_\epsilon)$ such that \sa{$\norm{\grad_x f(\bx_\epsilon,\by_\epsilon)}\leq \epsilon$ and $\norm{\by_\epsilon-\by^*(\bx_\epsilon)}\leq \epsilon$ within $\cO(L^4\rv{\kappa^5}\epsilon^{-2}\log(1/\epsilon))$ gradient complexity} \rv{when $L$ and $\mu$ are both unknown}, where $\by^*(\cdot)=\argmax_{\by\in Y} f(\cdot,\by)$. On the other hand, \tiada{} is a single-loop scheme that can be considered as an extension of AdaGrad~\cite{duchi2011adaptive} to solve minimax problems in a similar spirit to two-time scale \gda{} algorithms. \tiada{} can compute $(\bx_\epsilon,\by_\epsilon)$ such that \sa{$\norm{\grad f(\bx_\epsilon,\by_\epsilon)}\leq \epsilon$ within $\cO(\sa{\kappa^{10}}\epsilon^{-2})$ 
gradient calls.
%\nsa{In the appendix, show the derivation of complexity statements for \tiada{} and \neada{}.} 
The main reason for $\cO(\kappa^{10})$ \na{constant} arising in the complexity bound is because of how \tiada{} adapts to the required time-scale separation. Indeed, as shown in \cite{li2022tiada}, the ratio of their stepsizes $\sigma_t/\tau_t$ is lower bounded by an increasing sequence, i.e.,  $\frac{\sigma_t}{\tau_t}\geq \frac{\sigma_0}{\tau_0}(v_{t+1}^y)^{\alpha-\beta}$, since $\alpha>\beta>0$, and $\{v_t^y\}$ is a monotonically increasing sequence --see the details on \tiada{} stepsize choice given in \cref{sec:goals}.}
% \begin{equation}
%     \frac{\tau_t}{\sigma_t} = \frac{\tau_0/ \max\{v_{t+1}^x, v_{t+1}^y\}^{\alpha}}{ \sigma_0/ (v_{t+1}^y)^{\beta}}\leq 
%     %\frac{\tau_0/ (v_{t+1}^y)^{\alpha}}{ \sigma_0/ (v_{t+1}^y)^{\beta}} = 
%     \frac{\tau_0}{\sigma_0 (v_{t+1}^y)^{\alpha-\beta}},
% \end{equation}}%
% as $v_t^y$ is the sum of previous gradient norms and is increasing. 
% As we discuss before, their step sizes will possibly diminish to 0 very quickly.
% \nsa{Discuss Sec. 2.1 in \tiada{} paper. Talk about the step size rule and its connection to time-scale separation.}
The complexity for \neada{} and \tiada{} for the deterministic WCSC problems are provided in \cref{table:related_work}. \rv{In~\cite{yang2022nest} and~\cite{li2022tiada},} the analyses of these algorithms are given for \textit{smooth} \sa{WCSC} problems with $g(\cdot)=0$ and $h(\cdot)=\mathbf{1}_{Y}(\cdot)$; and it is not clear whether their analyses can be extended to the non-smooth setting we consider in~\eqref{eq:main-problem}. \sa{Furthermore, for smooth %WCWC 
problems, extragradient~(EG) methods \sa{with backtracking are proposed} in~\cite{lee2021fast,pethick2023escaping}; but, they make additional assumptions \sa{(\textit{stronger} than smoothness of $f$, e.g., negative co-monotonicity, or requiring the (weak) Minty Variational Inequality (MVI) corresponding to $\grad f$ to have a solution)} that in general \sa{do} not hold for WCSC problems \sa{we consider} in the form of \eqref{eq:main-problem}.}
\looseness=-10

\subsection{Adaptive algorithms on a toy example}
%\xtodo{this subsection is repeated and not necessary?}
\sa{In \cref{toy-example}, \rv{assuming $\mu$ is known,} we test two adaptive algorithms \rv{that are agnostic to $L$,} \agdap{} (ours) and \tiada{}, on a toy quadratic WCSC problem, and compare them against \agda{}~\cite{boct2020alternating} (which requires knowing $L$) as a benchmark. 
% \qsr{And in this toy example, we assume that $\mu$ is known for fair comparisons.}
We plotted the dual and primal step sizes and the gradient norm (as a measure of stationarity) against the number of gradient calls. %we can see that even though the adaptive algorithm \cite{li2022tiada} for nonconvex settings is able to converge eventually, they need to ensure that the step sizes must be always diminishing 
As shown in \cref{fig:Q-mu1-dual-stepsize}, $\sigma_t$ stays significantly larger than $1/L$ for \agdap{} while it diminishes quickly below $1/L$ for \tiada{}; moreover, in Figure~\ref{fig:Q-mu1-primal-stepsize}, we also observe that $\tau_t$ %slowly decreases as a step function 
\na{quickly increases and stays over $\frac{1}{L\kappa^2}$ for \agdap, while it decreases very quickly for \tiada{} similar to $\sigma_t$ behavior, and stabilizes at a value lower than \agdap{}.} %and their convergence speed is highly dependent on the previous gradients and initialization, which are hard to modify. But for our methods, our step sizes 
For \agdap, unlike \tiada{}, the step sizes $(\tau_t,\sigma_t)$ are not necessarily decreasing at every step, \na{i.e., while step size sequences for \tiada{} are monotonically decreasing by construction, the general trend we observe for \agdap{} is that they increase and become stable at a large value.} Indeed, in designing \agdap{}{} we try to exploit the local smoothness of $f$ through estimating the local Lipshitz constant at $(\bx_t,\by_t)$, which can be much smaller than the global Lipschitz constant $L$ --hence, providing more information about the landscape and allowing for relatively large step sizes. That is why we observe rapid convergence for \agdap{}{} compared to \agda{} and \tiada{}.} 
%Here, we skip the comparisons between GDA-B and \neada{}, since the inner loop of \neada{} is unknown for us, i.e., we cannot access to the dual stepsizes. For the primal stepsizes, it will show as the same as \tiada{}. Also, \neada{} is a double-loop algorithm.
\begin{figure}[h]
   \centering
   \begin{subfigure}[b]{0.33\textwidth}
       \centering       
       \includegraphics[width=\textwidth]{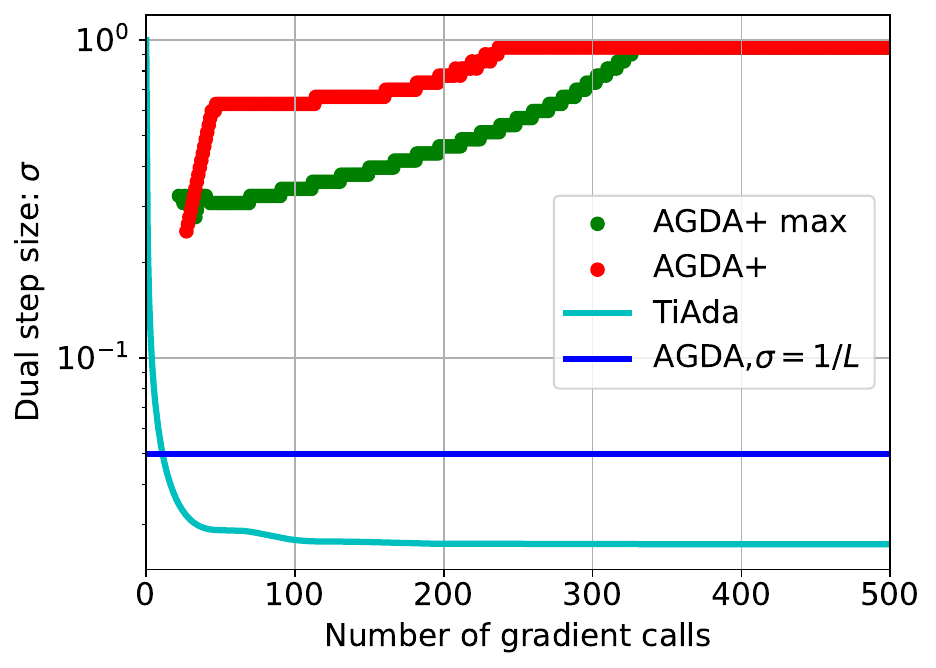}
       \caption{Dual step size: $\sigma$}
       \label{fig:Q-mu1-dual-stepsize}  
   \end{subfigure}
   \begin{subfigure}[b]{0.33\textwidth}
       \centering         
       \includegraphics[width=\textwidth]{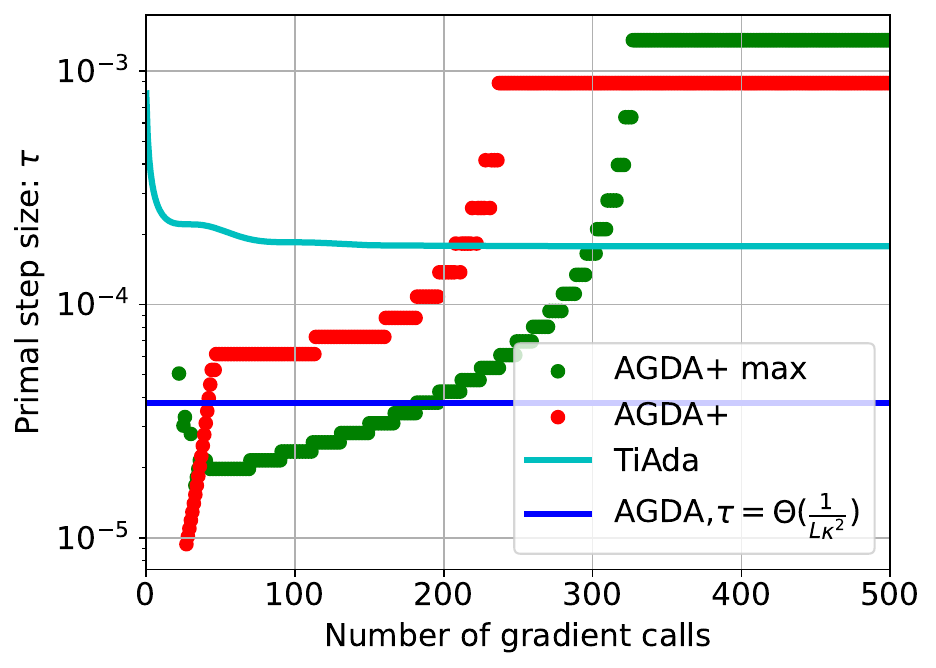}
        \caption{Primal step size: $\tau$}
        \label{fig:Q-mu1-primal-stepsize} 
   \end{subfigure}
   %    \begin{subfigure}[b]{0.24\textwidth}
   %     \centering         
   %     \includegraphics[width=\textwidth]{figure/figure_intro/Q_stdx_0_stdy_0_toy_muy_1_kappa_20_b_1_ratio.pdf}
   %      \caption{Ratio: $\sigma/\tau$}
   %      \label{fig:Q-mu1-ratio-stepsize} 
   % \end{subfigure}
   \begin{subfigure}[b]{0.33\textwidth}
       \centering       
       \includegraphics[width=\textwidth]{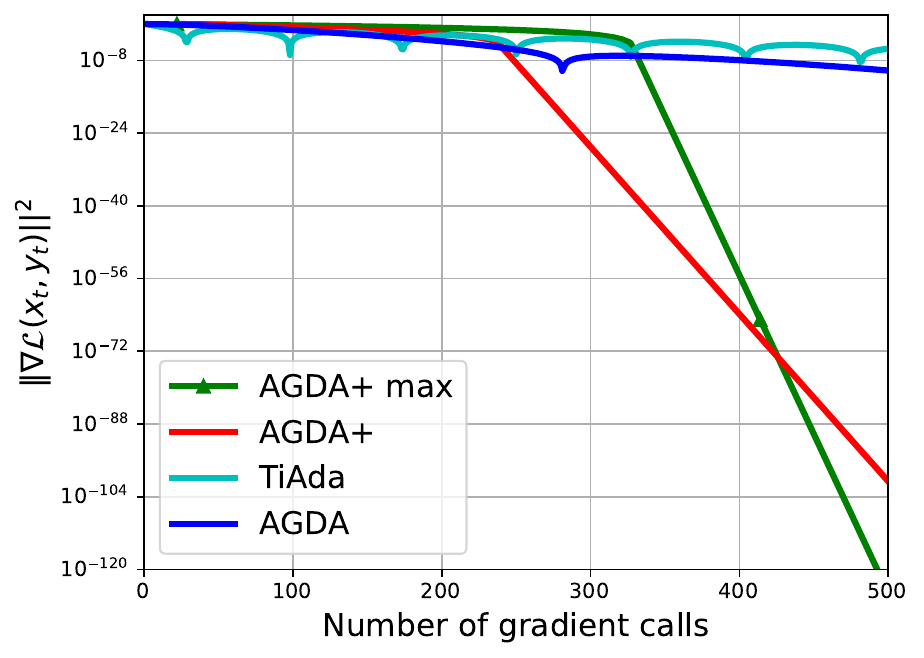}
       \caption{Convergence}
       \label{fig:Q-mu1-convergence}  
    \end{subfigure}
\caption{\na{Primal-dual step size sequences and convergence behavior comparison of \agdap{} against \texttt{TiAda}~\cite{li2022tiada} and \texttt{AGDA}~\cite{boct2020alternating} on a toy minimax problem which is a slightly modified version of \cite[Eq.~(2)]{li2022tiada}: $\min_{x\in\reals}\max_{y\in\reals}\cL(x,y) = -\frac{L}{2}x^2 + Lxy - \frac{1}{2}y^2 $ with $L=20$. All three methods are initialized from the point $(x,y) = (1,0.01)$. \agdap{} has two different implementations: one is labeled as \texttt{AGDA+max} corresponding to \msf{True} in~\cref{alg:agda}, and the other is simply labeled as \agdap{} corresponding to \msf{False}.} \na{The benchmark stepsizes, i.e., $\sigma=\Theta(1/L)$ and $\tau=\Theta(\frac{1}{L\kappa^2})$ are the largest %deterministic 
step sizes with a guaranteed convergence when $L$ \rv{and $\mu$ are} known, see~\cite{boct2020alternating,lin2020gradient}.}}
\label{toy-example}
\end{figure}
\vspace*{-5mm}

\section{Goal of this paper and its contributions}  
\label{sec:goals}
Many efficient first-order primal-dual algorithms have been proposed \sa{for solving WCSC minimax problems of the form in~\eqref{eq:main-problem},} e.g.,~\cite{lu2020hybrid,boct2020alternating,chen2022accelerated,xu2023unified}; \sa{however, except for~\cite{li2022tiada,yang2022nest,xu2024stochastic}, all of these algorithms require the value of \textit{global} Lipschitz constant $L$ corresponding to $\grad f$ \rv{and the concavity modulus $\mu$ corresponding to $f(\bx,\cdot)$ }%to be able to 
so that one can properly set the step sizes needed for the convergence theory to hold.} %\nsa{Add the discussion from the proposal about $L$ is generally unknown, or in case it is known using the global constant leads to conservative step sizes.} 
\sa{Although requiring that $L$ {and $\mu$} are known is a relatively standard assumption when deriving rate statements for first-order schemes, it should be emphasized that in a variety of practical ML applications, %these constants
$L$ {and/or $\mu$} may not be readily available and computing {them} might be challenging; moreover, even if \rv{$L$ is} known, the use of global constants can often lead to conservative step sizes, thereby causing slow convergence in practice. An effective alternative to address this challenge involves adopting line-search or backtracking approaches. These approaches leverage \textit{local} smoothness for determining the step sizes, which leads to significantly larger \na{steps}.} %sizes.}
%\xtodo{I rewrite it by my color.}
%\xz{The requirement regarding $L$ is frequently encountered in the development of rate statements for first-order methodologies. It should be recognized, though, that in a variety of practical applications, obtaining these constants might be challenging, or the use of global constants can often lead to unduly conservative step sizes, thereby slowing down the convergence process. An effective alternative to address this challenge involves the integration of line-search or backtracking approaches. These approaches leverage local Lipschitz constants, enabling more precise and efficient determination of step sizes.}
%or some strong assumptions, like Lipschitz constants and the uniform bounds for the stochastic gradients.

\sa{To the best of our knowledge, the only adaptive algorithms designed specifically for the WCSC minimax problems with non-bilinear coupling and agnostic to 
%generally unknown Lipschitz constant 
the smoothness parameter $L$} are \tiada{}~\cite{li2022tiada}, \neada{}~\cite{yang2022nest} and \sgdab~\cite{xu2024stochastic}; \sa{that said, their %experimental results 
practical performance is not on par with their desirable theoretical properties (see Figures~\ref{fig:Q},~\ref{fig:sin-Q} and~\ref{fig:DRO}). \na{The practical problems with these methods arise especially when the local Lipschitz constant fluctuates over the problem domain since none of these methods can adopt \textit{non-monotone} step size sequences exploiting the local topology. For \tiada{}~\cite{li2022tiada} and \neada{}~\cite{yang2022nest}, we also observe practical issues in convergence to a stationary point also when the condition number of the problem $\kappa\triangleq L/\mu\gg 1$ and/or the global Lipschitz constants $L\gg 1$;} indeed, this may be due to large $\cO(1)$ constants with high powers of $\kappa$ and $L$ hidden in their complexity bounds.} %Even though they give the theoretical convergence analysis and use the decaying stepsizes, they were unable to explore the property of nonconvex minimax problems thoroughly. Since their methods are highly dependent on the initialization and during each step, they cannot find the appropriate time-scale between the primal and dual step sizes. The reason is that their step sizes are highly counted on the randomly generated gradient estimates, which leads to extremely slow convergence rate. 
\sa{Both of these algorithms adopt AdaGrad~\cite{duchi2011adaptive} type step size sequences, and while this choice makes the algorithm agnostic to $L$ and $\mu$,
%\qstodo{I think only Tiada is agnostic to both $L$ and $\mu$} 
it %also potentially
may also cause practical performance issues: since initial iterates are potentially far from stationarity (with gradient norms not close to $0$), %in the initial phase of the algorithm the step sizes diminish to $0$ very quickly, 
the step sizes tend to shrink rapidly in the early stages of the algorithm,
as they are inversely proportional to the %partial
cumulative sum of gradient norms %over
from all %past
previous iterations. These quickly decaying step sizes cause slow convergence in practice (see \cref{toy-example}).}
\na{Recently, Pethick et al.~\cite{pethick2023escaping} propose a backtracking method for solving \eqref{eq:main-problem} assuming that $f$ is $L$-smooth (without assuming concavity in $y$) and that the  weak Minty variational inequality~(MVI) corresponding to~\eqref{eq:main-problem} has a solution; and they show that the limit points of the iterate sequence belongs to the zero-set of the operator defining the weak MVI without providing any complexity guarantee for the adaptive algorithm agnostic to $L$. It is crucial to emphasize that for nonconvex minimax problems in the form of~\eqref{eq:main-problem} since the operator defining the weak MVI is non-monotone, assuming $\nabla f$ is $L$-Lipschitz does not imply that the corresponding weak MVI has a solution; indeed, the weak MVI does not necessarily have a solution even if we further assume $f$ is (strongly) concave in $\by$. Therefore, the adaptive method in~\cite{pethick2023escaping} does not apply to the WCSC setting we consider in this paper since the corresponding weak MVIs may not have a solution.}

\na{Acknowledging the reasons stated in~\cite{yang2022nest,li2022tiada} explaining %a need of 
the need for novel GDA-based algorithms that are agnostic to $L$, and that are tailored for nonconvex WCSC minimax problems, in a recent work~\cite{xu2024stochastic} the authors have proposed the \sgdab{} %that is also agnostic to $L$ and it 
which incorporates dual stepsize backtracking with a proximal GDA method to compute near stationary point of \eqref{eq:main-problem}. \rv{Assuming $\mu$ is known,} it has been shown in~\cite{xu2024stochastic} that \sgdab{} can compute an $\epsilon$-stationary point for \eqref{eq:main-problem} with an upper complexity bound of $\cO(L\kappa^2\log(\kappa)\epsilon^{-2})$. To briefly summarize, \sgdab{} is a double-loop method: in the outer iteration $k\in\N$, given an estimate $L_k\geq \mu$ of $L$, one sets $\sigma_k=1/L_k$ and $\tau_k=\rho\mu^2\sigma_k^3$ for $\rho=(\sqrt{13}-1)/24\in(0,1)$, and executes \eqref{eq:gda} for $T=\cO(\frac{1}{\tau_k\epsilon^2})=\cO(L\kappa^2/\epsilon^2)$ prox-GDA iterations (inner iterations) such that $(\tau_t,\sigma_t)=(\tau_k,\sigma_k)$ is kept constant for $t=1,\ldots,T$, and at the end of $T$ inner iterations, a backtracking condition is checked; if the condition holds, \sgdab{} terminates, otherwise a new outer iteration starts with a larger estimate of the global Lipschitz constant $L_{k+1}=L_k/\gamma$ for some $\gamma\in(0,1)$. It is established in \cite{xu2024stochastic} that the backtracking condition holds in $\cO(\log(\kappa))$ outer iterations, which leads to $\cO(L\kappa^2\log(\kappa)/\epsilon^{-2})$ gradient complexity bound. These results for a GDA-based algorithm agnostic to $L$ are sound and on par with \cite{lin2020gradient,boct2020alternating} that require knowing $L$ to set constant primal-dual stepsizes $(\tau,\sigma)$ with convergence guarantees; that said, there are two significant shortcomings of \sgdab{}: \textit{(i)} the sequences $\{\tau_k\}$ and $\{\sigma_k\}$ are \textit{monotonically decreasing}; hence, \sgdab{} cannot exploit the local smoothness structure as \agdap{} can, and \textit{(ii)} between two backtracking iterations \sgdab{} requires $\cO(\frac{1}{\tau_k\epsilon^2})$ gradient calls, which can be as large as $\cO(L\kappa^2/\epsilon^2)$ in the worst case; on the other hand, \agdap{} only requires $\cO(1)$ gradient calls between two backtracking iterations \rv{(indeed, $2$ gradients calls on average, see~\cref{thm:avg--backtrack-time})} --hence, it is more robust to wrong estimates of $L$ as it can quickly correct it.}

\sa{From another perspective, it has been established that having an appropriate \textit{time-scale separation} between the primal ($\tau$) and dual ($\sigma$) step sizes is shown to be \textit{sufficient} for ensuring the convergence of single-loop gradient based methods on WCSC problems, e.g., for the constant step size regime $(\tau_t,\sigma_t)=(\tau,\sigma)$, \na{both} \cite{lin2020gradient,boct2020alternating} require $\sigma/\tau=\Omega(\kappa^2)$ and $\sigma=\cO(1/L)$ for the (proximal) %gradient ascent descent~(GDA) 
\na{GDA} method. More recently, \cite{li2022convergence} establish that $\sigma/\tau=\Omega(\kappa)$ is necessary and sufficient for local convergence of %\gda{}
GDA to a Stackelberg equilibrium on smooth weakly convex-weakly concave (WCWC) minimax problems, where $\kappa$ denotes the local condition number for $\by$.} %That said
\rv{On the other hand}, \sa{as also mentioned in~\cite{li2022tiada}, \na{if $(\tau_t,\sigma_t)$ within a single-loop method for solving WCSC minimax problems is naively chosen using} adaptive step sizes designed for minimization algorithms, e.g., AdaGrad~\cite{duchi2011adaptive}, Adam~\cite{KingBa15}, \na{the resulting primal-dual stepsize sequence does not necessarily} adapt to the sensitive \textit{time-scale} ratio for minimax problems; hence, directly using these methods does not guarantee convergence --see the simple quadratic problem given as a counter example in~\cite{yang2022nest}.} 

\sa{Existing work has partially addressed the aforementioned issue for the smooth case, i.e., $g=h=0$, in two different ways: \textit{(i)} using a double-loop algorithm \na{\neada{}} in which given $\bx_t$, a strongly concave maximization problem $\max_{\by\in\cY} f(\bx_t,\by)$ is inexactly solved using AdaGrad to get $\by_{t+1}$, and then the primal variable $\bx_{t+1}=\bx_t-\tau_t\grad_{\bx} f(\bx_t,\by_{t+1})$ is computed using an AdaGrad step size $\tau_t$; \textit{(ii)} alternatively, as a single-loop method, \tiada{} uses Jacobi-type GDA updates, i.e., in \eqref{eq:y-update} instead of {$\grad_{\by} f(\bx_{t+1},\by_{t})$}, the partial gradient $\grad_{\by} f(\bx_t,\by_t)$ is used with 
\begin{equation}
\label{eq:tiada_step}
\sigma_t=\sigma_0/(v_{t+1}^y)^\beta,\qquad \tau_t=\tau_0/(\max\{v_{t+1}^x, v_{t+1}^y\})^\alpha,
\end{equation}
where $\alpha, \beta\in(0,1)$ such that $\alpha>\beta$ and $v^x_{t+1}=v^x_t+\norm{\grad_{\bx}f(\bx_t,\by_t)}^2$ and $v^y_{t+1}=v^y_t+\norm{\grad_{\by}f(\bx_t,\by_t)}^2$. This choice implies that $\{\sigma_t/\tau_t\}$ is increasing %\nearrow \infty$ 
and $\{\sigma_t\}$ is decreasing; %\searrow 0$; 
hence, the step size ratio might eventually increase above the desirable threshold for convergence, i.e., $\Theta(\kappa)$. That said, with this update strategy, it may well be the case that $\limsup_t \sigma_t/\tau_t\gg \kappa^2$, and $\limsup_t \sigma_t$ can be close to $0$, much smaller than $1/L$; %one does not need both $\tau_t$ and $\sigma_t$ diminish to $0$; 
therefore, it is possible that this strategy may lead to very conservative step sizes.} In contrast to \tiada{}'s %diminishing
decreasing step size choice, we are interested in adaptively choosing step sizes such that $\sigma=\Omega(1/L)$ and 
%\xtodo{Our $\tau$ is in fact $1/(L\kappa^4)$}
$\tau_t=\Omega(1/(L\kappa^4))$ while ensuring $\sigma_t/\tau_t=\Theta(\kappa^4)$, which would simultaneously allow us to achieve better gradient complexity in terms of the $\kappa$ dependence and observe much faster convergence in practice.
%\qstodo{In \tiada{}, $\sigma_t / \tau_t$ is bounded by a nondecreasing sequence, it doesn't mean that $\sigma_t / \tau_t$ is nondecreasing, right? This part violates the toy example, figure 1.}
\sa{  
%compared to \tiada{}'s 
%as $\sigma_t/\tau_t\nearrow\infty$ and $\sigma_t\searrow 0$ for \tiada{}
%decreasing stepsize choice. 
%It is worth noting that 
Compared to $\sigma=\Theta(1/L)$ and $\sigma/\tau=\Theta(\kappa^2)$ conditions required for convergence in \cite{lin2020gradient,boct2020alternating}, \agdap{} imposes a larger ratio, i.e., $\sigma_t/\tau_t=\Theta(\kappa^4)$; but, this is the price we pay as \agdap{} is agnostic to the parameters $L$ and $\mu$ while \cite{lin2020gradient,boct2020alternating} are not. Indeed, in~\cref{sec:proof_overview}, we briefly explain why our complexity bound of $\cO(\na{L}\kappa^4\epsilon^{-2})$ for the (proximal) GDA with backtracking, i.e., \agdap, has an extra $\kappa^2$ factor when compared to $\cO(L\kappa^2\epsilon^{-2})$ complexity derived for \gda{} and \agda{} in~\cite{lin2020gradient,boct2020alternating}. \na{Neither for \neada{} nor \tiada{}, the explicit dependence on $\kappa$ and $L$ of their complexity bounds are not stated in respective papers; to properly compare their bounds with ours, in \cref{sec:complexity-others} we derive \neada{} and \tiada{} complexities explicitly showing their dependence on $\kappa$ and $L$ --see also \cref{table:related_work}.}}

\sa{To briefly summarize, for solving nonconvex--strongly concave minimax problems as in~\eqref{eq:main-problem}, we propose \agdap{} through combining prox GDA updates in~\eqref{eq:gda} with \na{a \textit{nonmonotone} backtracking procedure}. Given the current iterate $(\bx_t,\by_t)$, to design an adaptive, non-diminishing step size sequence satisfying the necessary time-scale ratio through backtracking (without requiring to know neither $L$ nor $\mu$), we %need to 
come up with easy-to-check backtracking conditions based on some \textit{local} smoothness of $f$ around $(\bx_t,\by_t)$. Our objective is to use backtracking on the dual step size ($\sigma_t$) and choose the primal step size ($\tau_t$) adaptively until the backtracking condition is satisfied for $(\tau_t,\sigma_t)$. Then $(\bx_{t+1},\by_{t+1})$ is computed \na{according to %the prox \gda{} within 
the proximal GDA} update rule in~\eqref{eq:gda}.}
\vspace*{-3mm}
\paragraph{Our contributions:} %\rv{Suppose $L\gg 1$ and $0<\mu\ll 1$. For the simplicity of $\cO(1)$ constants, suppose $L=\Theta(1/\mu)$.} 
\begin{enumerate}[label=\roman*., nosep]%[topsep=0pt, itemsep=-1ex]
    \item \sa{We propose a %single-loop 
    proximal GDA algorithm with backtracking, \agdap, for solving deterministic WCSC minimax problems as in \eqref{eq:main-problem}. To the best of our knowledge, this is the \textit{first} %one-step
    \na{\textit{nonmonotone}} backtracking %technique
    framework with complexity guarantees for this general class of structured minimax problems, and \na{it requires at most $3$ gradient calls on average between two backtracking iterations (see~\cref{thm:avg--backtrack-time}).} %while the existing methods can only applied to certain special cases or need the possible unknown parameters; see more details in the related work. 
    Compared to the existing adaptive methods for WCSC minimax problems, which use %diminishing
    monotonically decreasing step size sequence, our method can take larger steps, 
    non-diminishing to $0$ with a guarantee that
    $\liminf_t \sigma_t=\Theta(1/L)$ and $\limsup_t \sigma_t/\tau_t$ is finite, leading to better $\cO(1)$ complexity constants and better empirical performance.} 
    %see more details in \cref{toy-example}.}
    \item %We show that the output of
    \sa{\agdap{} %is 
    can generate} an $\epsilon$-stationary point \sa{(in terms of a metric based on Nash equilibrium)} %under 1st order SP optimality metric for the objective function 
    with $\cO(\sa{L}\kappa^4 B_0%L^4 
    \epsilon^{-2})$ %iterations.
    \sa{gradient complexity}, where $B_0=F(x_0)-F^*$ denotes the primal suboptimality of the initial point \rv{--for the case $\mu$ is unknown, see the discussion in \cref{remark:mu-discussion-with-max-solver,remark:mu-discussion-without-max-solver}}. This complexity result is optimal in terms of \sa{its} dependence on $\epsilon$. \sa{Compared to another single-loop method \tiada{} and a double-loop method \neada{}, of which gradient complexities are $\cO(\kappa^{10} \epsilon^{-2})$ and $\cO(L^4\rv{\kappa^{5}}\epsilon^{-2}\log(1/\epsilon))$, the $\cO(1)$ constant of \agdap{} gradient complexity has a better dependence on the problem parameters $\kappa$ and $L$}.
    %\qstodo{should we use "best" rather than "better"?}
    \item \sa{Neither \neada{} nor \tiada{} can handle closed convex functions $g$ and $h$; hence, \agdap{} is the first adaptive proximal method for solving WCSC minimax problems of the form in~\eqref{eq:main-problem} \na{that can adapt to the fluctuations of local Lipschitz constants over the problem domain using nonmonotone step sizes.} \na{More precisely, in contrast to \sgdab{} which only estimates the global Lipschitz constant (that is why \sgdab{} keeps the step size constant for $T=\cO(L\kappa^2/\epsilon^2)$ inner iterations before checking the backtracking condition), \agdap{} keeps track of two estimates: one for the \textit{local} Lipchitz constant corrsponding to the current iterate and the other one for the \textit{global} constant, which allows the method to adapt to the local topology.}}
    \item {Finally, we demonstrate the efficiency of \agdap{}  on a %distributionally robust learning 
    \sa{DRO} problem over a multi-layer perceptron.}
\end{enumerate}
\vspace*{-3mm}
\paragraph{Notation.}\na{Throughout the paper, $\mathbb{N}$ denotes the set of nonnegative integers, and $\mathbb{N}_+\triangleq\mathbb{N}\setminus\{0\}$; moreover, let $\bar\N\triangleq\N\cup\{+\infty\}$. Whenever we use extended reals for comparison, we assume that (i) $L'<\infty$ for any $L'\in\reals$, (ii) $+\infty=+\infty$ and $+\infty\not <+\infty$. Given $n\in\N$, $[n]\triangleq\{k\in\N:~k\leq n\}$. We also adopted the Landau notation, i.e., for $f,g:\reals_+\to\reals_+$, we use $f=\cO(g)$ and $f=\Omega(g)$ if there exists some $C>0$ and $n_0\in\reals_+$ such that $f(n)\leq C g(n)$ and $f(n)\geq C g(n)$, respectively, for all $n\geq n_0$; moreover, if $f=\cO(g)$ and $f=\Omega(g)$, then we use $f=\Theta(g)$.}
\vspace*{-2mm}
\begin{algorithm}[h]
\caption{\texttt{\agdap{}}}
{\small
\begin{algorithmic}[1]
\STATE{\bf Input:} \na{$(\bx_0,\by_0)\in \mathcal{X}\times\cY$},\quad $\zeta,\tilde{\mu},\tilde l>0\quad \mbox{s.t.}\quad \tilde l>\tilde\mu$,\quad  $\gamma_0,\gamma\in(0,1)$,\quad $r\in\mathbb{N}_+$ 
\STATE{\bf Initialize:}  
\na{$\gamma_1\gets\gamma^r$,\quad $\gamma_2\gets\gamma$, \quad %$\tilde l\gets \xzf{\tilde{\mu}}/\gamma_2$,\quad
$l_0\gets \tilde l$,\quad $L_0\gets \tilde l$,\quad  ${\mu_0\gets \tilde{\mu}}$,\quad $\underline{\mu}\gets
\begin{cases}
    \mu,\ \text{if}\ \mu\ \text{is known}\\
    0,\ \text{otherwise}
\end{cases}$} 
%\STATE \na{{\bf Compute} $\by_0:\quad \cL(\bx_0,\by_0)+\zeta\geq \max_{\by\in\cY}\cL(\bx_0,\by)$} \alglinelabel{algeq:y0}
\STATE $t\gets 0$,\quad \na{$l_t^\circ\gets l_t$,}\quad \na{$\texttt{flag}\gets\texttt{False}$}
\LOOP
    \IF{$\texttt{max\_solver}==\texttt{True}$ \textbf{and} $\texttt{flag}==\texttt{False}$}
    \STATE \na{{\bf Compute} $\by_t:\quad \cL(\bx_t,\by_t)+\zeta\geq \max_{\by\in\cY}\cL(\bx_t,\by)$} \alglinelabel{algeq:haty-max}
    \ELSIF{$\texttt{max\_solver}==\texttt{False}$}
    \STATE $\hat\by_t\gets\prox{h/L_t}\big(\by_t+\frac{1}{L_t}\grad_y f(\bx_{t},\by_t)\big)$, \quad $d_t\gets\norm{\hat\by_t-\by_t}$ \alglinelabel{algeq:haty-d}
    \STATE $\by_t\gets\hat\by_t$  \alglinelabel{algeq:haty-y}
    \ENDIF
    \STATE $\texttt{flag}\gets\texttt{True}$
\WHILE{${l}_t \leq {L}_t$}
\STATE $\sigma_{t} \gets  1/l_{t}$
\STATE $\tau_t \gets
\frac{(1-\gamma_0)}{l_t}\left({4+\frac{1}{\gamma_2}} +{\frac{4(1-\sigma_{t} {\mu_t})(2-\sigma_{t}{\mu_t})(15L_{t}-8{\mu_t})L_{t}^3}{\mu_t^4}} \right)^{-1}$
\alglinelabel{algeq:step-size}
\STATE \na{$\tilde\bx_{t+1}\gets\prox{\tau_t g}\big(\bx_t-\tau_t\grad_{\bx} f(\bx_t,\by_t)\big)$} 
\STATE \na{$\tilde\by_{t+1}\gets\prox{\sigma_t h}\big(\by_t+\sigma_t\grad_y f(\tilde\bx_{t+1},\by_t)\big)$} \alglinelabel{algeq:y-update}
\IF{ \sa{$\texttt{flag}==\texttt{True}$}}
\IF{$\texttt{max\_solver}==\texttt{True}$}
\STATE $\Delta_{t}\gets \frac{2\zeta}{{\mu_t}}$,\quad\sa{$\Lambda_t\gets \zeta + \sqrt{\frac{{2}\zeta}{{\mu_t}}}~L_t\norm{\by_t -\na{\tilde\by_{t+1}}}$} \alglinelabel{algeq:delta_def1}
\ELSE
\STATE \na{$\Delta_t\gets \min\{(1+\frac{2L_t}{{\mu_t}})^2 d_t^2,~\bar\cD_y^2\}$,\quad $\Lambda_t\gets 2d_t L_t\norm{\by_t -\na{\tilde\by_{t+1}}}$} \alglinelabel{algeq:delta_def2}
\ENDIF
\STATE $R_t\gets 0$
\ENDIF
\IF{\na{$\texttt{BacktrackCond}(t)==\texttt{True}$} \alglinelabel{algline:check-agda+} } 
    \STATE \na{$\bx_{t+1}\gets\tilde\bx_{t+1}$, \quad $\by_{t+1}\gets\tilde\by_{t+1}$}
    \STATE {${l}_{t+1} \gets \na{\max\{\gamma_2 l_{t}, \tilde l\}}$,\quad \na{$l_{t+1}^\circ\gets l_{t+1}$},\quad $L_{t+1}\gets L_t$,\quad {$\mu_{t+1}\gets\mu_t$} \alglinelabel{algeq:l-update}}
    \STATE $C_t\gets\frac{(1-\sigma_{t} {\mu_t})(2-\sigma_{t} {\mu_t} )}{\sigma_{t} {\mu_t}} \frac{L_{t}^2}{\mu_t^2} \tau_{t}^2$
    \STATE $\Delta_{t+1}\gets (1-{\mu_t} \sigma_{t}/2)\Delta_t + C_t \na{\|G^{\tau_t}_x(\bx_t,\by_t)\|^2}$, 
    \alglinelabel{algeq:delta-update}
    %\STATE \sa{$\Lambda_{t+1} \gets 3(2\Delta_{t} + \Delta_{t+1})/\sigma_t$}
    \STATE \na{$\Lambda_{t+1} \gets
    6l_{t}(\Delta_{t+1}+2\Delta_t)-8{\mu_t}\Delta_t$} \alglinelabel{algeq:Lambda-t>tn}
    \STATE \na{$R_{t+1}\gets 2\tau_{t}^2l_{t}\left\|G^{\tau_{t}}_x(\bx_{t},\by_{t})\right\|^2-\sigma_{t}^2{\mu_t}\left\|{G^{\sigma_{t}}_y(\bx_{t+1},\by_{t})}\right\|^2$}
    \alglinelabel{algline:Rt} 
    \STATE  $t\gets t+ 1$,\quad \sa{$\texttt{flag}\gets\texttt{False}$}
\ELSE
    \STATE ${l}_t \gets {l}_t/\gamma_2$
\ENDIF
\ENDWHILE
\STATE $L_t\gets L_t/\gamma_1$,\quad {$\mu_t\gets \max\{\mu_t \gamma_1, \underline{\mu}\}$}
\ENDLOOP
\STATE \textbf{Recommended initialization:} $\gamma_0=1e-3,\ \gamma =0.95,\ r=2,\ \rv{\tilde l=\tilde{\mu}/\gamma}$.
\end{algorithmic}}%
\label{alg:agda}
\end{algorithm}
%\vspace*{-3mm}

\section{The \agdap{}~method}\vspace*{-3mm}
In this section, we propose \agdap, as outlined in \cref{alg:agda},
which uses the following backtracking conditions:\\
% \todo[inline]{Old 5(a): $\tau_t(1 - \na{2}\tau_t l_t ) \norm{G^{\tau_t}_x (\bx_t, \by_t)}^2 + \sigma_t \norm{G^{\sigma_t}_y (\bx_{t},\by_t)}^2\leq
%         \na{\Lambda_t+4(3l_t-2\mu)\Delta_t}+\cL(\bx_t, \by_t) - \na{\cL(\tilde\bx_{t+1},\tilde\by_{t+1})}-\frac{\sigma_t^2 \mu}{2} \norm{G^{\sigma_t}_y (\na{\tilde\bx_{t+1}},\by_t)}^2$}
% \todo[inline]{The tight form of (5a) should have its first term as $\Big(\tau_t - \Big(2+\frac{l_{t-1}}{l_t}\Big)\tau_t^2 l_t
%             \Big)\norm{G_x^{\tau_t}(\bx_t,\by_t)}^2$}
\boxed{
\begin{minipage}{.15\textwidth}
\boxed{\cond{t}:}\quad
\end{minipage}
\begin{minipage}{.85\textwidth}
\vspace*{-3mm}
\begin{subequations}
\label{eq:backtrack}
\begin{align}
        \na{\small
        \begin{aligned}
         \MoveEqLeft[2]\Big(\tau_t - \Big(2+\frac{1}{\gamma}\Big)\tau_t^2 l_t
            \Big)\norm{G_x^{\tau_t}(\bx_t,\by_t)}^2 + {\sigma_t}\norm{{G^{\sigma_t}_y(\bx_{t},\by_t)}}^2+
            \sigma_{t}^2\frac{{\mu_t}}{2}\left\|{G^{\sigma_{t}}_y(\tilde\bx_{t+1},\by_{t})}\right\|^2 
            \\
            &
        \leq  \Lambda_t+4(3l_t-2{\mu_t})\Delta_t+\cL(\bx_t, \by_t) - \cL(\tilde\bx_{t+1},\tilde\by_{t+1})+R_t
    \end{aligned}} \label{eq:ls-conv}\\
    \na{\small
    \begin{aligned}
        \MoveEqLeft[2] f(\tilde\bx_{t+1},\by_{t}) + \langle \grad_y f(\tilde\bx_{t+1}, \by_{t}), \tilde\by_{t+1} - \by_{t} \rangle\leq f(\tilde\bx_{t+1}, \tilde\by_{t+1})+ \frac{l_{t}}{2} \|\tilde\by_{t+1} - \by_{t}\|^2
    \end{aligned}} \label{eq:Lyy-1}\\
    {\small
    \begin{aligned}
    \MoveEqLeft[2]\|\nabla_y f\left(\na{\tilde\bx_{t+1},\tilde\by_{t+1}}\right)-\nabla_y f\left(\na{\tilde\bx_{t+1}}, \by_{t}\right)\| \leq  l_t\left\|\na{\tilde\by_{t+1}}-\by_{t}\right\|
    \end{aligned}} \label{eq:Lyy-2}\\
    {\small
    \begin{aligned}
    \MoveEqLeft[2]\norm{G^{\sigma_t}_y(\tilde\bx_{t+1},\by_{t})}^2 \leq 2\left(\tfrac{4(1-\sigma_t{\mu_t})}{\sigma^2_{t}}+2l_t^2 \right) \Delta_t + 2l_t^2\tau_t^2\norm{G_x^{\tau_t}(\bx_{t},\by_{t})}^2,
    \end{aligned}} \label{eq:final-ls-xt+1yt}
\end{align}
\end{subequations}%
\hspace{-25mm} where $\Lambda_t$ and \na{$R_t$} are defined within \cref{alg:agda}.
\end{minipage}}\\[2mm]
%where $\Lambda_t$ and \na{$R_t$} are defined within \cref{alg:agda}.
\begin{remark}
\rv{Given $\tilde l,\tilde\mu>0$ such that $\tilde l>\tilde \mu$, we have $l_t \geq \tilde l > \tilde\mu$ and $\mu_t \geq \tilde\mu$ for $t\geq0$; hence, $\mu_t \sigma_t = \frac{\mu_t}{l_t} < \frac{\mu_t}{\tilde\mu}\leq 1$.}
\end{remark}
%If $\cD_y$ is known, then in \cref{alg:agda}, we set $\bar D_y = \cD_y$; otherwise, we set $\bar D_y = \infty$ for the case $\cD_y$ is unknown. %we consider two cases for \texttt{correction=False}: i) when $\cD_y$ is known, set $\bar D_y = \cD_y$; ii) when $\cD_y$ is unknown, set $\bar D_y = \infty$.
%\vspace*{-2mm} %\xtodo{Note: Algorithm line reference needs correction}
% Following that, we will provide a detailed discussion of the specific backtracking conditions, elaborating on their importance and the implications they hold within the context of the algorithm.

\na{\agdap{}, displayed in \cref{alg:agda}, takes \ms{} flag as an input. If \msf{True}, one needs to \textit{inexactly} solve a strongly concave maximization problem every time the global Lipschitz constant estimate $L_t$ increases, i.e., whenever local Lipschitz constant estimate $l_t$ goes beyond the global estimate $L_t$, according to \agdap{}, one increases $L_t$, decreases $\mu_t$, and replaces the current $\by_t$ iterate by a $\zeta$-optimal solution of $\max_{\by\in \cY} \cL(\bx_t,\by)$. Later in \cref{thm:avg--backtrack-time}, we show that the number of such subproblem solves is bounded above by \rv{$\cO(\log(\cR))$ where
$\cR \triangleq \max\{\tilde\mu/\mu,~L/\tilde l,~1\}$.} This technique %benefits us to
helps us in getting a \rv{slightly} better gradient complexity bound by %decreasing 
reducing the dual error $\norm{\by_t-\by^*(\bx_t)}$ when $L_t\leq L$. On the other hand, when \msf{False}, \rv{the implementation becomes much simpler as} \agdap{} does not rely on dual error reduction mechanism through inexactly solving maximization subproblems; instead, it 
%uses a crude 
\rv{employs a bound on the dual error that is} based on a simple gradient update.}

\na{In the first result, we establish that \agdap{} is well-defined.}
\begin{theorem}
    \label{thm:backtrack}
    \na{For all $t\in\mathbb{N}$, the backtracking condition in~\eqref{eq:backtrack} holds in finite time, i.e., \emph{\shortalgline{\ref{algline:check-agda+}}} in \agdap{} is checked finitely many times for each $t$; hence, \agdap{} is well defined for \emph{$\texttt{max\_solver}\in\{\texttt{True}, \texttt{False}\}$.}} 
\end{theorem}
\begin{proof}
See \cref{appendix:ls_condition}.
\end{proof}

\na{We next discuss the motivation, rationale underlying \agdap{} and the associated challenges with it.} %the development of this algorithm.
\subsection{Motivation for \agdap{}~and challenges}  %algorithm construction}
\sa{Compared to the \textit{simple} backtracking approaches used for nonconvex minimization problems~\cite{ beck2014introduction}, %where the convergence of iterates at a given time step $t$ depends solely on parameters from the preceding steps, the scenario is 
designing a backtracking scheme for %nonconvex-(strongly) concave 
WCSC minimax problems is notably harder and requires an intricate analysis. Indeed, consider the minimization problem 
$-\infty<p^*\triangleq\min_{\bx\in\cX} p(\bx)$ where $p(\cdot)$ is $L$-smooth, and let $\{\bx_t\}_{t\geq 0}$ be generated according to $\bx_{t+1} = \bx_t - \tau_t \grad p(\bx_t)$ for $\tau_t>0$ satisfying 
\begin{equation}\label{eq:descent-min}
    {\small
    \sa{\frac{\tau_t}{2}\|\grad p(\bx_t)\|^2} \leq  p(\bx_t)-p(\bx_{t+1}).}%
\end{equation}
Due to well-known descent inequality for smooth functions~\cite{nesterov_convex}, \eqref{eq:descent-min} is guaranteed to hold for all $\tau_t \in (0, 1/L]$; therefore, one can trivially use backtracking on $\tau_t$ until \eqref{eq:descent-min} holds, i.e., given some $\widetilde{L}>0$ and $\gamma\in(0,1)$, for each $t\geq 0$, one can initialize $\tau_t\gets 1/\widetilde{L}$, and shrinks $\tau_t\gets\gamma\tau_t$ until \eqref{eq:descent-min} holds, which requires at most $\cO(\log(L/\widetilde{L}))$ backtracking iterations for each $t\geq 0$. Thus, summing up \eqref{eq:descent-min} from $t=0$ to $T-1$ yields 
%that
{\small
\begin{equation*}
   %T\frac{\gamma}{2L}
   \min_{t=0,\ldots, T-1}\|\grad p(\bx_t)\|^2 \leq
   %\sum_{t=0}^{T-1}  \frac{\tau_t}{2}\|\grad p(\bx_t)\|^2 \leq  
   \frac{2L}{\gamma}(p(\bx_0)- p^*)/T, %p(\bx_{T})
\end{equation*}}%
which 
%implies the convergence of the algorithm under certain assumptions.
follows from $\tau_t\geq \gamma/L$ for $t\geq 0$.
%Therefore, one can simply search a proper $\sigma_t$ for each step such that \cref{eq:decent-min} holds. 
The important point here is that given $\bx_t$, the choice of $\tau_t$ using \eqref{eq:descent-min} is independent of previous steps $\{\tau_i\}_{i=0}^{t-1}$; moreover, the future steps $\{\tau_i\}_{i>t}$ do not impose any restrictions on $\tau_t$ choice.} 
%searches of $\{\sigma_t\}_{t\geq0}$ are independent for each $t$.

Compared to the naive backtracking method discussed above for smooth minimization, backtracking methods for \textit{convex-concave} minimax problems \cite{hamedani2021primal,jiang2022generalized} and their analyses are more complicated; that said, given the current iterate $(\bx_t,\by_t)$, the admissible primal-dual step sizes $(\tau_t,\sigma_t)$ can still be chosen independent of the past and future step sizes.
However, %in the case of nonconvex-(strongly) concave 
this is not the case for WCSC minimax problems; indeed, the analyses provided
%current analytical framework presented 
in the literature for GDA-type methods, e.g., \cite{lin2020gradient,boct2020alternating}, suggest that, given $(\bx_t,\by_t)$ at iteration $t$, $(\tau_t,\sigma_t)$ should be chosen incorporating the effect of future iterates $\{(\bx_i,\by_i)\}_{i=t+1}^{\infty}$ along with the future step sizes $\{\sigma_i,\tau_i\}_{i=t+1}^{\infty}$. %significantly influence the convergence analysis for the preceding sequence $\{\bx_i,\by_i\}_{i\leq t}$. Specifically,
More precisely, for given $T\in\N_+$, the existing analysis in the literature provide guarantees of the %following 
form:
\begin{equation*}
{\small
   \sum_{t=0}^T \nu_t\Big(\{\tau_i,\sigma_i\}^{T}_{i= t}\Big){\|G^{\tau_t,\sigma_t}(\bx_t,\by_t)\|^2}= \cO(1),}%
\end{equation*}
where $\nu_t(\cdot)$ is a function of current and future step sizes. Because of this complicated form, choosing $\{\tau_i,\sigma_i\}^{T}_{i=0}$ such that $0<\nu_t\big(\{\tau_i,\sigma_i\}^{T}_{i= t}\big)=\Theta(1)$ for $t=0,\ldots,T$ is not trivial at all if $L$ is \textit{not known} --existing results go around this problem through choosing constant step sizes $(\tau_t,\sigma_t)=(\tau,\sigma)$ for $t\geq 0$ such that $\sigma=\Theta(1/L)$ and $\tau=\Theta(1/(L\kappa^2))$, assuming \rv{$L$ and $\mu$ are known.}
%As a result, the searches of $\{\tau_t,\sigma_t\}_{t\geq 0}$ are not independent any more. Therefore, exerting control over future step sizes becomes preferable within this framework. 
%Indeed,

\sa{To overcome the difficulty due to future affecting the past, in \agdap, stated in Algorithm \ref{alg:agda}, the variable $L_t$ is strategically introduced to %manage 
account for the effect of future stepsizes, while $l_t$ represents the local Lipschitz constant that we seek for \na{using a nonmonotone backtracking procedure.} %that is sought. 
%To our knowledge, it remains an open question whether it is possible to derive a descent equation suitable for use as a backtracking condition, without necessitating consideration of future updating information.
%The existence of $l_t$ notably enhances Algorithm \ref{alg:agda}, facilitating the identification of a smaller local Lipschitz constant even when operating under a larger control parameter $L_t$. 
Since $L$ and $\mu$ are not known (hence, $\kappa$), in \agdap, we propose tracking the time-scale ratio in an adaptive manner by setting $\sigma_t=1/l_t$ and imposing the ratio
{\small $\sigma_t/\tau_t=\na{\frac{4}{1-\gamma_0}\Big(1+\frac{1}{4\gamma_2} + (1-{\mu_t}/l_t)(2-{\mu_t}/l_t)\frac{(15L_t-8\mu_t)L_{t}^3}{\mu_t^4}\Big)=\cO\Big(\frac{L_t^4}{\mu_t^4}\Big)}$}, which involves $\mu_t,l_t$ and $L_t$, i.e., see 
%establishes an appropriate nonlinear time ratio for $\sigma_t$ and $\tau_t$, as stated in 
\algline{\ref{algeq:step-size}}.} %in \cref{alg:agda}. 
%It's important to note that even if we set $l_t$ equal to $L_t$ directly, without any search process, 
\sa{In an alternative implementation, one can set $l_t=L_t$, i.e., replace $l_t$ with $L_t$ everywhere in \agdap, and backtrack $L_t$ through checking \eqref{eq:backtrack} rather than $l_t$ \na{--this corresponds to choosing $r=1$; hence, $\gamma_1=\gamma_2$.} For this alternative, the algorithm statement becomes simpler while retaining its theoretical convergence guarantee. However, this approach forfeits the flexibility offered by being able to search for local Lipschitz constants, and leads to a strictly decreasing $\sigma_t$ and $\tau_t$, potentially causing worse practical performance. 
%which in turn can lead to suboptimal performance in practical applications. 
Therefore, we adopted the double search as in \cref{alg:agda}, i.e., both $L_t$ and $l_t$, as this approach \na{can exploit local smoothness leading to better performance compared to searching for $L_t$ alone, i.e., potentially leading to non-monotonic step sizes through adapting to decreases in local smoothness constants along the trajectory of the iterate sequence.}}

%Another significant challenge pertains to the computational feasibility of backtracking conditions in practical scenarios. 
\sa{Designing a backtracking condition that is easy-to-check in practice is a significant challenge due to reasons explained above, \na{i.e., mainly because of future step sizes affecting the past ones in the existing analysis, and of the requirement for imposing time-scale separation without knowing $L$ \rv{and/or $\mu$}}. Indeed, the majority of existing algorithmic analysis for WCSC minimax problems rely on 
%Much of the current analysis of nonconvex-(strongly) concave mini-problems is heavily dependent on 
terms that are either computationally expensive to evaluate or intractable in practice such as 
%For instance, a commonly used convergence metric is 
the gradient mapping \rv{$\cG^\tau(\cdot)$ or the gradient for Moreau Envelope of the primal function, i.e., $\grad F_\lambda(\cdot)$, where $F_\lambda(\bx)\triangleq\min_{\bw\in\cX}F(\bw)+\frac{1}{2\lambda}\norm{\bw-\bx}^2$,} or quantities such as $\delta_t\triangleq\norm{\by_t-\by^*(\bx_t)}^2$, e.g., see~\cite{lin2020gradient,boct2020alternating,yang2022faster,chen2022accelerated}. It should be emphasized that a method being parameter agnostic to $L$ and $\mu$ due to its use of adaptive step sizes does not immediately imply that the method is suitable for backtracking; for instance, %\cite{yang2022nest} 
while \tiada{}~\cite{li2022tiada} is an adaptive method providing guarantees on 
%tractable measures such as $G(\bx,\by)$
an easy-to-compute metric $\frac{1}{T}\sum_{t=0}^{T-1}\norm{\grad f(\bx_t,\by_t)}^2$, some crucial parts of its analysis %some of their analyses 
still depend on the primal function $F(\bx)$; hence, it is not trivial to modify this method to incorporate it within a backtracking framework. On the other hand, while some existing methods such as \cite{xu2023unified} are more suitable for backtracking, their complexity bounds have suboptimal $\kappa$-dependence even if $L$ is known --see~\cref{table:related_work}.} 

\sa{\na{In this paper,} we have addressed all these issues and proposed, for the first time in the literature, a first-order method with \na{a \textit{nonmonotone} step-size search/backtracking procedure} (agnostic to the \textit{global} smoothness parameter $L$ {and to the convexity modulus $\mu$}) for computing \na{approximate stationary points} of WCSC minimax problems with provable complexity guarantees.
%in our analysis, achieving lower convergence complexity.
Similar to the existing analyses, $\delta_t$ also arises in our convergence proof, and we need to control this quantity.
%Indeed, for the sole untractable term $\delta_t$, which represents the dual optimality error, we have devised an alternative, 
To bound $\{\delta_t\}$, we constructed an auxiliary control sequence $\{\Delta_t\}$ such that, \textit{unlike} $\delta_t$, $\Delta_t$ can be easily computed in practice, and %Detailed discussion on this will be provided in the subsequent sections \tbd.
we argued that $\delta_t\leq \Delta_t$ whenever $L_t\geq L$ {and $\mu_t \leq \mu$}. 
%While $\Delta_t$ serves as a measure to control the dual optimality gap, we establish that a sufficient condition for this control is ensuring $L_t\geq L$. However, the relationship between $\delta_t$ and $\Delta_t$ remains unclear when $L_t\leq L$. 
Throughout \na{the runtime of \agdap{}}, the estimate for the global constant $L_t$ always %denotes 
provides an upper bound on the local Lipschitz constant $l_t$. Since every time the backtracking condition in~\eqref{eq:backtrack} fails we increase our estimate using $l_t\gets l_t/\gamma_2$, at some point we may end up $l_t/\gamma_2>L_t$, which immediately implies that $L_t<L$. Whenever this happens, we need to correct $\Delta_t$ bound by resetting $\by_t$ \na{using one of two different ways: \textit{(i)} if \msf{True}, then for a given tolerance $\zeta>0$, $\by_t$ is reset to} an inexact solution of the inner-max problem such that $ \cL(\bx_t,\by_t)+\zeta\geq F(\bx_t)$, i.e., see \algline{\ref{algeq:haty-max}} and \na{this would imply that $\delta_t\leq\frac{2\zeta}{\rv{\mu_t}}=\Delta_t$} --one can easily argue that such corrections can only happen %$\cO(\log_{1/\gamma_1}(\kappa))$ 
\rv{$\cO(\log(\cR))$ times, where
$\cR \triangleq \max\{\tilde\mu/\mu,~L/\tilde l,~1\}$,} and for each correction, the inexact maximizer can be computed within $\cO(\sqrt{\kappa}\log(\cD_y^2/\zeta))$ gradient calls without requiring $L$ when $\mu$ is known, e.g., see~\cite{calatroni2019backtracking,rebegoldi2022scaled} -- otherwise, the gradient complexity can be bounded by $\cO(\kappa\log(\cD_y^2/\zeta))$ when both $\mu$ and $L$ are unknown; \na{alternatively, \textit{(ii)} if \msf{False}, then $\by_t$ is reset to $\hat\by_t=\prox{\frac{1}{L_t} h}\big(\by_t+\frac{1}{L_t}\grad_y f(\bx_{t},\by_t)\big)$ and one can use a crude bound on $\delta_t$, i.e., if $\cD_y^2$ is available, then one can set $\Delta_t=\cD_y^2$; otherwise, one can still bound $\delta_t$ by exploiting the strong concavity of $f(\bx_t,\cdot)$ leading to the bound $\delta_t\leq (1+\frac{2L_t}{\mu_t})^2 \norm{\hat\by_t-\by_t}^2=\Delta_t$. %Indeed, within the alternative implementation, i.e., using crude bounds on $\delta_t$,
\rv{To compare these two alternatives, consider the scenario where $\mu$ and $\cD_y$ are known, and we adopt the initialization $\tilde\mu=\mu$ and $\tilde l=\mu/\gamma$,} e.g., for DRO problem in~\eqref{eq:dro-problmm} the inner maximization is solved over a unit simplex and $\mu>0$ is the value of the regularizer set by the practitioner. \rv{Under this setting, for the case \msf{False},} one does not need to solve any inner max problem, and the corresponding gradient complexity of \agdap{} is $\cO(L\kappa^4(B_0+\kappa L\cD_y^2)\epsilon^{-2})$; on the other hand, 
for \msf{True}, we require additional $\cO(\sqrt{\kappa}\log(\kappa)\log(1/\epsilon))$ gradient calls in total for solving at most $\log_{1/\gamma_1}(\kappa)$ inner max problems with accuracy $\zeta=\epsilon^2$, and this corrective action \rv{slightly} improves the gradient complexity of \agdap{}~to $\cO(L\kappa^4(B_0+\log(\kappa) L\cD_y^2)\epsilon^{-2})$, where $B_0=F(\bx_0)-F^*$ denotes the initial suboptimality.}}
\subsection{Convergence guarantees for \agdap{}}
% \qsr{In this section, we denote that \xzf{$\bar L  \triangleq \max\{L, \frac{{\tilde\mu \tilde l}}{ \mu}\}$, $\bar \mu \triangleq \max\{\underline{\mu}/\gamma_1, \min\{\mu, \frac{{\tilde\mu}\tilde l}{ L}\}\}$, $\tilde\kappa\triangleq\max\{\frac{\bar{L}}{\tilde l}, \frac{\tilde \mu}{\bar\mu}\}$} and $\bar{\kappa}=\frac{\bar L}{\bar \mu}$.}\\
% On average each \agdap{}~iteration needs at most $3$ backtracking/step-size search iterations. 
%rather than $\cO(\log_{1/\gamma_2}^2(\kappa))$. 
%\xtodo{T should be $T=\Omega(\log^2(\kappa))$, when $T$ is not infinity, the average is $\leq 4$. If $T\rightarrow\infty$, the average is 2.}
%\xtodo{\cref{thm:avg--backtrack-time} only need \cref{ASPT:lipshiz gradient}}
\begin{theorem}\label{thm:avg--backtrack-time}
    % Suppose \cref{ASPT:lipshiz gradient} holds. %\cref{ASPT:lipshiz gradient,aspt:bounded-Y,aspt:primal_lb} hold. 
    % \sa{Let $I_t$ denote the number of backtracking iterations required at iteration $t$ of \agdap. Then $\frac{1}{T}\sum_{t=0}^{T-1}I_t\xz{\leq 4}$ for all $T=\Omega(\log^{\xz{2}}(\kappa))$ independent of $\epsilon>0$. Moreover, $\lim_{T\to\infty}\frac{1}{T}\sum_{t=0}^{T-1}I_t=2$.}
    % %Then the average times for checking backtracking conditions of \cref{alg:agda} is 2.
    \na{Suppose \cref{ASPT:lipshiz gradient} hold. Throughout an \agdap{} run, $L_t$ increases {and $\mu_t$ decreases} at most $\lceil \frac{1}{r}\log_{1/\gamma}(\rv{\cR})\rceil$ many times, where $\rv{\cR \triangleq \max\{\tilde\mu/\mu,~L/\tilde l,~1\}}$. Moreover, {let $I_t$ denote the number of backtracking iterations required by \agdap{} at iteration $t\in\N$, i.e., the number of times \emph{\cond{t}} is checked within iteration $t$ for \cref{alg:agda}. Then $\frac{1}{T}\sum_{t=0}^{T-1}I_t\rv{\leq} 2+\frac{\log_{1/\gamma}(%\qsr{\frac{\bar L}{\tilde\mu}}
    \rv{\cR})}{T+1}$ for all $T\in\N$; hence, $\frac{1}{T}\sum_{t=0}^{T-1}I_t\leq 3$ for all $T\geq \lfloor \log_{1/\gamma}(%\kappa
    \rv{\cR})\rfloor$} %Finally, 
    and $\lim_{T\rightarrow \infty}  \frac{1}{T+1}\sum_{t=0}^T I_t \rv{\leq} 2$.}
\end{theorem}
\begin{proof}
See \cref{sec:backtrack-bound}.
\end{proof}
\sa{Now we %will show 
provide the main convergence result for \agdap.} 
% \begin{theorem}%[Complexity result for $\mu>0$]
% \label{thm:convergence-K-flag-true-main}
% Suppose \cref{ASPT:lipshiz gradient,aspt:bounded-Y,aspt:primal_lb} hold. %and $\mu>0$. 
% Given any $\bx_0\in\dom g$ and $\epsilon>0$,
% %and $\{\bx_t,\by_t,\tilde\by_t\}_{t\geq 0}$ are generated by \cref{alg:agda} with correction==True.
% \agdap{}\ can generate $(\bx_\epsilon,\by_\epsilon)$ such that \qsr{$\norm{ G^{\tau,\sigma}(\bx_\epsilon,\by_\epsilon)} \leq \epsilon$} within $\cO(\kappa^4 L^2 \epsilon^{-2})$ 
% %and $\cO(\kappa^5 L \epsilon^{-2})$ 
% \agdap{}~iterations. 
% %when \texttt{correction} is set to \texttt{True} and \texttt{False}, respectively.
% Moreover, each \agdap{}~iteration requires at most $\log_{1/\gamma_2}^2(\kappa)$ backtracking iterations.
% \end{theorem}
% \begin{proof}
% See 
% %\cref{thm:ls-hold-main}
% \tbd and \cref{cor:complexity-flag-true}.
% \end{proof}
\begin{theorem}
\label{thm:main}
    \na{Suppose \cref{ASPT:lipshiz gradient,aspt:bounded-Y,aspt:primal_lb} hold. For \agdap{}, displayed in~\cref{alg:agda},} then
\begin{equation}
\label{eq:complexity-bound-final}
{\small
\begin{aligned}
        \MoveEqLeft\sum_{t=0}^{T}\gamma_0\tau_t\norm{G_x^{\tau_t}(\bx_t,\by_t)}^2
        + {\sigma_t}\norm{G_y^{\sigma_t}(\bx_t,\by_t)}^2 
        + \frac{{\mu_t}\sigma_t^2}{2}\norm{G_y^{\sigma_t}(\bx_{t+1},\by_t)}^2\\ 
        &\leq \cL(\bx_0,\by_0)-F^* 
        + \rv{\Big(\Big\lceil \frac{1}{r}\log_{\frac{1}{\gamma}}(\cR)\Big\rceil+1\Big)\Big({\Big(\bar{L}/\gamma^r+\rv{L}\Big)\cD_y^2}+\bar\Lambda\Big)+60\tilde l\cdot\frac{\tilde l}{\tilde \mu}\frac{{\cR}^{2+\varsigma}}{1-{\gamma^{(2+\varsigma)r}}}\bar\Delta}\triangleq\Gamma(\bx_0,\by_0,\gamma_0,\gamma,r),
    \end{aligned}}%
\end{equation}
holds for all $T\in\N_+$ with $\varsigma=0$ if $\mu$ is known, and otherwise with $\varsigma=1$; moreover, $\bar L  \triangleq 
    %\max\{L, \frac{{\tilde{\mu}}}{\mu}\tilde{l}\}
    \rv{\cR\;\tilde l}$ and $\bar \mu \triangleq \max\{\underline{\mu}/\gamma_1,~\rv{\tilde\mu/\cR}
    %\min\{\mu, \frac{\tilde l}{L}{\tilde{\mu}}\}
    \}$ with $\rv{\cR \triangleq \max\{\tilde\mu/\mu,~L/\tilde l,~1\}}$, and
\begin{align*}
\bar{\Lambda}\triangleq
        \begin{cases}
           \zeta +\frac{\bar L}{\sqrt{\bar\mu}}\cD_y\sqrt{\frac{2\zeta}{\gamma^{3r}}}, &  \text{if}\; \emph{\msf{True}},\\
           \frac{2}{\gamma^r}{\bar{L}}\cD^2_y, &\text{if}\; \emph{\msf{False}}.
        \end{cases};\quad
\bar{\Delta}\triangleq
        \begin{cases}
           \frac{2\zeta}{{\gamma^r}}\frac{1}{\bar{\mu}}, &  \text{if}\; \emph{\msf{True}},\\
           \min\{(1+\frac{2}{\gamma^{2r}}\rv{\frac{\bar L}{\bar \mu}})\cD_y,~\overline{\cD}_y\}^2, &\text{if}\; \emph{\msf{False}}.
        \end{cases}
\end{align*}
\end{theorem}
\begin{proof}
    See~\cref{sec:conv-complexity-wcsc-agda}.
\end{proof}
% \xzf{
% \begin{lemma}\label{lemma:bound-kn-mut-Lt}
%     $k_n\leq \max\{\log_{1/\gamma_1}(\frac{\tilde{\mu}}{\mu}), \log_{1/\gamma_1}(\frac{\gamma_2L}{\tilde \mu})\}$, thus, $\mu_t\geq \bar{\mu}\triangleq \min\{\mu\gamma_1, \frac{\tilde{\mu}^2}{\gamma_2 L}\}$, and $L_t\leq \bar{L}\triangleq\max\{L/\gamma_1, \frac{\tilde{\mu}^2}{\gamma_2\mu}\}$. $\xzrr{\cR}=\max\{\tilde{\mu}/\mu, L/\tilde{\mu}\}$, $\bar{\kappa}=\bar{L}/\bar{\mu}$.
% \end{lemma}
% }
% \xzf{
% \begin{remark}
%     Suppose $L\geq 1/\mu$, then $\bar{\mu}=1/L$ and $\bar{L}=L$, and $\xzrr{\cR} = L$, $\bar{\kappa}=L^2\geq \kappa$.
% \end{remark}
% }
%\xtodo{Shall we only use $\varsigma$ for $\tilde \kappa$, and just set $\varsigma=1$ for $\gamma$ part? The current notation is a little bit confused for readers. }
\begin{corollary}\label{cor:complexity-flag-true}
\na{Suppose \cref{ASPT:lipshiz gradient,aspt:bounded-Y,aspt:primal_lb} hold, and \emph{\msf{true}}. Given arbitrary $\bx_0\in\dom g$, $\zeta>0$, $\gamma_0,\gamma\in(0,1)$ and $r\in\N_+$, let $\{(\bx_t,\by_t)\}_{t\in\N}$ denote the iterate sequence  generated by \cref{alg:agda} for the parameter choice $\tilde l=\tilde \mu/\gamma$. Then, for any given $\epsilon>0$,} $(\bx_\epsilon,\by_\epsilon)=(\bx_{t^*_\epsilon},\by_{t^*_\epsilon})$ is $\epsilon$-stationary in the sense of Definition~\ref{def:eps-stationarity} with $t^*_\epsilon\triangleq\argmin\{\norm{G^{\tau_t,\sigma_t}(\bx_t,\by_t)}:\ 0\leq t\leq T_\epsilon\}$ for \na{any} $T_\epsilon\in\N_+$ such that 
% \todo{the simplified complexity is $\frac{\bar{L}\bar{\kappa}^4}{\epsilon^{2}}$}
\begin{equation}
\label{eq:iter_complexity-true}
{\small
    % \begin{aligned}
    %         & T_\epsilon = \Omega\left(
    %  \frac{\cL(\bx_0,\by_0) - F^* 
    %  %\xqs{+(\frac{1}{\gamma_1} + \frac{3}{2})L\cD_y^2} 
    %  + \zeta
    %  \frac{\tilde{\mu}\xzrr{\cR}^{(2+\varsigma)}}{\bar{\mu}}
    %   /(\gamma^{(4+\varsigma)-(\xzf{2+\varsigma})r}(1-\gamma^{(2+\varsigma)r})) + \frac{1}{r}\log_{1/\gamma}(\xzrr{\cR}) \Big(\xzf{\frac{\bar{L}\cD_y}{\gamma^{3r/2}}}
    %  \max\Big\{\cD_y,~\sqrt{\frac{2\zeta}{\xzf{\bar{\mu}}}}\Big\} + L\cD_y^2 + \zeta 
    %  \Big)}{\xzf{\gamma^{8r+1}}\gamma_0(1-\gamma_0)}
    % \cdot\frac{\bar{L}\bar{\kappa}^4}{\epsilon^2}\right).
    % \end{aligned}}%
    \begin{aligned}
            & T_\epsilon = \Omega\left(
     \frac{\cL(\bx_0,\by_0) - F^*  
     + \zeta \frac{\tilde \mu}{\bar{\mu}}
     {\cR}^{(2+\varsigma)}
      /(\gamma^{r+2}(1-\gamma^{(2+\varsigma)r})) + \frac{1}{r}\log_{1/\gamma}({\cR}) \Big({\frac{\bar{L}\cD_y}{\gamma^{r}}}
     \max\Big\{\cD_y,~\sqrt{\frac{\zeta}{{\gamma^r \bar{\mu}}}}\Big\} + L\cD_y^2 + \zeta 
     \Big)}{{\gamma^{8r+1}}\gamma_0(1-\gamma_0)}
    \cdot\frac{\bar{L}\bar{\kappa}^4}{\epsilon^2}\right),
    \end{aligned}}%
\end{equation}
\rv{where $\varsigma=0$ if $\mu$ is known, and otherwise  $\varsigma=1$.} 
%Then, $(\bx_\epsilon,\by_\epsilon)$ is $\epsilon$-stationary in the sense of Definition~\ref{def:eps-stationarity}. %,i.e., $\norm{G(\bx_\epsilon,\by_\epsilon)} \leq \epsilon$.
Moreover, \agdap{} requires checking the backtracking condition in \eqref{eq:backtrack} at
most $3 T_\epsilon$ times. Thus, choosing $\zeta = \cO(\epsilon^2)$ implies that the total gradient complexity can be bounded by
\begin{equation}
\label{eq:grad_complexity}
\cO\left({\tilde \mu \bar{\kappa}^5{\cR}^{2+\varsigma} }+\log(\cR)\rv{\kappa^{(1+\varsigma)/2}}\log\Big(\frac{1}{\epsilon}\Big)+\frac{{\bar{L}^{3/2}\bar{\kappa}^{9/2}}\log({\cR})\cD_y}{\epsilon}+\frac{\cL(\bx_0,\by_0)-F^*+{\bar L}\cD_y^2\log(\cR)}{\epsilon^2}{\bar{L}\bar{\kappa}^4}\right);
\end{equation}
hence, for $0<\epsilon=\cO\Big(
%\frac{\sqrt{\xzf{\bar{L}}\cD_y^2}}{\kappa}
{\sqrt{\bar \mu}\cD_y}
\Big)$, the total gradient complexity can be bounded by $\cO\Big(\frac{\cL(\bx_0,\by_0)-F^*+{\bar L}\cD_y^2\log(\cR)}{\epsilon^2}{\bar L\bar{\kappa}^4}\Big)$. 
%\xtodo{I am not sure this $\epsilon=\cO(\sqrt{\bar\mu}\cD_y)$ is correct or not because I get a different value from arxiv version. (This comment is stil valid 11/7/2024)}
\end{corollary}
\begin{remark}\label{remark:mu-discussion-with-max-solver}
\rv{When $\mu$ is known, setting $\tilde\mu=\underline{\mu}=\mu$ and 
$\tilde l=\mu/\gamma$ implies that $\cR=\max\{1,\gamma\kappa\}$; hence, $\bar L=\max\{\mu/\gamma, L\}$, and $\bar\mu=\mu/\gamma^r$. Therefore, for any $\epsilon>0$ such that $\epsilon=\cO(\sqrt{\mu}\cD_y)$, choosing $\zeta=\cO(\epsilon^2)$, the total gradient complexity can be bounded by $\cO\Big(\frac{\cL(\bx_0,\by_0)-F^*+{ L}\cD_y^2\log(\kappa)}{\epsilon^2}{ L {\kappa}^4}\Big)$.} %$\cO(\frac{L^2\kappa^4\log(\kappa)\cD_y^2}{\epsilon^2})$.

\rv{On the other hand, for the case $\mu>0$ is \textit{unknown}, treating the initial parameter estimates, $\tilde \mu$ and $\tilde l$, as $\cO(1)$ constants implies that $\cR=\cO(\max\{1/\mu,L\})$; hence, $\bar{L}=\cO(\max\{1/\mu,L\})$ and $\bar\mu=\cO(\min\{\mu,1/L\})$. Thus, $\bar \kappa =\cO(\max\{L^2,1/\mu^2\})$ in the worst case scenario. Note that when $\mu L=\cO(1)$, we have $\bar\kappa=\kappa$. That said, in practice one can choose $\tilde \mu$ and $\tilde l$ better than blindly setting them to some arbitrary constants. Indeed, for any $x',x''\in\dom g$ and $y',y''\in\dom h$, it holds that $\mu\leq \mu(z',z'')\triangleq\fprod{\grad_y f(x',y')-\grad_y f(x',y''), y'-y''}/\norm{y'-y''}^2$ and
{\small $$L(z',z'')\triangleq\max\{\fprod{\grad f(z')-\grad f(z''), z'-z''}/\norm{z'-z''}^2,~\norm{\grad f(z')-\grad f(z'')}^2/\fprod{\grad f(z')-\grad f(z''), z'-z''}\}\leq L,$$}%
where $z'=(x',y')$ and $z''=(x'',y'')$. Now consider setting $\tilde \mu=\sqrt{\mu(z',z'')L(z',z'')}$ and $\tilde l=\tilde\mu/\gamma$; thus, we have $\tilde \mu=c \sqrt{\mu L}$ and $\tilde l=c \sqrt{\mu L}/\gamma$ for some $c>0$. By sampling $z'$ and $z''$ more than once, one can get a better estimate of $\sqrt{\mu L}$. Whenever $c=\cO(1)$, i.e., $c$ does not depend on $L$ or $\mu$, this choice of $\tilde \mu$ and $\tilde l$ implies that $\rv{\cR \triangleq \max\{c\sqrt{\kappa},~\frac{\gamma}{c}\sqrt{\kappa},~1\}}=\cO(\sqrt{\kappa})$; hence, $\bar L  = \rv{\cR\;\tilde l}=\cO(L)$ and $\bar \mu\geq \rv{\tilde\mu/\cR}=\Omega(\mu)$, which implies that $\bar\kappa=\cO(\kappa)$.}
    % \xzrr{Assuming $\mu$ is unknown and $\tilde \mu,\tilde l$ are $\cO(1)$ constants, it follows that $\cR=\cO(\max\{1/\mu,L\})$, $\bar{L}=\cO(\max\{1/\mu,L\})$ and $\bar\mu=\cO(\min\{\mu,1/L\})$, thus $\bar \kappa =\cO(\max\{L^2,1/\mu^2\})$ and the total gradient complexity is $\tilde\cO(\frac{\max\{L,1/\mu\}^{10}\log(\max\{L,1/\mu\})\cD_y^2}{\epsilon^2})$.}
\end{remark}
\begin{proof}
\na{According to \agdap{}, %in \cref{alg:agda}, 
\rv{$\gamma_0\in (0,1)$} and $\tau_t<\sigma_t$ for all $t\in\N$; therefore, using \eqref{eq:complexity-bound-final}, we immediately get 
\begin{equation}
    \label{eq:complexity-intermediate-step}
    \min_{t\in\N}\{\norm{G^{\rv{\tau_t,\sigma_t}}(\bx_t,\by_t)}:\ 0\leq t\leq T\}\leq \epsilon,\quad\forall~T\in\N:\ \frac{\Gamma(\bx_0,\by_0,\gamma_0,\gamma,r)}{\gamma_0\sum_{t=0}^{T}\tau_t}\leq \epsilon^2.
\end{equation}}%
Moreover, \na{according to \cref{cor:bound-tn-kn} and initialization of \agdap{}, we have $%\xzf{\tilde{\mu}}/\gamma
\rv{\tilde l} \leq l_t
\leq {\bar{L}/\gamma}$, $%\xzf{\tilde{\mu}/\gamma}
\rv{\tilde l}\leq L_t\leq{\bar{L}/\gamma^r}$ and {$\gamma^r \bar\mu\leq \mu_t\leq \tilde\mu$} for all $t\in\N$; therefore,} our choice of $\tau_t,\sigma_t$ implies that $\gamma/\bar{L}\leq\sigma_t\leq
%\frac{\gamma}{\xzf{\tilde\mu}}
\rv{1/\tilde l}$ and
{\small
\begin{equation*}
    % \xzf{
    % \tau_t\geq \xzf{\frac{1-\gamma_0}{4+\frac{1}{\gamma_2} +60~
    % {(1-\frac{1}{\gamma^{r+1}\bar{\kappa}})(2-\frac{1}{\gamma^{r+1}\bar{\kappa}})}\frac{\bar{\kappa}^4}{\gamma^{8r}}}\cdot \frac{\gamma}{\bar{L}}},
    % }
    \rv{
    \tau_t\geq {\frac{1-\gamma_0}{4+\frac{1}{\gamma} +60~
    {(1-\frac{\gamma^{r+1}}{\bar{\kappa}})(2-\frac{\gamma^{r+1}}{\bar{\kappa}})}\frac{\bar{\kappa}^4}{\gamma^{8r}}}\cdot \frac{\gamma}{\bar{L}}};
    }
\end{equation*}}%
%and we also have $\sigma_t/\tau_t\geq 5$; 
hence, for all $t\in\N$, we have
%\begin{equation}\label{eq:bound-coeff-of-grad}
    $\gamma_0\tau_t \geq \frac{1}{125}{\gamma_0(1-\gamma_0){\gamma^{8r+1}}/(\bar{L}\bar{\kappa}^4})$.
%\end{equation}

\na{Then, the bound in \eqref{eq:complexity-intermediate-step} leads to the iteration complexity bound in~\eqref{eq:iter_complexity-true}. Moreover, since $1/(r\gamma^r)\geq 1$ for all $r\in\N+$ and $\gamma\in(0,1)$, we have $T_\epsilon\geq \log_{1/\gamma}(\rv{\cR})$. Therefore, \cref{thm:avg--backtrack-time} implies that \agdap{} checks the backtracking condition in~\eqref{eq:backtrack} at most $3T_\epsilon$ many times, which require $\cO(T_\epsilon)$ gradient calls, excluding the gradient calls requires to compute $\by_{\hat t_n}$ for $n=0,\ldots, N(T)$, where $N(T)$ denotes the number of \agdap{} iterations $t\leq T$ in which $L_t$ increases and $\mu_t$ decreases --according to~\cref{thm:avg--backtrack-time}, $N(T)=\cO(\log(\rv{\cR}))$ uniformly for all $T\in\N$. Recall that one needs to compute $\by_{\hat t_n}$ such that $\cL(\bx_{\hat t_n},\by_{\hat t_n})\geq F(\bx_{\hat t_n})-\zeta$ for some fixed $\bx_{\hat t_n}$. %Given the strong concavity modulus $\mu>0$, 
When $\mu$ is known, there are backtracking methods~\cite{calatroni2019backtracking,rebegoldi2022scaled} for solving these strongly concave maximization problems to compute $\zeta$-optimal solution $\by_{\hat t_n}$ without requiring to know $L$, and their gradient complexity can be bounded by $\cO(\sqrt{\kappa}\log(1/\zeta))$ \rv{--for the case $L$ and $\mu$ are both unknown, the same complexity bound holds for the smooth case ($h(\cdot)=0$) to compute an $\zeta$-optimal solution in terms of gradient norm~\cite{lan2023optimal}; \rv{that said, when $h(\cdot)\neq 0$, the complexity bound to compute $\by_{\hat t_n}$ such that $\cL(\bx_{\hat t_n},\by_{\hat t_n})\geq F(\bx_{\hat t_n})-\zeta$ for some fixed $\bx_{\hat t_n}$ is $\cO(\kappa\log(1/\zeta))$.}} 
Thus, the gradient complexity of \agdap{} can be bounded by $\cO(T_\epsilon+ N(T_\epsilon)\rv{\kappa^{(1+\varsigma)/2}}\log(1/\zeta))$ \rv{with $\varsigma=0$ for the case $\mu$ is known, and $\varsigma=1$ otherwise}, which leads to \eqref{eq:grad_complexity}. Finally, it can be trivially checked that for 
$
\epsilon=\cO\Big(
%\sqrt{L\cD_y^2}/\kappa
{\sqrt{\bar \mu}\cD_y}
\Big)
$, the last term in \eqref{eq:grad_complexity} is the dominant term.}
\end{proof}
\begin{remark}
    \na{Since $\zeta>0$ can be chosen small without causing significant computational burden, ignoring the $\zeta$ related terms, the dominant terms in \eqref{eq:iter_complexity-true} are $\frac{1}{\gamma^{{8}r+1}\gamma_0(1-\gamma_0)}\Big(\cL(\bx_0,\by_0)-F^*+\frac{\log(\cR){\bar L}\cD_y^2}{\rv{\gamma^{r}}\log(1/\gamma^r)}\Big)\frac{{\bar L\bar{\kappa}^4}}{\epsilon^2}$. If it is known that $\log(\cR)L\cD_y^2\gg \cL(\bx_0,\by_0)-F^*$, then the complexity bound is dominated by
    $\frac{\rv{10}}{\gamma_0(1-\gamma_0)}\frac{\log(\cR)\rv{\bar L}\cD_y^2}{\gamma^{\rv{10 r}}\log(1/\gamma^{\rv{10r}})}\frac{{\bar L\bar{\kappa}^4}}{\epsilon^2}$ since $r\geq 1$. Therefore, to minimize the bound, one can choose $\gamma_0=1/2$ and $\gamma=\frac{1}{\sqrt[\rv{10r}]{e}}$, which maximizes $\gamma^{10 r}\log(1/\gamma^{10 r})$, e.g., choosing $r=2$ leads to $\gamma\approx {0.95}$ with overall gradient complexity $\cO({\bar{L}^2\cD_y^2\bar{\kappa}^4\log(\cR)/\epsilon^2})$.}
\end{remark}
\begin{corollary}\label{cor:complexity-flag-false}
\na{Suppose \cref{ASPT:lipshiz gradient,aspt:bounded-Y,aspt:primal_lb} hold, and \emph{\msf{false}}. Given arbitrary $(\bx_0,\by_0)\in\dom g\times\dom h$, $\gamma_0,\gamma\in(0,1)$ and $r\in\N_+$, let $\{(\bx_t,\by_t)\}_{t\in\N}$ 
%denote the iterate sequence 
be generated by \cref{alg:agda}. Then, for any given $\epsilon>0$,} $(\bx_\epsilon,\by_\epsilon)\triangleq(\bx_{t^*_\epsilon},\by_{t^*_\epsilon})$ is $\epsilon$-stationary in the sense of Definition~\ref{def:eps-stationarity} with \na{$t^*_\epsilon\triangleq\argmin\{\norm{G^{\tau_t,\sigma_t}(\bx_t,\by_t)}:\ 0\leq t\leq T_\epsilon\}$} for \na{any} $T_\epsilon\in\N_+$ such that 
\begin{equation}
\label{eq:iter_complexity-false}
    \begin{aligned}
            & T_\epsilon = \Omega\left(
     \frac{\cL(\bx_0,\by_0) - F^* 
      + \frac{\log_{\frac{1}{\gamma}}(\cR)+1}{r\gamma^r}{\bar L}\cD_y^2+\frac{\rv{\tilde \mu}
      %\gamma^{\xzf{(2+\varsigma)r-(4+\varsigma)}}
      }{\gamma^2(1-\gamma^{{(2+\varsigma)r}})}{{\cR}^{(2+\varsigma)}}\cdot\min\{(1+\frac{2}{\gamma^{{2r}}}{\bar{\kappa}})\cD_y,~\overline{\cD}_y\}^2}{\gamma^{{8}r+1}\gamma_0(1-\gamma_0)}
    \cdot\frac{{\bar L\bar{\kappa}^4}}{\epsilon^2}\right),
    \end{aligned}
\end{equation}
\rv{where $\varsigma=0$ if $\mu$ is known, and $\varsigma=1$ otherwise.}
%\na{Then, $(\bx_\epsilon,\by_\epsilon)$ is $\epsilon$-stationary in the sense of Definition~\ref{def:eps-stationarity}.} %, i.e., $\norm{G(\bx_\epsilon,\by_\epsilon)} \leq \epsilon$.
Moreover, \agdap{} requires checking the backtracking condition in \eqref{eq:backtrack} at most $3 T_\epsilon$ times. Thus, for the case $\cD_y$ is known, i.e., $\overline{\cD}_y=\cD_y$, the total gradient complexity can be bounded by \rv{$\cO\Big(\Big(\cL(\bx_0,\by_0) - F^*+\Big[\log(\cR)\bar L+\tilde \mu{\cR}^{2+\varsigma}\Big]\cD_y^2\Big)\frac{{\bar L\bar{\kappa}^4}}{\epsilon^2}\Big)$}; on the other hand, if $\cD_y$ is unknown, i.e., $\overline{\cD}_y=+\infty$, then total gradient complexity can be bounded by \rv{$\cO\Big(\Big(\cL(\bx_0,\by_0) - F^*+\Big[\log(\cR)\bar L+\tilde \mu{\cR}^{2+\varsigma}\bar\kappa^2\Big]\cD_y^2\Big)\frac{{\bar L\bar{\kappa}^4}}{\epsilon^2}\Big)$}.
%$\cO\Big((\cL(\bx_0,\by_0) - F^*+\xzf{\tilde{\mu}\xzrr{\cR}^{2+\varsigma}\bar{\kappa}^2\cD_y^2})\frac{\qsr{\bar L\bar{\kappa}^4}}{\epsilon^2}\Big)$.
\end{corollary}
\begin{proof}
    \na{The proof follows from the same arguments used in~\cref{cor:complexity-flag-true}; therefore, it is omitted.}
\end{proof}
\begin{remark}\label{remark:mu-discussion-without-max-solver}
    Using the arguments in \cref{remark:mu-discussion-with-max-solver}, one can argue that when $\mu$ is \textit{known}, by setting $\tilde\mu=\underline{\mu}=\mu$ and $\tilde l=\mu/\gamma$, the total gradient complexity is bounded by 
    %$\cO(\frac{L^2\kappa^{7+\zeta}\cD_y^2}{\epsilon^2})$.
    \rv{$\cO\Big(\Big(\cL(\bx_0,\by_0) - F^*+L\cD_y^2\kappa\Big)\frac{{L{\kappa}^4}}{\epsilon^2}\Big)$ when $\cD_y$ is known; and by $\cO\Big(\Big(\cL(\bx_0,\by_0) - F^*+L\cD_y^2\kappa^3\Big)\frac{{L{\kappa}^4}}{\epsilon^2}\Big)$ when $\cD_y$ is unknown.}
    %assuming $\mu$ is unknown, the total gradient complexity is $\cO(\frac{\max\{L,1/\mu\}^{15+\zeta}\cD_y^2}{\epsilon^2})$.
\end{remark}
% \begin{remark}
%     \xzrr{With \texttt{max solver = true}, we achieve an improvement of $\kappa^{3+\zeta}$ when $\mu$ is known, and $\max\{L, 1/\mu\}^{5+\zeta}$ when $\mu$ is unknown.}
% \end{remark}
%\sa{Above two results imply that for $\epsilon\in(0,\cO(\kappa^2 L))$, i.e., this would imply $T=\Omega(\log^2(\kappa))$, the gradient complexity of \agdap{}~is $\sum_{t=0}^{T-1}I_t=\cO(\kappa^4L^2\epsilon^{-2})$.}
% \begin{corollary}\label{cor:ncmc}%[Complexity result for $\mu=0$]
% Suppose \cref{ASPT:lipshiz gradient,aspt:bounded-Y,aspt:primal_lb} hold, \sa{except we assume $\mu=0$.} Given any $\bx^0\in\dom g$, $\epsilon>0$. \agdap, \sa{applied to %\cref{q:wcmc} 
% $\min_{\bx}\max_{\by}\hat\cL(\bx,\by)$ such that $\hat\cL(\bx,\by)=\cL(\bx,\by)- \frac{\hat{\mu}}{2} \norm{\by-\hat{\by}}^2$ with $\hat{\mu}=\frac{\epsilon}{2\cD_y}$, is guaranteed to generate $(\bx_\epsilon,\by_\epsilon)$ such that
% \qsr{$\norm{ G^{\tau,\sigma}(\bx_\epsilon,\by_\epsilon)} \leq \epsilon$} within $\cO\Big(L^6\cD_y^4\epsilon^{-6}\Big)$ \agdap{} iterations, 
% %with \texttt{correction=True}, 
% %where each iteration 
% each requiring $%\bar{\ell}\triangleq\lceil
% \cO(\log_{\frac{1}{\gamma_1}}(L\cD_y/\epsilon))$ %\rceil$ 
% backtracking iterations.}
% \end{corollary}
% \begin{proof}
% See \cref{thm:wcmc-complexity}.    
% \end{proof}
%\xtodo{The complexity results and proofs for both false and true are good to me.}
\subsection{Overview of the proof and some remarks}
\label{sec:proof_overview}
%\nsa{Xuan, here explain why we get extra $\kappa^2$}
\na{For the sake of simplicity of the argument, suppose that $B_0\triangleq F(\bx_0)-F^*\gg \log(\kappa)L\cD_y^2$, \rv{the convexity modulus $\mu$ is known} and consider the smooth case, i.e., $g =0$. Then, according to \cref{cor:complexity-flag-true}, the upper complexity bound for \agdap{} is 
%$\cO\Big(\frac{F(\bx_0)-F^*+L\cD_y^2\log(\kappa)}{\epsilon^2}L\kappa^4\Big)$.
$\cO(L\kappa^4B_0\epsilon^{-2})$.} \sa{Here we are going to explain why our complexity bound  %$\cO(L^2\kappa^4\epsilon^{-2})$ for the (proximal) \gda{} with backtracking, i.e., \agdap, 
 has an extra %$L\kappa^2$
 \na{$\kappa^2$} factor when compared to $\cO(L\kappa^2B_0\epsilon^{-2})$ 
 %$\cO\Big(\frac{L\kappa^2(F(\bx_0)-F^*)+\kappa L^2\cD_y^2}{\epsilon^2}\Big)$ 
 complexity derived for \gda{} and \agda{} in~\cite{lin2020gradient,boct2020alternating}.} 
 %For the simplicity of argument, consider the smooth case, i.e., $g =0$.
 The analyses in both ~\cite{lin2020gradient,boct2020alternating} use constant step sizes, i.e., $(\tau_t,\sigma_t)=(\tau,\sigma)$, and are built upon establishing a descent result for the primal function $F(\bx)=\max_{\by\in\cY}f(\bx,\by)-h(\by)$. In particular, Lemma C.5 in \cite{lin2020gradient} implies that for all $t\geq 0$,
\begin{equation}\label{eq:decent-lemma-primal-example}
    \Theta(\tau)\|\grad F(\bx_t)\|^2\leq F(\bx_t) - F(\bx_{t+1}) + \Theta(\tau L^2)\delta_t,\quad \hbox{where}\quad \delta_t=\norm{\by_t-\by^*(\bx_t)}^2.
\end{equation}
\sa{Moreover, using the bound for $\delta_t$ in \cite[Equation (C.9)]{lin2020gradient}, one can get $\sum_{t=0}^T\delta_t=\cO(\kappa \cD_y^2)+ \Theta(\tau^2 \kappa^4) \sum_{t=0}^T\|\grad F(\bx_t)\|^2$; hence, after summing up \eqref{eq:decent-lemma-primal-example}, using the $\sum_{t=0}^T\delta_t$ bound in it, and then dividing both sides by $\tau$, one obtains
\begin{equation*}\label{eq:decent-lemma-primal-example2}
{\small
\begin{aligned}
           &\Theta(1-\tau^2 L^2\kappa^4)\sum_{t=0}^T \|\grad F(\bx_t)\|^2\leq \frac{F(\bx_0) - F^*}{\tau} + \sa{\cO(\kappa L^2 \cD_y^2)};
           %\\
           %&+ \sum_{t=0}^T\Theta(\tau_t^3 L^2\kappa^4)\|\grad F(\bx_t)\|^2 + \sa{\cO(\tau \kappa L^2 D^2)}
\end{aligned}}%
\end{equation*}
thus, setting $\tau=\cO(1/(L\kappa^2))$ implies that one can get $\sum_{t=0}^T\|\grad F(\bx_t)\|^2\leq \epsilon^2$ for \na{$T=\cO(L \kappa^2 B_0 \epsilon^{-2})$}.}
%\qstodo{should we use $F^*$ or $\underline{F}$? Xuan: $F^*$}
\sa{On the other hand, if one wants to extend this analysis to incorporate \textit{backtracking}, the %corresponding 
resulting backtracking condition %involves
requires checking inequalities involving evaluation of $F$ and $\grad F$, \na{i.e., since $g=0$, we have $F=\Phi$ and $\grad F=\grad \Phi$;} however, computing either of these quantities requires solving an optimization problem to exact optimality, which is not feasible in practice. Even if a relaxed condition %involving 
based on an inexact %minimizer
\na{maximizer} can be derived, this would still imply that the resulting algorithm be a double-loop one, i.e., combining nested outer and inner loops \na{--there would be $\cO(L\kappa^2\epsilon^{-2})$ outer iterations and for each outer iteration, one needs to call for a maximization solver to compute approximations for $F$ and $\grad F$ at a given point $\bx$.} The majority of double-loop methods iterates over the inner loop until a sufficient condition for some key inequality, needed for the convergence proof, to hold; however, since these sufficient conditions %checked is 
are often \textit{not necessary} for this key inequality to hold, the resulting algorithm may be wasting precious computing resources unnecessarily to obtain a certificate for the key inequality through verifying a sufficient condition.}  

\sa{A viable alternative is to design a backtracking condition involving the coupling function $f$ and $\grad f$, which are directly available through the first-order oracle.} 
%rather than the primal function $F$ and $\grad F$, which requires solving an optimization problem.
\na{In our analysis, we consider the monotonically increasing sequence $\{\hat{t}_n\}_{n\geq 0}\subset\N$ denoting all the discrete-time instances at which the quantity $L_t$ increases (formally stated in Definition~\ref{def:tn}) with the convention that $\hat t_0=0$. In \cref{cor:bound-tn-kn} in the appendix, we show that during the runtime of \agdap, there can %happen only 
can be at most $\bar N=\cO(\log_{\frac{1}{\gamma_1}}(\kappa))$ many such discrete time points; moreover, for any $n\in[\bar N]$,  
the definition of $\{\hat t_n\}$ implies that $\hat t_n<\hat t_{n+1}$ and $L_t = L_{\hat{t}_n}$ for all $t\in\mathbb{N}$ such that $t\in [\hat{t}_{n},\hat{t}_{n+1})$.} 

\na{In \cref{lemma:complicated-imply-short-true-forproof} provided in the appendix, using an alternative control on $\delta_t=\norm{\by_t-\by^*(\bx_t)}^2$, namely $\Delta_t$, which can be easily computed as in \cref{alg:agda}, we establish that}
%we obtain
\begin{equation}\label{eq:decent-lemma-f-example}
{\small
\begin{aligned}
        \na{\Theta\Big(\tau_t\Big(1-\frac{\tau_t}{\sigma_t}\Big)\Big)}\|G_x^{\tau_t} (\bx_t,\by_t)\|^2 +  %\Theta(\sigma_t^2\mu)
        \na{\sigma_t}{\|G_y^{\sigma_t}(\bx_t,\by_t)\|^2}\leq \na{\cL(\bx_t,\by_t) - \cL(\bx_{t+1},\by_{t+1}) + S_t},\quad\forall~t\in\N,
\end{aligned}}%
\end{equation}
\na{where $S_t=\Theta\Big(\frac{\Delta_t}{\sigma_t}\Big) + \Lambda_t$ if $t=\hat t_n$ for some $n\in[\bar N]$, and $S_t=\cO\Big(\frac{\Delta_t}{\sigma_t}+\frac{\Delta_{t-1}}{\sigma_{t-1}}+\frac{\tau_{t-1}^2}{\sigma_{t-1}}\norm{G_x^{\tau_{t-1}}(\bx_{t-1},\by_{t-1})}^2\Big)$ if $t\in(\hat t_n,~\hat t_{n+1})$ for some $n\in[\bar N]$, where $\Lambda_t$ given in \texttt{lines}~\ref{algeq:delta_def1} and \ref{algeq:delta_def2} of \cref{alg:agda} satisfies $\Lambda_{\hat t_n}=\cO(\sqrt{2\zeta/\mu}L_{\hat t_n}\cD_y)$ when \msf{True} and $\Lambda_{\hat t_n}=\cO(L_{\hat t_n}\cD_y^2)$ when \msf{False}. Thus, for any given $n\in[\bar N]$, summing \eqref{eq:decent-lemma-f-example} over $t\in [\hat t_n, \hat t_{n+1})$, we get }
\begin{equation}\label{eq:decent-lemma-f-example-2}
{\small
\begin{aligned}
        \MoveEqLeft\na{\sum_{t=\hat t_n}^{\hat t_{n+1}-1}\Big[\Theta\Big(\tau_t\Big(1-\frac{\tau_t}{\sigma_t}\Big)\Big)}\|G_x^{\tau_t} (\bx_t,\by_t)\|^2 +  %\Theta(\sigma_t^2\mu)
        \na{\sigma_t}{\|G_y^{\sigma_t}(\bx_t,\by_t)\|^2}\Big]\\
        &\leq \na{\cL(\bx_{\hat t_n},\by_{\hat t_n}) - \cL(\bx_{\hat t_{n+1}},\by_{\hat t_{n+1}}) + \Lambda_{\hat t_n}+\sum_{t=\hat t_n}^{\hat t_{n+1}-1}\Theta\Big(\frac{\Delta_t}{\sigma_t}\Big)+\cO(L\cD_y^2)},
\end{aligned}}%
\end{equation}
where the last $\cO(L\cD_y^2)$ term is to account for the fact that $\by_{\hat t_{n+1}}$ is computed using either \shortalgline{\ref{algeq:haty-max}} or \shortalgline{\ref{algeq:haty-y}} of \cref{alg:agda}, rather than \shortalgline{\ref{algeq:y-update}}.
\na{According to \agdap, when $t\in(\hat t_n,\hat t_{n+1})$ for some $n\in[\bar N]$, we have
${\Delta_t}=B_{t-1}{\Delta_{t-1}}+{C_{t-1}}\norm{G_x^{\tau_{t-1}}(\bx_{t-1},\by_{t-1})}^2$ for $B_{t-1}=1-\mu\sigma_{t-1}/2\in(0,1)$; this recursion implies that $\sum_{t=\hat t_n}^{\hat t_{n+1}-1}\Theta\Big(\frac{\Delta_t}{\sigma_t}\Big)=\cO\Big(\sum_{t=\hat t_n}^{\hat t_{n+1}-1}\frac{\tau_t^2}{\sigma_t}\kappa^4\norm{G_x^{\tau_t}(\bx_t,\by_t)}^2+\frac{L_{\hat t_n}}{\mu}\Delta_{\hat t_n}\Big)$. Thus, using the previous bound on $\sum_{t=\hat t_n}^{\hat t_{n+1}-1}\Theta\Big(\frac{\Delta_t}{\sigma_t}\Big)$ within \eqref{eq:decent-lemma-f-example-2}, 
for any $T\in\N$ such that $T\in(\hat t_{\bar n},\hat t_{\bar n+1})$ for some $\bar n\in[\bar N]$, we get
\begin{equation*}
{\small
    \begin{aligned}
         \sum_{t=0}^T\Big[\Theta\Big(\tau_t\Big(1-\kappa^4\frac{\tau_t}{\sigma_t}\Big)\Big)\|G_x^{\tau_t}(\bx_t,\by_t)\|^2 +  \sigma_t\|G_y^{\sigma_t}(\bx_t,\by_t)\|^2\Big]
        \leq F(\bx_0) - F^* + \cO\Big((\bar n+1)L\cD_y^2+\sum_{n=0}^{\bar n}\Big(\Lambda_{\hat t_n}+\frac{L_{\hat t_n}^2}{\mu}\Delta_{\hat t_n}\Big)\Big),
    \end{aligned}}%
\end{equation*}
for details of the upper bound, see~\eqref{eq:complexity-bound} and~\eqref{eq:first-bound-Xi-K} in the appendix.
Note that $\sum_{n=0}^{\bar{n}}L_{\hat t_n}\leq\sum_{n=0}^{\bar N}\frac{\mu}{\gamma_2\gamma_1^n}=\cO(L)$ and $\sum_{n=0}^{\bar{n}}L^2_{\hat t_n}\leq\sum_{n=0}^{\bar N}\frac{\mu^2}{\gamma_2^2\gamma_1^{2n}}=\cO(L^2)$.
For \msf{True}, $\Lambda_{\hat t_n}=\cO(\sqrt{\frac{2\zeta}{\mu}}L_{\hat t_n}\cD_y)$ and $\Delta_{\hat t_n}=\frac{2\zeta}{\mu}$; while for \msf{False}, $\Lambda_{\hat t_n}=\cO(L_{\hat t_n}\cD_y^2)$ and $\Delta_{\hat t_n}\leq \cD_y^2$ when $\cD_y$ is known. Therefore, when \msf{True}, $\sum_{n=0}^{\bar n}\Big(\Lambda_{\hat t_n}+\frac{L_{\hat t_n}^2}{\mu}\Delta_{\hat t_n}\Big)=\cO\Big(\sqrt{\frac{2\zeta}{\mu}}\cD_y\sum_{n=0}^{\bar{n}}L_{\hat t_n}\Big)+\cO\Big(\frac{2\zeta}{\mu^2}\sum_{n=0}^{\bar n}L_{\hat t_n}^2\Big)$; consequently, setting $\tau_t = \cO(\frac{\sigma_t}{\kappa^4})$ and $\zeta=\cO(\epsilon^2)$ implies that for 
$T=\Omega\Big(L\kappa^4\Big(F(\bx_0)-F^*+\log(\kappa)L\cD_y^2\Big)\epsilon^{-2}\Big)$, we have $\sum_{t=0}^T \norm{G^{\tau_t,\sigma_t}(\bx_t,\by_t)}^2\leq \epsilon^2$ 
--since $\sigma_t=\Theta(1/L)$ by construction. On the other hand, for \msf{False}, $\sum_{n=0}^{\bar n}\Big(\Lambda_{\hat t_n}+\frac{L_{\hat t_n}^2}{\mu}\Delta_{\hat t_n}\Big)=\cO(L\kappa\cD_y^2)$; consequently, setting $\tau_t = \cO(\frac{\sigma_t}{\kappa^4})$ implies that for 
$T=\Omega\Big(L\kappa^4\Big(F(\bx_0)-F^*+\kappa L\cD_y^2\Big)\epsilon^{-2}\Big)$, we have $\sum_{t=0}^T \norm{G^{\tau_t,\sigma_t}(\bx_t,\by_t)}^2\leq \epsilon^2$.}

\section{Numerical experiments}
The experiments are conducted on a PC with 3.6 GHz Intel Core i7 CPU and NVIDIA
RTX2070 GPU. \sa{We consider two problem classes: %quadratic problems 
WCSC minimax problems with synthetic data, and Distributionally Robust Optimization~(DRO) with real data. 
We tested \agdap{} against two adaptive algorithms, \tiada~\cite{li2022tiada} and \sgdab~\cite{xu2024stochastic}, and we used \gda~\cite{lin2020gradient}, \agda~\cite{boct2020alternating} and \smagda~\cite{yang2022faster} as benchmark. On our test problems \tiada{} has constantly performed better than \texttt{NeAda} --similar to the comparison results in \cite{li2022tiada}; this is why we only report results for \tiada{} and \sgdab{}, which are the main competitor of \agdap{} as parameter agnostic algorithms for WCSC minimax problems.} {The code is available at \href{https://github.com/XuanZhangg/AGDA-Plus.git}{https://github.com/XuanZhangg/AGDA-Plus.git}.}
%distributionally robust optimization, and fair classification.  
\paragraph{\sa{Algorithm parameter settings.}} %For algorithms in other paper, to make 
To have a fair comparison, we test %them 
all the methods using step sizes with theoretical guarantees, and \rv{we assumed that the concavity modulus $\mu>0$ is known for all the experiments.} Indeed,
according to \cite{lin2020gradient,boct2020alternating}, 
%letting $\tau$ and $\sigma$ be the primal and the dual step sizes, 
we set $\tau=\Theta(\frac{1}{L\kappa^2})$, $\sigma=\Theta(\frac{1}{L})$ for both \gda{} and \agda. For \smagda, we set their parameters according to~\cite[Theorem 4.1]{yang2022faster}, \na{i.e., $\tau=\frac{1}{3L}$, $\sigma=\frac{1}{144L}$, $p=2L$, and $\beta=\frac{\sigma\mu}{1200}$.} For \tiada, we set $\alpha=0.6$ and $\beta = 0.4$ as recommended in~\cite{li2022tiada}, %the paper,
and tune the initial stepsizes $\tau_0$ and $\sigma_0$ from $\{100, 10, 1, 0.1, 0.01\}$ and set $v^x_0=v^y_0=1$ --see~\eqref{eq:tiada_step} for \tiada{} step size rule. For \texttt{SGDA-B}, we set the max iteration budget per contraction, $T$, to $4000$ for DRO problem in~\eqref{eq:dro-problmm} and $10,000$ for other problems, i.e., the total gradient complexity would be $T\log_{\frac{1}{\gamma}}(\kappa)$ where $\gamma\in(0,1)$ is the contraction parameter --see \cref{sec:goals} for the details of \sgdab{}, and {we set the contraction parameter as $\gamma=0.95$ and $\gamma=0.9$ for the DRO and synthetic problems, respectively.} For \agdap, we set $\gamma_0=10^{-3}$, $\gamma=0.95$, and $r=2$, \na{i.e., $\gamma_2=\gamma=0.95$ and $\gamma_1=\gamma^2=0.9025$.} 
\subsection{\na{Synthetic data: Quadratic WCSC problems}}
\label{sec:random-data}
%We first test on the regularized bilinear SP problem defined as
\sa{Consider the quadratic WCSC minimax problem with of the following form:\
\begin{equation} \label{eq:bilinear-problem}
{\small
    \cL(\bx,\by) = \na{\frac{1}{2}}\bx^\top Q\bx
    +
    \bx^\top A \by - \frac{\mu_y}{2}\|\by\|^2,}%
\end{equation}
where 
%$A$ and $Q$ are matrices in 
$A\in\reals^{m\times n}$, $Q\in\mathbb{R}^{m\times m}$,  and $\mu_y>0$.
This class of problems include \emph{Polyak-Lojasiewicz game} \cite{chen2022faster,zhang2023jointly}, \emph{image processing}~\cite{chambolle2011first}, and \emph{robust regression}~\cite{xu2008robust}.}

\sa{In our tests, we set $m=n=30$, $\mu_y=1$ and randomly generate $A$ and $Q$ such that $A = V\Lambda_A V^{-1}$ and $Q = V\Lambda_Q V^{-1}$, where $V\in\mathbb{R}^{30\times30}$ is an %randomly generated 
orthogonal matrix, 
%in $\mathbb{R}^{30\times30}$ and 
$\Lambda_A$ and $\Lambda_Q$ are diagonal matrices.
%with equal differences 
%and normalized it according to 
We set $\Lambda_Q = \frac{\Lambda^0_Q}{\|\Lambda^0_Q\|_2}\cdot L$ for $L\in\{5,10,20\}$, where $\Lambda^0_Q$ is a random diagonal matrix with diagonal elements being 
sampled uniformly at random from the interval $[-1, 1]$ \na{and $\norm{\Lambda^0_Q}_2$ denotes the spectral norm of $\Lambda^0_Q$; therefore, $\norm{Q}_2=L$.}} %Additionally, 
We choose $\Lambda_A$ such that $\Lambda_Q + \frac{1}{\mu_y}\Lambda_A^2 \succeq 0$ and $\norm{\Lambda_Q}_2\geq\norm{\Lambda_A}_2$. 
Given  $\bx$, we can compute $\by^*(\bx) = \frac{1}{\mu_y}A^\top \bx$; consequently, the primal function $F(\bx) = \na{\frac{1}{2}}\bx^\top (Q+\frac{1}{\mu_y} A A^\top) \bx$ and it is lower bounded by zero. Thus, $\cL$ is weakly convex-strongly concave such that it is $L$-smooth in $(\bx,\by)$. 
%We test all algorithms with  $\mu_y=1$ and $L\in\{2,5,10\}$  and set $\gamma_1=0.9,\gamma_2=0.95,\gamma_0=10^{-3}$ for \agdap.
% In~\cref{fig:Q-kappa-5} and \cref{fig:Q-kappa-10}, 
In~\cref{fig:Q},
we plot $\|\grad \cL(\bx_t,\by_t)\|^2/\|\grad \cL(\bx_0,\by_0)\|^2$ against the number of gradient calls, $t\in\N$, in the
x-axis. {The results are obtained for 10 simulations, initialized from i.i.d. random points $(\bx_0,\by_0)$ such that each entry of these vectors follow a uniform distribution taking values between $-100$ and $100$}. We run each algorithm until 10,000 iterations or it computes $(\bx_{t^*},\by_{t^*})$ such that $\norm{\grad\cL(\bx_{t^*},\by_{t^*})}^2/\norm{\grad\cL(\bx_{0},\by_{0})}^2\leq 1e-6$, whichever comes first. We observe that both backtracking methods, i.e., \sgdab{} and \agdap{}, outperform the others on this test, possibly because %and are not sensitive to the condition number. That said, \agdap{} 
they can %efficiently search the local Lipshitz constant
\na{adaptively determine the largest step sizes %leading to 
achieving convergence to stationarity,} and %achieve
exhibit stable and fast performance when compared to AdaGrad-based method \tiada{} and others using constant stepsize based on the global Lipschitz constant $L$, i.e., \smagda{}, \agda{}, and \gda{}. \na{In our tests, we observe that along the trajectory of \sgdab{}, the average local Lipschitz constants do not oscillate significantly and the average value is relatively high compared to the average local Lipschitz constant observed along the \agdap{} trajectory; therefore, checking once every $T$ gradient descent ascent steps helps \sgdab{} adopt relatively larger step size that is constant for $T$ iterations compared to step size values adopted by \agdap{}, and that is why \sgdab{} exhibits a faster convergence behavior on this test. To verify this observation, in the next section, we test both methods on quadratic WCSC minimax problems with sinusoidal perturbations causing significant oscillation in local Lipschitz constant values.}
\begin{figure*}[h]
     % \centering
     % \begin{subfigure}[b]{0.24\textwidth}
     %     \centering
     %    \includegraphics[width=\textwidth]{figure/exp:deterministic/Q_stdx_0_stdy_0_muy_1_kappa_2_b_1_z.pdf}
     %     \caption{$\kappa=2,\lambda=0$}
     %     \label{fig:Q-kappa-2}
     % \end{subfigure}
\begin{subfigure}[b]{0.33\textwidth}
         \centering
\includegraphics[width = \textwidth]{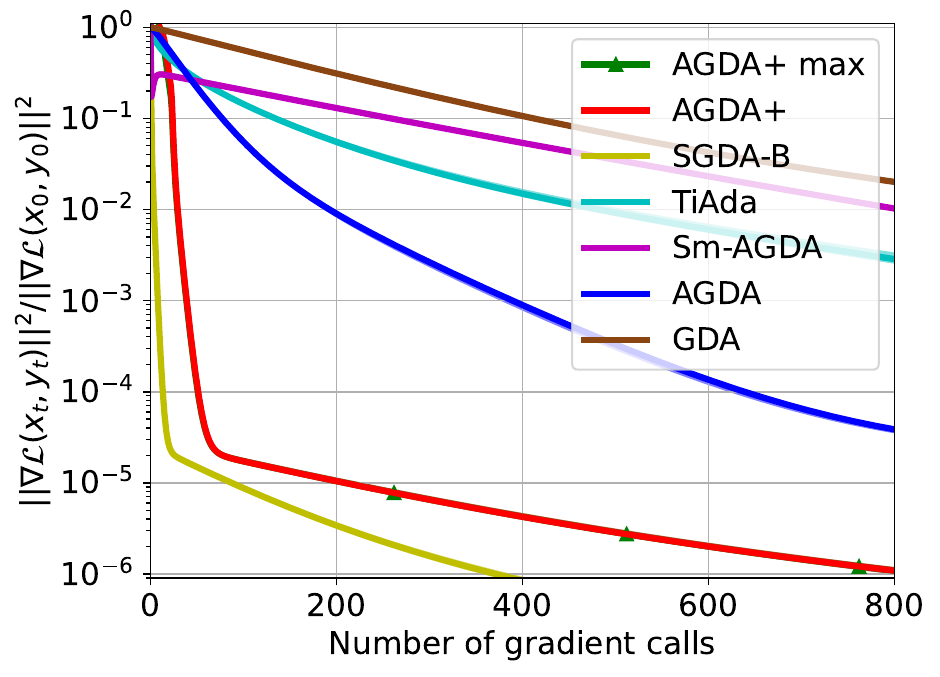}
         \caption{$\kappa=5$}
         \label{fig:Q-kappa-5}
     \end{subfigure}
     \begin{subfigure}[b]{0.319\textwidth}
         \centering
\includegraphics[width = \textwidth]{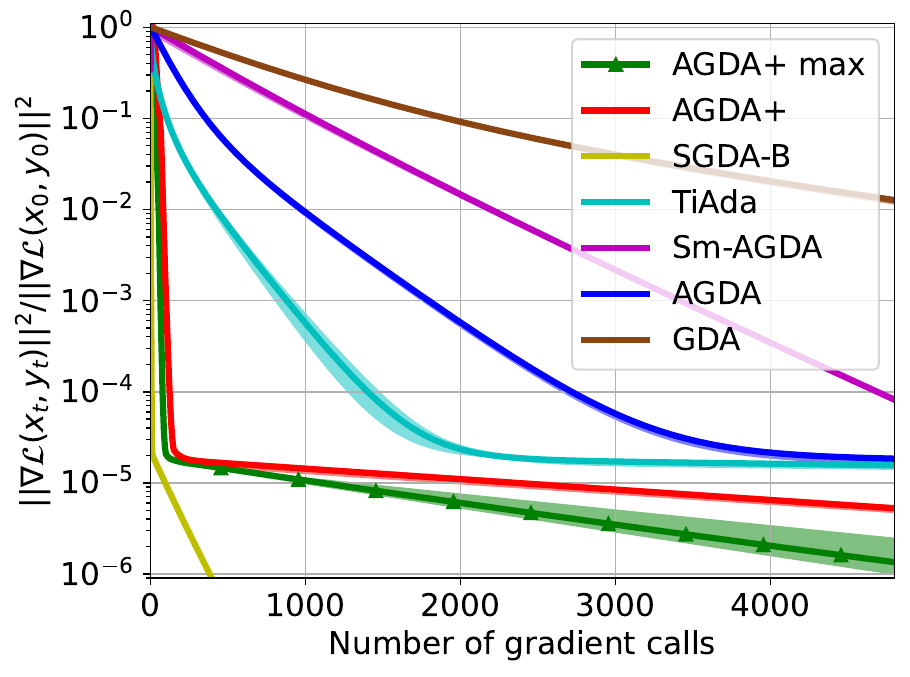}
         \caption{$\kappa=10$}
         \label{fig:Q-kappa-10}
     \end{subfigure}
     \begin{subfigure}[b]{0.338\textwidth}
         \centering
\includegraphics[width = \textwidth]{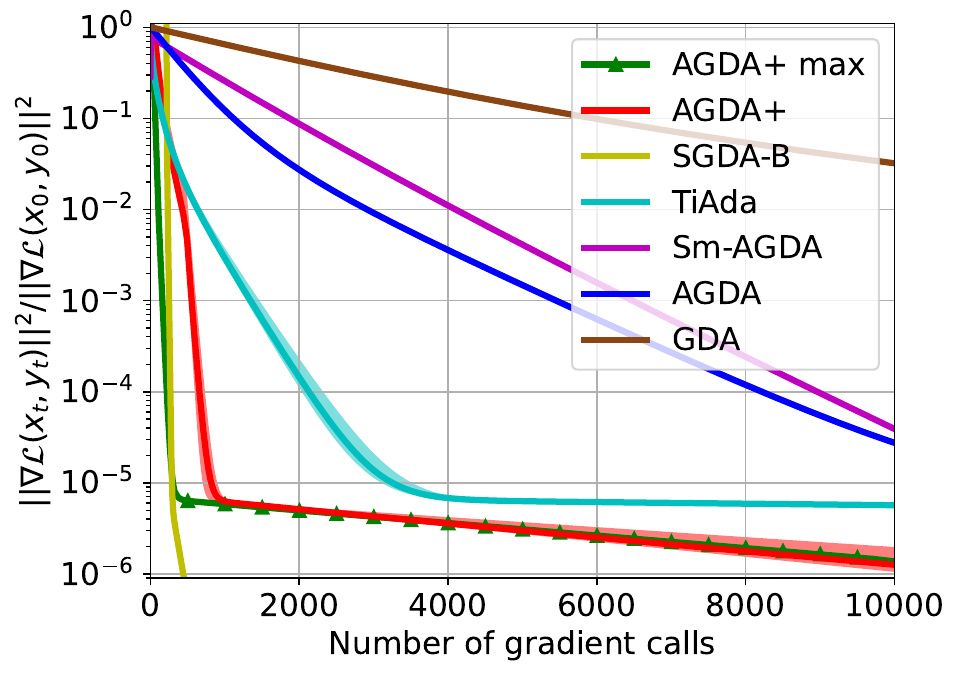}
         \caption{$\kappa=20$}
         \label{fig:Q-kappa-20}
     \end{subfigure}
%      \begin{subfigure}[b]{0.24\textwidth}
%          \centering
% \includegraphics[width = \textwidth]{figure/exp:deterministic/Q_stdx_0_stdy_0_muy_1_kappa_50_b_1_z.pdf}
%          \caption{$\kappa=50$}
%          \label{fig:Q-kappa50-sto}
%      \end{subfigure}
     \\
     \begin{subfigure}[b]{0.33\textwidth}
         \centering
\includegraphics[width = \textwidth]{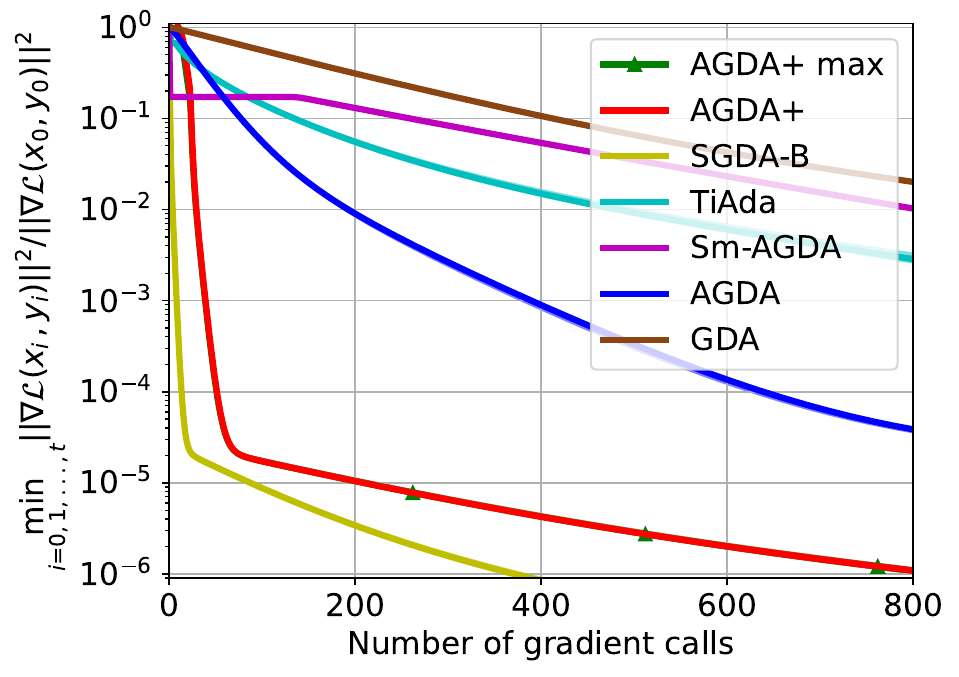}
         \caption{$\kappa=5$}
         \label{fig:Q-kappa-5-last}
     \end{subfigure}
     \begin{subfigure}[b]{0.32\textwidth}
         \centering
\includegraphics[width = \textwidth]{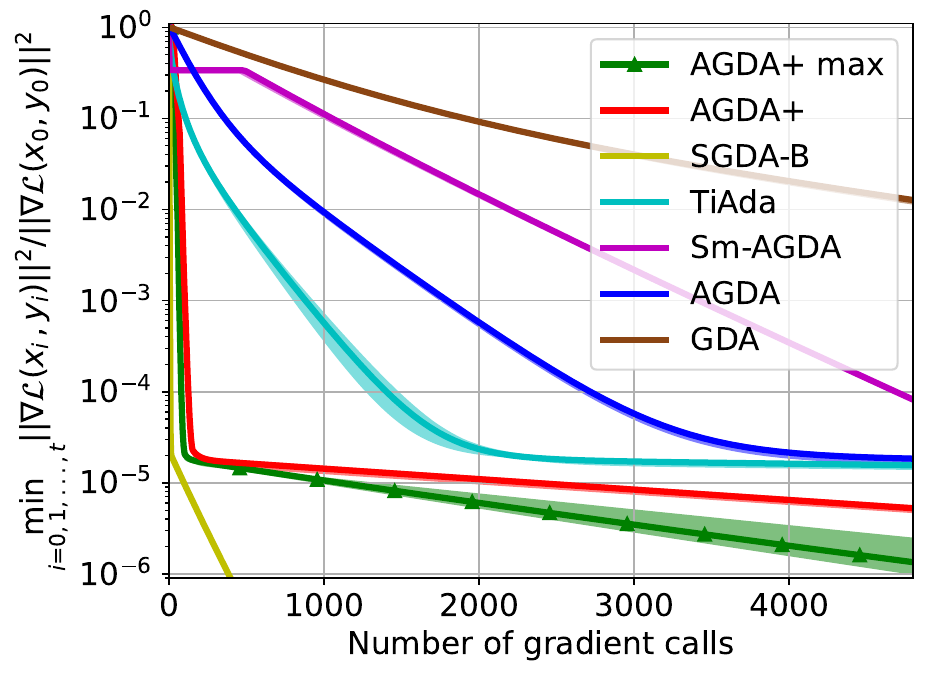}
         \caption{$\kappa=10$}
         \label{fig:Q-kappa-10-last}
     \end{subfigure}
     \begin{subfigure}[b]{0.338\textwidth}
         \centering
\includegraphics[width = \textwidth]{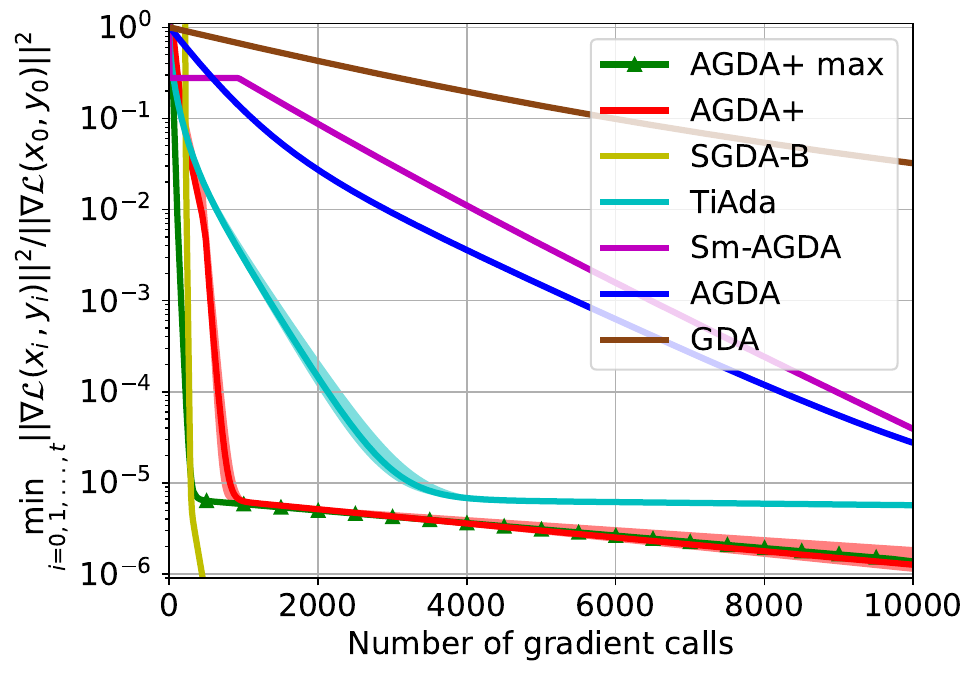}
         \caption{$\kappa=20$}
         \label{fig:Q-kappa-20-last}
     \end{subfigure}
%      \begin{subfigure}[b]{0.24\textwidth}
%          \centering
% \includegraphics[width = \textwidth]{figure/exp:deterministic/Q_stdx_0_stdy_0_muy_1_kappa_50_b_1_z_last.pdf}
%          \caption{$\kappa=50$}
%          \label{fig:Q-kappa50-sto}
%      \end{subfigure}
\caption{
Comparison of \agdap{} against \texttt{SGDA-B}, \tiada,  \smagda, \agda, and \gda,  and~on the quadratic WCSC minimax problem stated in~\eqref{eq:bilinear-problem} using $10$ simulation runs. Solid lines indicate the average over 10 runs and shaded regions around the solid lines denote the range statistic. The plots in the top row display the last-iterate performance while the plots in the bottom row display the best-iterate performance.
}
\label{fig:Q}
\centering
\end{figure*}
\begin{figure*}[h]
    \centering
    \includegraphics[width=0.32\textwidth]{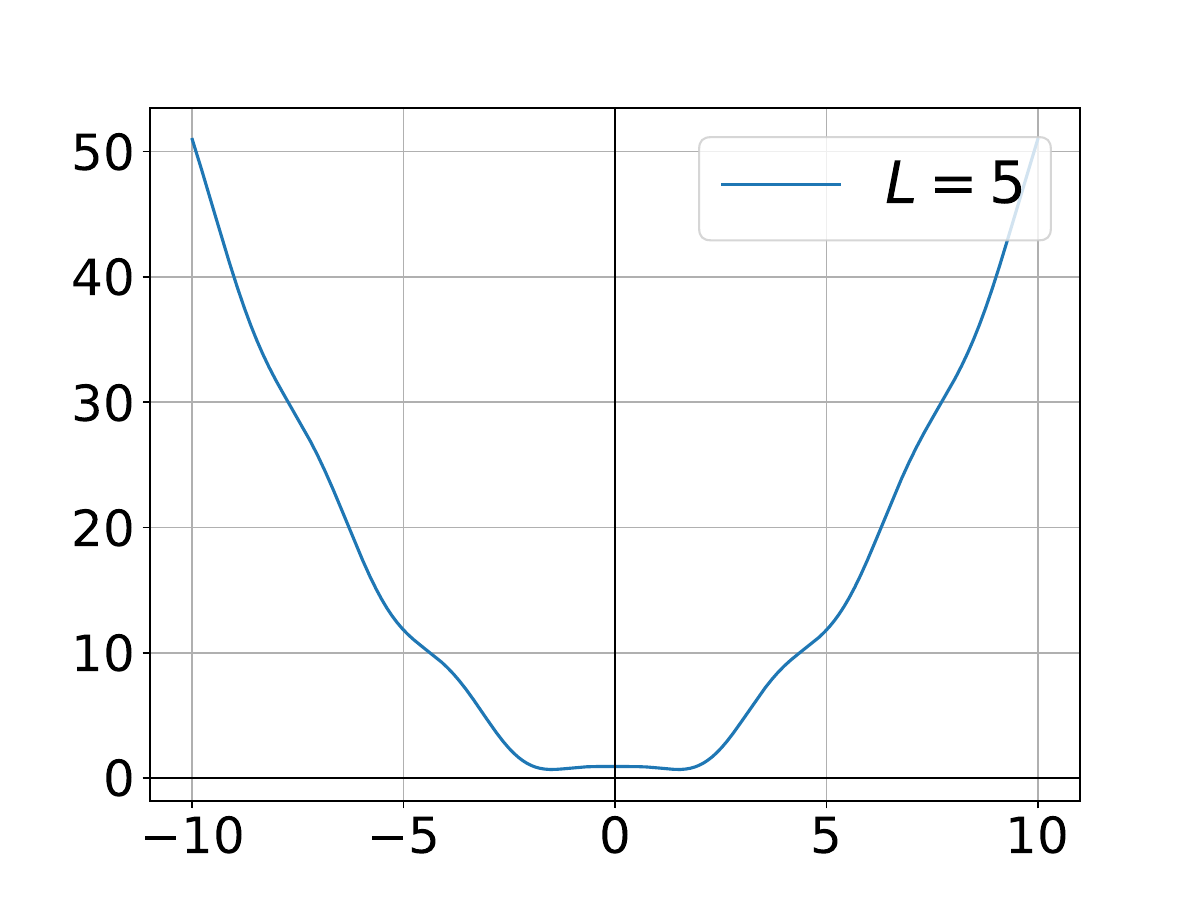}
    \includegraphics[width=0.32\textwidth]{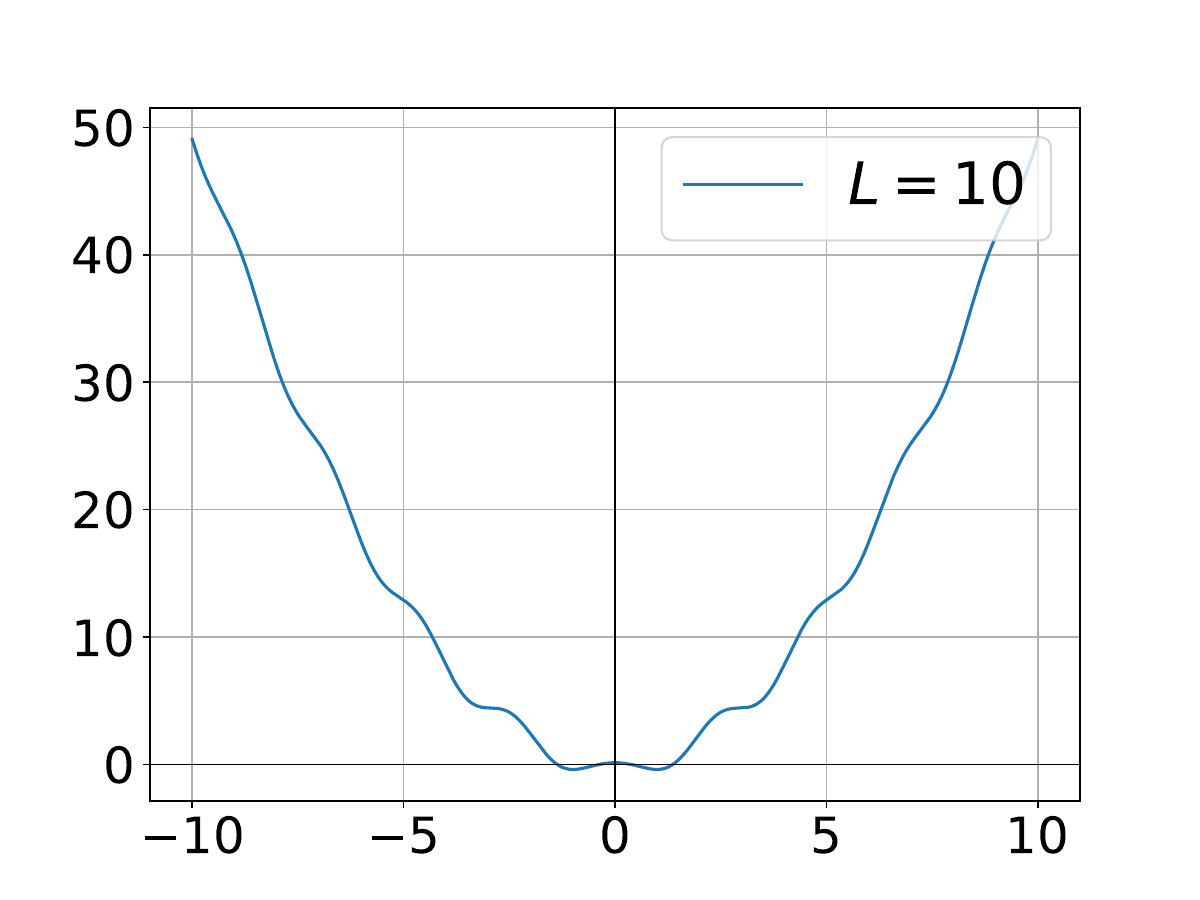}
        \includegraphics[width=0.32\textwidth]{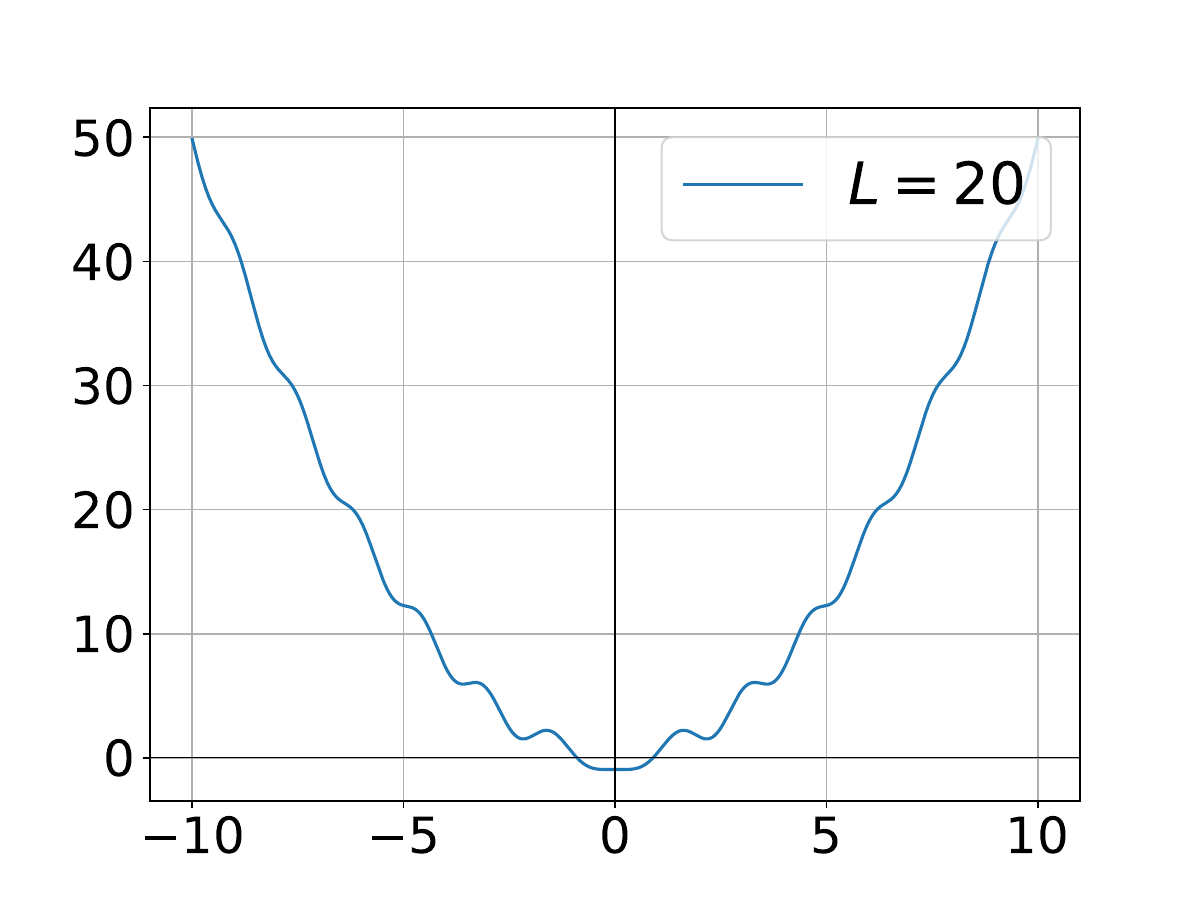}
    \caption{ $s(x)=x^2+\sin(\sqrt{L-1}\cdot\sqrt{x^2+1})$}
    \label{fig:sin-toy}
    \vspace*{-3mm}
\end{figure*}
\subsection{\na{Synthetic data: Quadratic WCSC problems with sinusoidal perturbation}}
\label{sec:sin-random-data}
We next test the same set of methods on a quadratic WCSC problem with sinusoidal perturbation:
\begin{equation} \label{eq:sin-bilinear-problem}
    \cL(\bx,\by) = \sin(\sqrt{L-1}\cdot{\sqrt{\|\bx\|^2+1})}+\frac{1}{2}\bx^\top Q\bx
    +
    \bx^\top A \by - \frac{\mu_y}{2}\|\by\|^2,
\end{equation}
where 
%$A$ and $Q$ are matrices in 
$A\in\reals^{m\times n}$, $Q\in\mathbb{R}^{m\times m}$,  and $\mu_y>0$. Indeed, for this class of problems, gradients exhibit high fluctuations, which pose a significant challenge to optimization methods. We have included a toy example in \cref{fig:sin-toy} to illustrate this. Let $s(\bx)=\sin(\sqrt{L-1}\cdot\sqrt{\norm{\bx}^2+1})$ and $q(\bx)=\frac{1}{2}\bx^\top Q\bx$; hence, using the notation from \eqref{eq:main-problem}, $\cL(\bx,\by)=f(\bx,\by)=s(\bx)+q(\bx)+\frac{1}{2}\bx^\top A\by-\frac{\mu_y}{2}\norm{\by}^2$. It can be easily checked that $\grad s(\cdot)$ is $(L-1)$-Lipschitz. In our tests, we set $m=n=30$, $\mu_y=1$, and pick $L\in\{5,10,20\}$. Furthermore, the matrix $A$ and $Q$ are generated %by the almost the same way as that for problem (\ref{eq:bilinear-problem}) 
\na{as in~\cref{sec:random-data} except for $\Lambda_Q$ selection; indeed, we set $\Lambda_Q = \frac{\Lambda^0_Q}{\|\Lambda^0_Q\|_2}$, and choose $\Lambda_A$ such that $\Lambda_Q + \frac{1}{\mu_y}\Lambda_A^2 \succeq 0$ and $\norm{\Lambda_Q}_2\geq\norm{\Lambda_A}_2$. Thus, $\cL=f$ is $L$-smooth since $s(\cdot)$ and $q(\cdot)$ are smooth with constants $L-1$ and $1$, respectively, and $\norm{A}_2\leq \norm{Q}_2=1$. 
Since $\by^*(\bx) = \frac{1}{\mu_y}A^\top \bx$ as in \cref{sec:random-data}, the primal function $F(\bx) = s(\bx)+\frac{1}{2}\bx^\top (Q+\frac{1}{\mu_y} A A^\top) \bx$ and $F(\cdot)\geq -1$ as $s(\bx)\geq -1$ for all $\bx$.}  
% randomly generate $A$ and $Q$such that $A = V\Lambda_A V^{-1}$ and $Q = V\Lambda_Q V^{-1}$, where $V\in\mathbb{R}^{30\times30}$ is an %randomly generated 
% orthogonal matrix, 
% %in $\mathbb{R}^{30\times30}$ and 
% $\Lambda_A$ and $\Lambda_Q$ are diagonal matrices.
% We set $\Lambda_Q = \frac{\Lambda^0_Q}{\|\Lambda^0_Q\|_2}$, where $\Lambda^0_Q$ is a random diagonal matrix with diagonal elements being 
% sampled uniformly at random from the interval $[-1, 1]$.
% We choose $\Lambda_A$ such that $\Lambda_Q + \frac{1}{\mu_y}\Lambda_A^2 \succeq 0$ and $\norm{\Lambda_Q}_2\geq\norm{\Lambda_A}_2$. Those conditions imply that $\cL$ is 
% WSCS and $L$-smooth. 
% Moreover, given  $\bx$, we can compute $\by^*(\bx) = \frac{1}{\mu_y}A^\top \bx$. Consequently, the primal function $F(\bx) =\frac{1}{2} \bx^\top (Q+\frac{1}{\mu_y} A A^\top) \bx$ and it is lower bounded by zero. 
In~\cref{fig:sin-Q},
we plot $\|\grad \cL(\bx_t,\by_t)\|^2/\|\grad \cL(\bx_0,\by_0)\|^2$ against the number of gradient calls, $t\in\N$, in the
x-axis. {The results are obtained for 10 simulations, initialized from i.i.d. random points $(\bx_0,\by_0)$ such that each entry of these vectors follow a uniform distribution taking values between $-100$ and $100$}. We run each algorithm until 10,000 iterations or it computes $(\bx_{t^*},\by_{t^*})$ such that $\norm{\grad\cL(\bx_{t^*},\by_{t^*})}^2/\norm{\grad\cL(\bx_{0},\by_{0})}^2\leq 1e-7$, whichever comes first. We observe that \agdap{} outperforms the others on this test, and 
it can efficiently search the local Lipshitz constant and achieve stable and fast performance.

\begin{figure*}[h]
     % \centering
     % \begin{subfigure}[b]{0.24\textwidth}
     %     \centering
     %    \includegraphics[width=\textwidth]{figure/exp:deterministic/Q_stdx_0_stdy_0_muy_1_kappa_2_b_1_z.pdf}
     %     \caption{$\kappa=2,\lambda=0$}
     %     \label{fig:Q-kappa-2}
     % \end{subfigure}
\begin{subfigure}[b]{0.33\textwidth}
         \centering
\includegraphics[width = \textwidth]{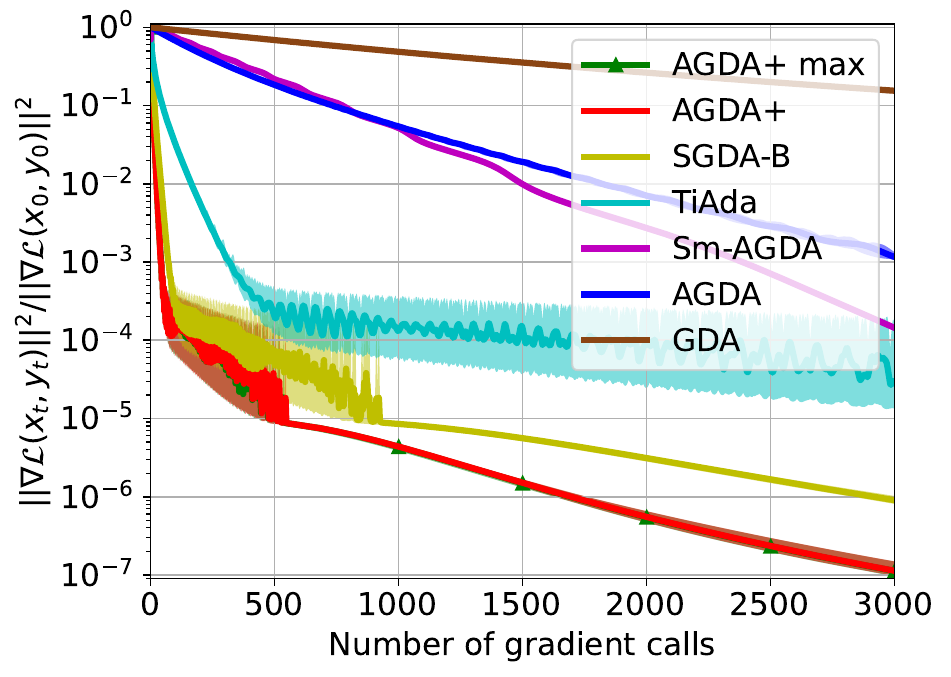}
         \caption{$\kappa=5$}
         \label{fig:sin-Q-kappa-5}
     \end{subfigure}
     \begin{subfigure}[b]{0.32\textwidth}
         \centering
\includegraphics[width = \textwidth]{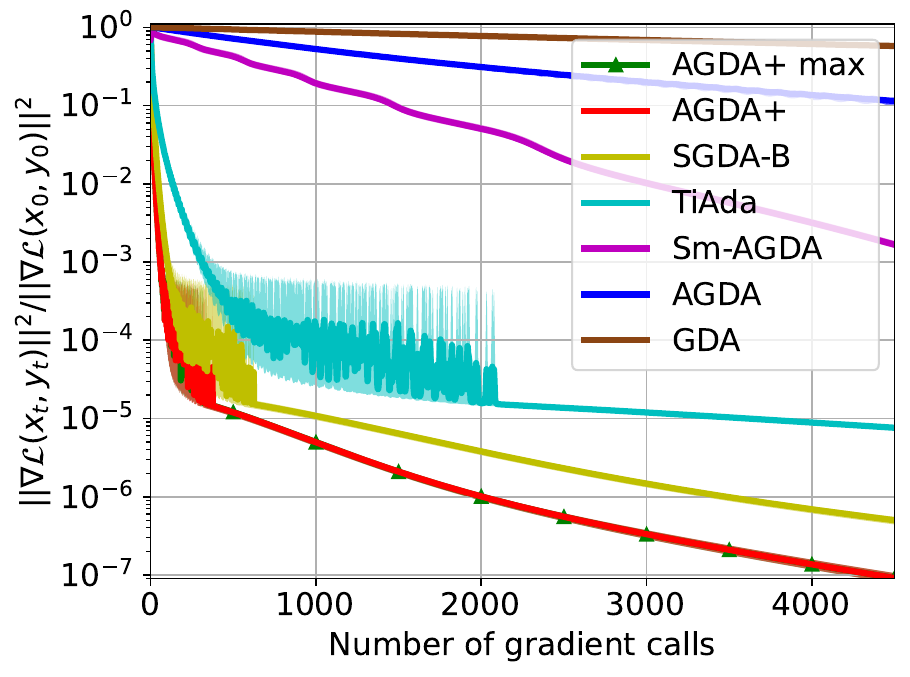}
         \caption{$\kappa=10$}
         \label{fig:sin-Q-kappa-10}
     \end{subfigure}
     \begin{subfigure}[b]{0.34\textwidth}
         \centering
\includegraphics[width = \textwidth]{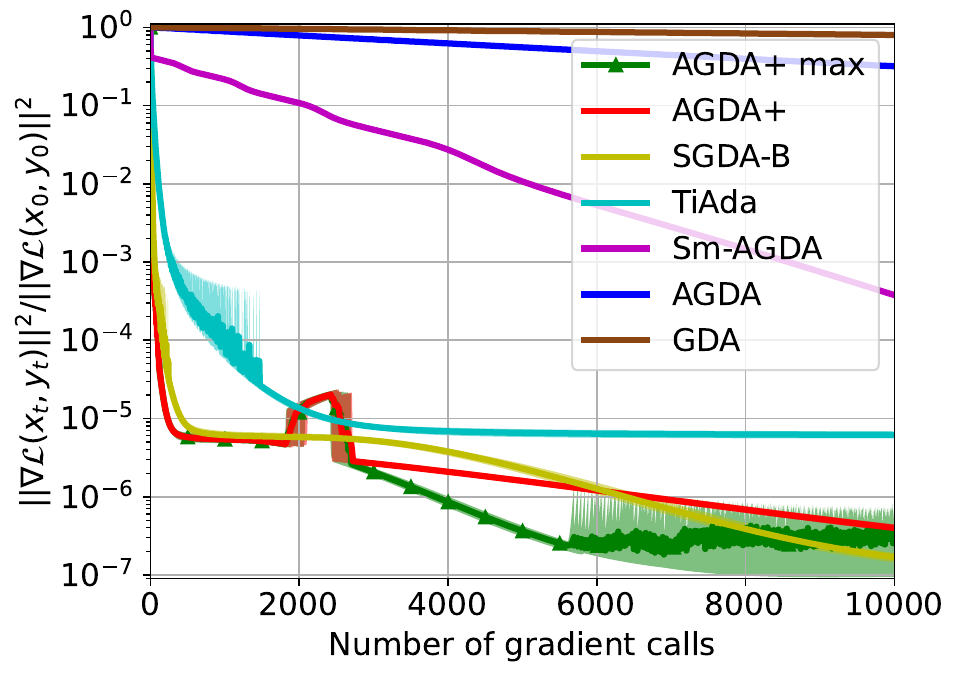}
         \caption{$\kappa=20$}
         \label{fig:sin-Q-kappa-20}
     \end{subfigure}
%      \begin{subfigure}[b]{0.24\textwidth}
%          \centering
% \includegraphics[width = \textwidth]{figure/exp:deterministic/sin_Q_stdx_0_stdy_0_muy_1_kappa_50_b_1_z.pdf}
%          \caption{$\kappa=50$}
%          \label{fig:sin-Q-kappa50-sto}
%      \end{subfigure}
     \\
     \begin{subfigure}[b]{0.33\textwidth}
         \centering
\includegraphics[width = \textwidth]{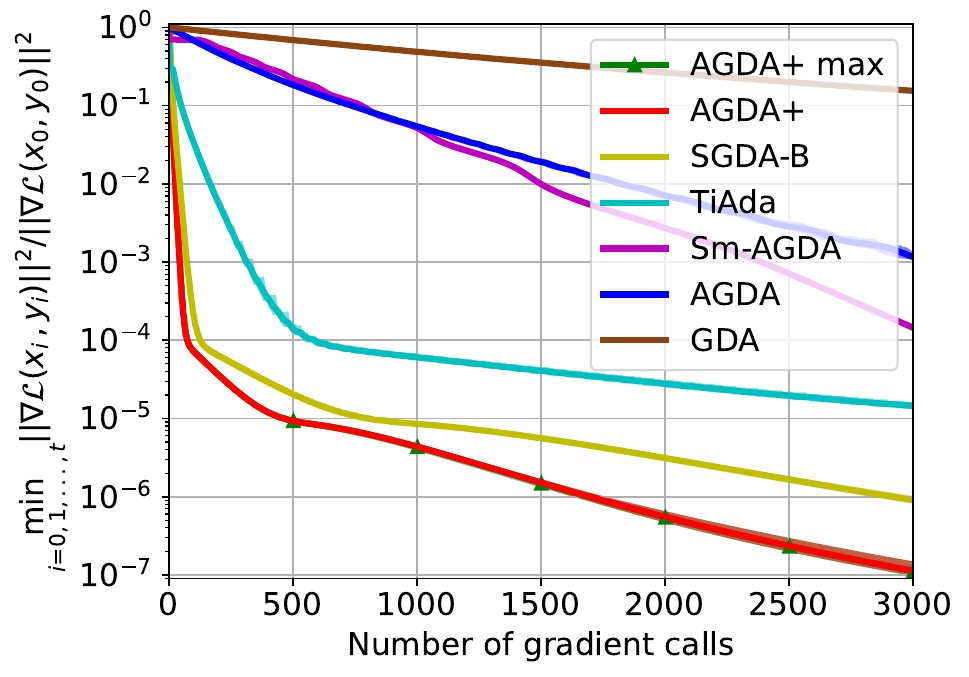}
         \caption{$\kappa=5$}
         \label{fig:sin-Q-kappa-5-last}
     \end{subfigure}
     \begin{subfigure}[b]{0.32\textwidth}
         \centering
\includegraphics[width = \textwidth]{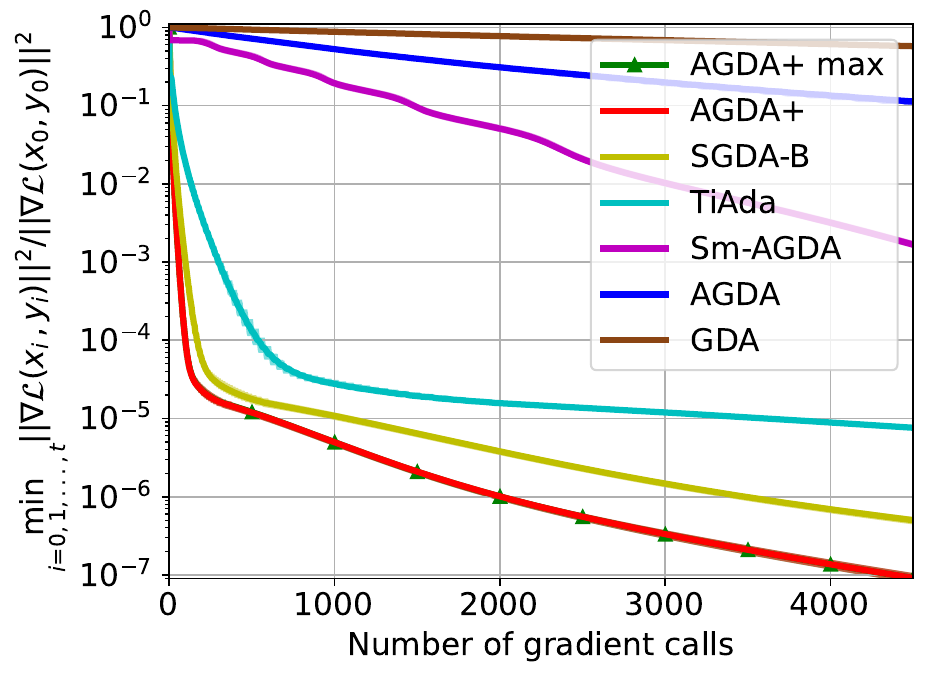}
         \caption{$\kappa=10$}
         \label{fig:sin-Q-kappa-10-last}
     \end{subfigure}
     \begin{subfigure}[b]{0.34\textwidth}
         \centering
\includegraphics[width = \textwidth]{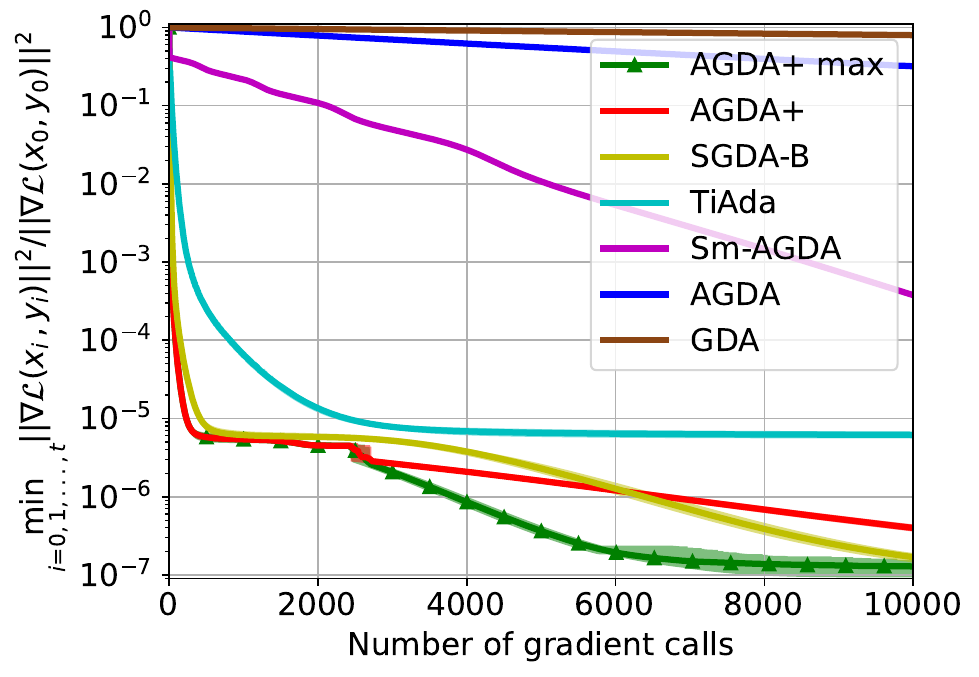}
         \caption{$\kappa=20$}
         \label{fig:sin-Q-kappa-20-last}
     \end{subfigure}
%      \begin{subfigure}[b]{0.24\textwidth}
%          \centering
% \includegraphics[width = \textwidth]{figure/exp:deterministic/Q_stdx_0_stdy_0_muy_1_kappa_50_b_1_z_last.pdf}
%          \caption{$\kappa=50$}
%          \label{fig:Q-kappa50-sto}
%      \end{subfigure}
\caption{
Comparison of \agdap{} against \sgdab{}, \tiada,  \smagda, \agda, and \gda{} on the quadratic WCSC minimax problem with sinusoidal perturbation stated in~\eqref{eq:sin-bilinear-problem} using $10$ simulation runs. Solid lines indicate the average over 10 runs and shaded regions around the solid lines denote the range statistic. The plots in the top row display the last-iterate performance while the plots in the bottom row display the best-iterate performance.}
\label{fig:sin-Q}
\centering
\end{figure*}
\subsection{\na{Real Data:} Distributionally Robust Optimization~(DRO) Problem} %with Neural Network}
%Next, we test on a distributionally robust optimization 
Consider the DRO problem 
\cite{namkoong2016stochastic},
\begin{equation}
{\small
\begin{aligned}
        \label{eq:dro-problmm}
        \min_{\bx\in\mathbb{R}^m} \max_{\by\in%\mathbb{R}^n
        \sa{S^n}}f(\bx,\by)=\sum_{i=1}^{n}y_i \ell_i(p(\ba_i;\bx),b_i) -r(\by), 
        %\\
        %s.t. &\quad \sum_{i=1}^{n}y_i=1,\;y_i\geq0,\; \forall\;i
\end{aligned}}%
\end{equation}
\sa{where $S^n=\{\by\in\reals^n:\ \sum_{i=1}^{n}y_i=1,\;y_i\geq0,\; i=1,\ldots,n\}$,} $\ba_i\in\reals^d$, and $b_i\in\{1,-1\}$ are binary classification paired data and label. \na{We adopted a slightly modified version of} the robust neural network training problem from \cite{zhang2022sapd+,Deng2022}. The function $p:\mathbb{R}^d\rightarrow \mathbb{R}^1$ denotes three-layer perceptron neural network. We use the binary logistic loss function
\sa{${\ell_i}(\bx,b_i)=\ln(1+\exp(-b_i \bx))$} and adopt \sa{the regularizer %function is defined as
$r(\by)=\frac{\mu}{2}\|\by - \mathbf{1}/n\|^2$ which makes sure that the worst case distribution will not be too far from the uniform distribution} --here, $\mathbf{1}$ denotes the vector with all entries equal to one, and we set $\mu=0.01$. 
%and $g(\by)$ is to control the weight $\by$ near the uncertainty set around the uniform distribution.
We %perform experiments 
conduct the experiment on the data set
% $i)$ \verb+a9a+ with $n = 32561$, $d = 123$;
%$i)$ 
\verb+gisette+ with $n = 6000$ and $d = 5000$, 
% $ii)$ \verb+sido0+ with $n = 12678$, $d = 4932$.
% The dataset
% \verb+sido0+ is obtained from
% Causality Workbench\footnote{http://www.causality.inf.ethz.ch/challenge.php?page=datasets} and
%\texttt{gisette} 
which can be downloaded from LIBSVM repository.\footnote{https://www.csie.ntu.edu.tw/~cjlin/libsvmtools/datasets/binary.html} We normalized the data set by letting $\ba_i = \frac{\ba_i - \min_j{a_{ij}}\cdot\mathbf{1}}{\max_{j}{a_{ij}}-\min_j{a_{ij}}}$, where $\ba_i = (a_{ij})_{j=1}^d$ is the $i$-th data for $i=1,\ldots,n$.

\paragraph{Parameter setting.} We used the full data set for all algorithms. We estimated $L$ 
%To estimate the problem parameters for other algorithms,
deriving a lower bound for it; indeed, since $\frac{\partial f}{\partial \by\partial \bx} =
 [-\frac{b_i\exp(-b_i p(\ba_i;\bx))}{1+\exp(-b_i p(\ba_i;\bx))}\frac{\partial p(\ba_i;\bx)}{\partial \bx}]_{i=1}^{n}$, 
 %we can indeed estimate a lower bound of $L$, i.e., 
 we have $L\geq \tilde{L}=\|\frac{\partial f}{\partial \by\partial \bx}\mid_{(\bx,\by)=(\bx_0,\by_0)}\|_2$, where $\bx^0$ is generated by Xaiver method for initializing deep neural networks \cite{glorot2010understanding} and $\by_0=\mathbf{1}/n$. {Since $L$ is unknown, we used $\tilde L$ to set the step sizes for \gda{}, \agda{}, and \smagda.} 
 %using lower bound of $L$ is fair for other algorithms}.  
 Indeed, we obtain $\tilde L\approx 0.0532$, thus $\kappa\approx 5.32$. \gda{} and \agda{} worked well with $\tilde L$; however, \smagda~sequence diverged so we removed it from the comparison.
 \na{In \cref{fig:DRO}, we observe that \agdap{} is competitive against the other methods we tested; this is mainly because \agdap{} is robust to the changes in the local Lipschitz constant over the problem domain and can efficiently use the local Lipshitz constant estimates to set relatively larger step sizes compared the other methods tested, which cannot use non-nonotonic step sizes. We also observe that throughout the run time, \agdap{} exhibit relatively lower range statistics over 10 runs than the other adaptive algorithms \sgdab{} and \tiada{}. Furthermore, the squared norm of gradient map values for the last iterate, $\norm{G(\bx_t,\by_k)}^2$, exhibits less fluctuations for \agdap{} when compared to \sgdab{}.}  
 %at the early stage and consistently enjoys fast convergence properties even after other algorithm gets slower.

 \begin{figure}[h]
     \centering
         \centering
\includegraphics[width = 0.33\textwidth]{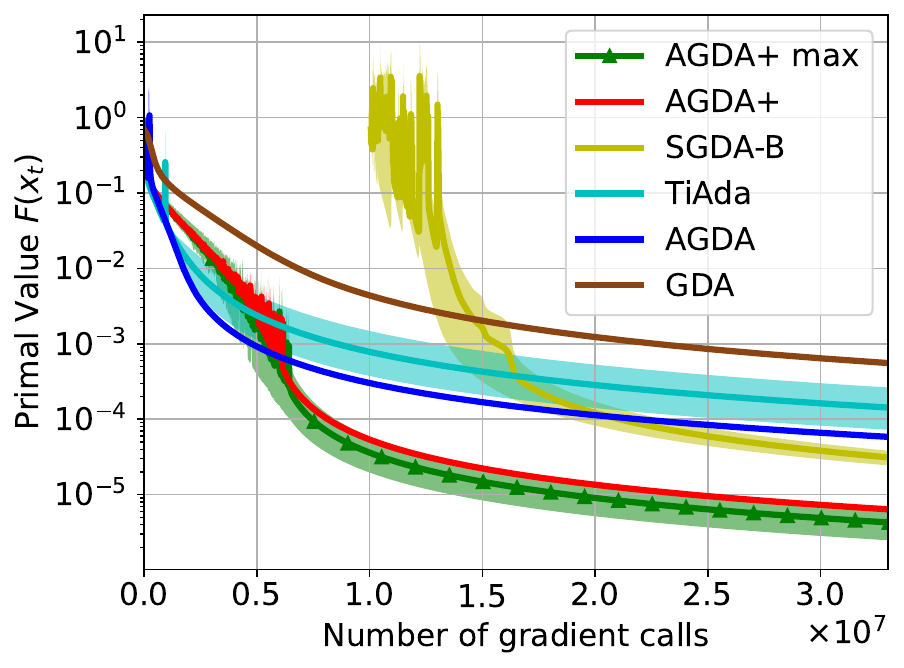}
\includegraphics[width = 0.33\textwidth]{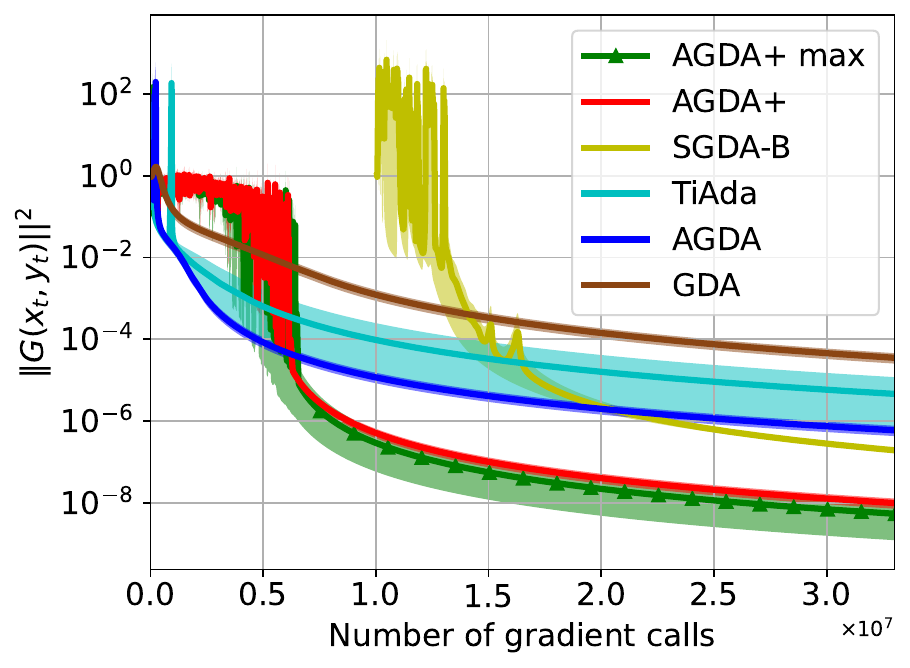}
\includegraphics[width = 0.33\textwidth]{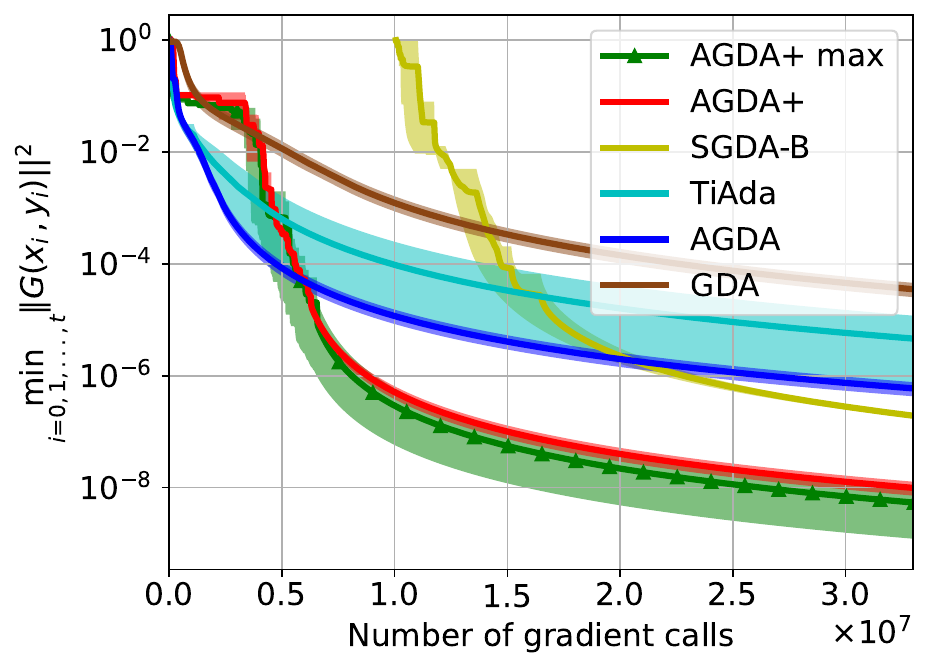}
% \includegraphics[width=0.24\textwidth]{figure/exp:deterministic/gisette_muy_0_01_kappa_1_b_6000_acc.pdf}
% \\
% \includegraphics[width = 0.3\textwidth]{figure/exp:sto/gisette_muy_0_01_kappa_1_b_200_loss.pdf}
% \includegraphics[width = 0.3\textwidth]{figure/exp:sto/gisette_muy_0_01_kappa_1_b_200_z.pdf}
% \includegraphics[width=0.3\textwidth]{figure/exp:sto/gisette_muy_0_01_kappa_1_b_200_acc.pdf} 
% \includegraphics[width=0.225\textwidth]{figure/exp:deterministic/gisette_muy_0_01_kappa_1_b_6000_acc.pdf}
% \includegraphics[width=0.23\textwidth]{figure/exp:sto/gisette_muy_0_01_kappa_1_b_200_acc.pdf} 
% \includegraphics[width = 0.23\textwidth]{figure/exp:deterministic/gisette_muy_0_01_kappa_1_b_6000_loss.pdf}\\
% \includegraphics[width = 0.235\textwidth]{figure/exp:sto/gisette_muy_0_01_kappa_1_b_200_loss.pdf}
% \includegraphics[width = 0.238\textwidth]{figure/exp:deterministic/gisette_muy_0_01_kappa_1_b_6000_z.pdf}
% \includegraphics[width = 0.23\textwidth]{figure/exp:sto/gisette_muy_0_01_kappa_1_b_200_z.pdf}
\caption{
Comparison of \agdap{} against \sgdab{}, \tiada, \agda, and \gda{} on the DRO problem stated in~\eqref{eq:dro-problmm} using $10$ simulation runs on the data set
\texttt{gisette} with $n = 6,000$ and $d = 5,000$. ``Primal value" denotes $F(\bx)=\max_{\by\in S^n} f(\bx,\by)$. %The batch sizes used are $6000$. 
Solid lines indicate the average over 10 runs and shaded regions around the solid lines denote the range statistic.
% and $200$, respectively, i.e., %we test on deterministic problems to verify our analysis, and stochastic problems for practice. 
% (left) uses deterministic gradients and (right) uses stochastic gradients.
%One epoch means one complete pass of the data set.
}
\label{fig:DRO}
\end{figure}

%  \begin{figure}[h]
%      \centering
%          \centering
%         \includegraphics[width=0.23\textwidth]{figure/exp:sto/gisette_muy_0_01_kappa_1_b_200_acc.pdf}
% \includegraphics[width = 0.23\textwidth]{figure/exp:sto/gisette_muy_0_01_kappa_1_b_200_loss.pdf}
% \includegraphics[width = 0.23\textwidth]{figure/exp:sto/gisette_muy_0_01_kappa_1_b_200_z.pdf}
% \caption{
% Comparison of \agdap{} against \gda, \agda, and \tiada~on the DRO problem in~\eqref{eq:dro-problmm} using $10$ simulation runs. ``Train error" denotes the fraction of wrong prediction, and ``loss" denotes $F(\bx)=\max_{\by\in\cY}\cL(\bx,\by)$.
% }
% \label{fig:DRO}
% \centering
% \end{figure}

\section{Conclusions}
In this paper, we considered deterministic double-regularized nonconvex strongly-concave minimax problems.
%This setting arises in many applications ranging from distributionally robust learning to GANs. 
We proposed a new method \agdap, which achieves the optimal iteration complexity of $\cO(\epsilon^2)$ with the best $\cO(1)$ constant among the existing adaptive gradient methods agnostic to Lipschitz constant $L$ \rv{and/or concavity modulus $\mu$}, i.e., \tiada{} and \neada{}; moreover, our adaptive step sizes are not necessarily monotonically decreasing, thus \agdap{} exhibits faster convergence behavior in practice when compared to the others using monotonic step size sequences. To the best of our knowledge, \agdap{} is the first adaptive gradient method that can exploit the local structure of the problem through adopting a non-monotone step size search scheme.
%We also showed that our \agdap{} can be adapted to the WCMC scenario.%Finally, we provided numerical experiments demonstrating that \agdap{} can achieve a state-of-the-art performance on distributionally robust learning arising in ML.

\bibliography{refs}
\bibliographystyle{plain}

\appendix
\onecolumn

\section{\rv{Essential definitions for the convergence proof of} \agdap{}}
\label{sec:general-alg}
We first provide some %useful
definitions that are essential for the proof. %that relate to important timings for running the algorithm.
\begin{defn}\label{def:tn} 
\na{Let $\{\hat{t}_n\}_{n\in\N}\subset\bar\N$ be a sequence such that 
\na{$\hat{t}_0=0$. For any fixed $n\in\N$, if $L_{\hat{t}_n}<+\infty$, then
$\hat{t}_{n+1} = \min\{t\in \bar\N:\ L_t>L_{\hat{t}_{n}}\}$} with the convention that $L_\infty\triangleq +\infty$; otherwise, if $L_{\hat{t}_n}=+\infty$, then $\hat t_{n+1}=\infty$.}
\end{defn}
%\xtodo{We don't need $\mu_t$ in the definition.}
%\xtodo{Should it be $t\in[\hat{t}_n, \hat{t}_{n+1})$}
\begin{defn}\label{def:N-k}
    Define ${N(t)}\triangleq\max\{n\in \mathbb{N}\xzr{:}\;t\geq \hat{t}_n\}$ for any $t\in \mathbb{N}$, and \na{$\bar N\triangleq \max\{n\in\N:\ L_{\hat t_n}<\infty\}$.}
    %integer $t\geq 0$.
\end{defn}
\begin{remark}
    The sequence $\{\hat{t}_n\}_{n\geq 0}$ denotes all the discrete-time instances at which the quantity $L_t$ increases {and $\mu_t$ decreases}.
    \na{The definition %of $\{\hat t_n\}_{n\geq 0}$ 
    implies that $\hat t_n<\hat t_{n+1}$ for all %$n\geq 0$
    $n\in[\bar N]$; moreover, $L_t = L_{\hat{t}_n}$ {and $\mu_t = \mu_{\hat{t}_n}$} for all $t\in\mathbb{N}$ such that $t\in [\hat{t}_{n},\hat{t}_{n+1})$.}  
\end{remark}
\begin{remark}
\na{For all $t\in \mathbb{N}$, it holds that} {$L_{t}=L_{\hat{t}_{N(t)}}$ {and $\mu_t = \mu_{\hat{t}_{N(t)}}$}.} 
%for all integer $t\geq 0$, where $\{\hat{t}_n\}$ are defined in Definition~\ref{def:tn}.
\end{remark}
%\xtodo{$\tilde y$ needs to be defined.}
% \begin{remark}
%     To be used in the proof, we define an auxiliary sequence $\{\tilde \by_t\}_{t\geq 1}$ such that %$\tilde \by_0=\by_0$ and 
%     for all $t\in\mathbb{N}$ we set $\tilde \by_{t+1}=\by_{t+1}$ right after $\by_{t+1}$ is computed in \algline{8}. The auxiliary sequence is never stored/computed in practice. Note that $\tilde \by_t = \by_t$ for all $t\notin\{ \hat{t}_n\}_{n\geq 0}$, and for $t\in\{ \hat{t}_n\}_{n\geq 0}$ since we reset $\by_t$ as in~\algline{4}, it is possible that $\tilde \by_t \neq \by_t$.
%     % For \texttt{correction==False}, $\tilde \by_t=\by_t$ for all $t\geq 0$. On the other hand, for \texttt{correction==True}, $\tilde \by_t = \by_t$ only when $t\notin\{ \hat{t}_n\}_{n\geq 0}$. 
%     %where $\{ \hat{t}_n\}$ are  defined in Definition~\ref{def:tn}.
% \end{remark}
%\xtodo{I think it is better to use $L_{t_n} = \frac{\mu}{\gamma_1^{k_n}}$ because part of our analysis bases on that.}
\begin{defn}\label{def:kn}
\na{Let $\cT\triangleq\{\hat t_n:\ n\in[\bar{N}]\}$ and $\cK\triangleq\{k_n:\ n\in[\bar{N}]\}$, where  $\mathbb{N}\ni k_n\triangleq\log_{1/\gamma_1}\Big(L_{\hat t_n}/\tilde l\Big)$ for all $n\in [\bar N]$.}
\end{defn}
%\qstodo{not sure about the definition of $k_n$ here}
\begin{remark}
\label{rem:L_tn}
Note that \na{$L_{t} = \tilde l/\gamma_1^{k_{N(t)}}$}
%=\tilde l/\gamma^{r k_{N(t)}}$}  
{and $\mu_t = \max\{\tilde\mu \gamma_1^{k_{N(t)}}, \underline{\mu}\}$} for all $t\in\mathbb{N}$. 
%since $\gamma_1=\gamma_2^r$ for some $r\in\mathbb{N}_+$ and $\gamma_2=\gamma\in(0,1)$. 
\rv{When $\mu$ is known, the initial estimate $\tilde\mu$ is set to $\mu$, and from the definition of $\underline{\mu}$ given in Algorithm~\ref{alg:agda}, we get $\mu_t=\mu$ always holds for all $t\in\mathbb{N}$.}
\end{remark}
%\todo{Xuan:should it be $\gamma_2^{rK_{N(t)}}$? And we also should say $\gamma_2 = \gamma_1^r$?}
%\nsa{Note that $\gamma_2=\gamma$}
{Now}, we {recall an important result}, frequently utilized {in the literature for} the convergence analysis of first-order algorithms \na{on WCSC problems.}
\begin{lemma}[Lemma 4.1 and Proposition 4.1 in~\cite{boct2020alternating}]
%[Proposition 1 \cite{chen2021proximal}]
\label{lem:primal_lipschitz}
Under \cref{ASPT:lipshiz gradient}, %,aspt:primal_lb} 
$\by^*(\cdot):\mathbb{R}^d \to \mathbb{R}^m$ is $\kappa$-Lipschitz and $\Phi(\cdot):\mathbb{R}^d \to \mathbb{R}$ is $(\kappa+1)L$-smooth, \na{where $y^*(\cdot)$ and $\Phi(\cdot)$ are defined in Definitions~\ref{def:opt-response} and \ref{def:max-function}, respectively.}  
\end{lemma}

\section{Proof of \cref{thm:backtrack}} %{Backtracking analysis for \agdap{}}
\label{appendix:ls_condition}
\na{In this section, we show that \agdap{} is well defined, in the sense that for all $t\geq 0$, the backtracking condition \cond{t} given in \eqref{eq:backtrack} holds within a finite number of checks. We first give some technical lemmas.}
% \xz{In this section, we first propose several fundamental conditions as stated in \cref{eq:ls-all}. We shows that  i) \cref{eq:ls-all} will hold when $L_t\in [L, L/\gamma_1]$ and ${l}_t \in [L,L_t]$ in \cref{lemma:complicated-ls-hold};  ii) \cref{eq:ls-all} holding implies \cref{eq:backtrack} holding in \cref{lemma:complicated-imply-short-ls-flag-true-part-1,lemma:complicated-imply-short-ls-flag-false-part-1}. Then we conclude  \cref{thm:ls-hold-main}.
% Following that, we show the proof of \cref{thm:avg--backtrack-time}.
% %the average times for checking \cref{eq:ls-conv} for each $t\geq 0$ is.
% % which implies the backtrack iteration complexity of \agdap{} has the same order of iteration complexity of \agdap{} in the long run, i.e., as $T$ large enough. Furthermore, the iteration complexity refers to how many times we accepted the $(x_t,y_t)$ and backtrack iteration complexity refers to how many times we check the condition \cref{eq:ls-conv}.
% }

\subsection{Technical lemmas}
In this section, we give some technical results that hold for all $t\in\mathbb{N}$.
\begin{lemma}\label{lemma:complicated-ls-hold}
    \na{Under Assumption~\ref{ASPT:lipshiz gradient}, for any iteration $t\in\mathbb{N}$, at the time when the backtracking condition \emph{\cond{t}} in \eqref{eq:backtrack} is checked during the execution of 
    \emph{\shortalgline{\ref{algline:check-agda+}}} of \cref{alg:agda}, 
    %or \emph{\shortalgline{\ref{algline:check-sagda+}}} of \cref{alg:sagda+}, 
    if %$L_t\in [L, L/\gamma_1]$ and ${l}_t \in [L,L_t]$, 
    \na{$l_t\geq L$} {and $\mu_t \in(0,\mu]$,} then all the inequalities below are guaranteed to hold, \na{where $\delta_t\triangleq\norm{\by_{t}-\by^*(\bx_t)}^2$,}}
\begin{subequations}\label{eq:ls-all}
\begin{equation}
    \label{eq:LS-1}
    \begin{aligned}
        &f(\tilde\bx_{t+1},\by_{t})
        -\fprod{\grad_x f(\bx_{t},\by_{t}),~\tilde\bx_{t+1}-\bx_{t}} 
        \leq  f(\bx_{t},\by_{t}) + \frac{l_t}{2} \|\tilde\bx_{t+1}-\bx_{t}\|^2,
    \end{aligned}
\end{equation}
\begin{equation}
    \label{eq:LS-2}
    \norm{\grad_y f(\na{\tilde\bx_{t+1}},\by_{t}) - \grad_y f(\bx_{t},\by_{t})} \leq l_t \norm {\na{\tilde\bx_{t+1}}-\bx_{t}},
\end{equation}
\begin{equation}
\label{eq:LS-Lyy}
    \begin{aligned}
         & f(\tilde\bx_{t+1},\by_{t}) + \langle \grad_y f(\tilde\bx_{t+1}, \by_{t}),~ \tilde\by_{t+1} - \by_{t} \rangle \leq f(\tilde\bx_{t+1}, \tilde\by_{t+1})+ \frac{l_{t}}{2} \|\tilde\by_{t+1} - \by_{t}\|^2,
         %, \forall s,\; s.t.\;t_{N(t)}\leq s\leq t
    \end{aligned}
\end{equation}
\begin{equation}\label{eq:LS-7}
    \|\nabla_y f\left(\tilde\bx_{t+1}, \tilde\by_{t+1}\right)-\nabla_y f\left(\tilde\bx_{t+1}, \by_{t}\right)\| \leq  l_t\left\|\tilde\by_{t+1}-\by_{t}\right\|,
\end{equation}
% \begin{equation}
%     \label{eq:LS-xt+1yt}
%     \norm{G_y(\bx_{t},\by_{t})}^2 \leq \left(\frac{4}{\sigma^2_{t}}+2l_t^2\right) \Delta_t;
% \end{equation}
\begin{equation}
    \label{eq:LS-xtyt}
    \norm{G^{\sigma_t}_y(\bx_{t},\by_{t})}^2 \leq \left(\frac{4(1-\sigma_t{\mu_t})}{\sigma^2_{t}}+2l_t^2 \right) \delta_t, 
\end{equation}
\begin{equation}
    \label{eq:LS-xt+1yt}
    \norm{G^{\sigma_t}_y(\tilde\bx_{t+1},\by_{t})}^2 \leq 2\left(\frac{4(1-\sigma_t{\mu_t})}{\sigma^2_{t}}+2l_t^2 \right) \delta_t + 2l_t^2\tau_t^2\norm{G_x^{\tau_t}(\bx_{t},\by_{t})}^2.
\end{equation}
\end{subequations}
\end{lemma}
% \begin{remark}
% \cref{lemma:complicated-ls-hold} holds regardless of whether \texttt{correction} is set to \texttt{True} or \texttt{False}.
% \end{remark}
\begin{proof}
\na{When %\cref{alg:agda} is checking at \shortalgline{\ref{algline:check-agda+}} 
%either of these algorithms 
\agdap{} checks the backtracking condition in \eqref{eq:backtrack} at time step $t$, if $l_t\geq L$, then} 
%$L_t\in [L, L/\gamma_1]$ and ${l}_t\in [L,L_t]$, 
the conditions \eqref{eq:LS-1}, \eqref{eq:LS-2}, \eqref{eq:LS-Lyy} and \eqref{eq:LS-7} naturally hold \na{since $\grad f$ is $L$-Lipschitz according to \cref{ASPT:lipshiz gradient}.} 
%and \xz{the backtracking condition, i.e., \cref{eq:Lyy-1,eq:Lyy-2}} for $s\leq t$. 
Therefore, we only need to show \cref{eq:LS-xt+1yt,eq:LS-xtyt} also hold \na{whenever $l_t\geq L$ {and $\mu_t \in(0,\mu]$}.}

\na{Let $\by^*_t\triangleq \by^*(\bx_t)$, since $G^{\sigma_t}_y(\bx_{t},\by_{t}^*)=\mathbf{0}$, it follows that for $l_t\geq L$, we have}
% \begin{equation}\label{eq:bound-Gy2delta}
% \begin{aligned}
%         \norm{G_y (\bx_{t},\by_{t})}^2 =&
%         \norm{G_y (\bx_{t},\by_{t}) - G_y (\bx_{t},\by^*(\bx_{t}))}^2 
%         \\
%         = &
%         \norm{\grad_y f(\bx_{t+1},\by_{t+1}) - \grad_y f(\bx_{t+1},\by^*(\bx_{t+1}))}^2 
%         \\
%         \leq & L_{t,3}^2\delta_{t+1}
% \end{aligned}
% \end{equation}
\begin{equation}\label{eq:bound-Gy2delta}
    \begin{aligned}
    \MoveEqLeft[2]\norm{G^{\sigma_t}_y(\bx_{t},\by_{t})}^2 =\norm{G^{\sigma_t}_y(\bx_{t},\by_{t}) - G^{\sigma_t}_y(\bx_{t},\by_{t}^*)}^2\\
        =& \frac{1}{\sigma_{t}^2} \norm{\prox{\sigma_{t} h}\Big(\by_{t}+\sigma_{t}\grad_y f(\bx_{t},\by_{t})\Big)-\prox{\sigma_{t}h}\Big(\by_{t}^*+\sigma_{t}\grad_y f(\bx_{t},\by_{t}^*)\Big)+\by_{t}^*-\by_{t}}^2\\
        \leq &\frac{1}{\sigma_{t}^2}\Big(2\norm{\by_{t}-\by_{t}^*+\sigma_{t}\grad_y f(\bx_{t},\by_{t})-\sigma_{t}\grad_y f(\bx_{t},\by_{t}^*)}^2 + 2\norm{\by_{t}^*-\by_{t}}^2 \Big)\\
        \leq &\Big(\frac{4(1-\sigma_t{\mu_t})}{\sigma_{t}^2}+2l_t^2\Big)\delta_{t},
    \end{aligned}
\end{equation}
% \begin{equation}
%     \begin{aligned}
%         &\norm{G_y(\bx_{t+1},\by_{t+1})}^2 =\norm{G_y(\bx_{t+1},\by_{t+1}) - G_y(\bx_{t+1},\by_{t+1}^*)}^2\\
%         &= \frac{1}{\sigma_{t+1}^2} \norm{\prox{\sigma_{t+1} h}\Big(y_{t+1}+\sigma_{t+1}\grad_y f(\bx_{t+1},\by_{t+1})\Big)\\
%         &\quad-\prox{\sigma_{t+1}h}\Big(\by_{t+1}^*+\sigma_{t+1}\grad_y f(\bx_{t+1},\by_{t+1}^*)\Big)+\by_{t+1}^*-\by_{t+1}}^2\\
%         &\leq \frac{1}{\sigma_{t+1}^2}\Big(2\norm{\by_{t+1}-\by_{t+1}^*+\sigma_{t+1}\grad_y f(\bx_{t+1},\by_{t+1})-\sigma_{t+1}\grad_y f(\bx_{t+1},\by_{t+1}^*)}^2 + 2\norm{\by_{t+1}^*-\by_{t+1}}^2 \Big)\\
%         &\leq \Big(\frac{6}{\sigma_{t+1}^2}+2L_3^2\Big)\delta_{t+1}
%     \end{aligned}
% \end{equation}
where we use the nonexpansive property of $\prox{\sigma_t h}(\cdot)$ for the first inequality, \na{and we use \cref{ASPT:lipshiz gradient} 
for the last inequality --more precisely, the \na{strong} concavity of $f(\bx_t,\cdot)$, i.e., $\fprod{\grad_y f(\bx_{t},\by_{t})-\grad_y f(\bx_{t},\by^*_{t}),~\by_{t}-\by_{t}^*}\leq -\mu\norm{\by_{t}-\by_{t}^*}^2$ and {$\mu_t\in(0,\mu]$}, and $L$-Lipschitz continuity of $\grad f$. Hence, \eqref{eq:LS-xtyt} holds.} 
% Moreover, because $L_t\in [L,L/\gamma_1]$ and \cref{eq:LS-Lyy} holds, \cref{lemma:Delta-Bound} implies that
% $
%     \delta_{t}\leq \Delta_{t}
% $, \cref{eq:bound-Gy2delta} further implies that
% \xzr{\begin{equation}\label{eq:bound-Gy2.5delta}
%         \norm{G^{\sigma_t}_y(\bx_{t},\by_{t})}^2 \leq \Big(\frac{4}{\sigma_{t}^2}+2l_t^2\Big)\Delta_{t}.
% \end{equation}}

Similarly, \na{for $l_t\geq L$}, we can obtain that
%\xzr
{\begin{equation}\label{eq:bound-Gy3delta}
    \begin{aligned}
         \MoveEqLeft[2]\|G^{\sigma_t}_y(\na{\tilde\bx_{t+1}},\by_t)-G^{\sigma_t}_y(\bx_{t},\by_t)\|^2 
        \\
        = & \frac{1}{\sigma_{t}^2} \norm{\prox{\sigma_{t} h}\Big(\by_{t}+\sigma_{t}\grad_y f(\na{\tilde\bx_{t+1}},\by_{t})\Big)-\prox{\sigma_{t}h}\Big(\by_{t}+\sigma_{t}\grad_y f(\bx_{t},\by_{t})\Big)}^2\\
        \leq &\norm{\grad_y f(\na{\tilde\bx_{t+1}},\by_{t})-\grad_y f(\bx_{t},\by_{t})}^2 
        \leq  l^2_t \|\na{\tilde\bx_{t+1}}-\bx_{t}\|^2
        = l^2_t\tau_t^2 \|G_x^{\tau_t}(\bx_t,\by_t)\|^2.
    \end{aligned}
\end{equation}
Furthermore, using the simple bound, %following inequality 
\begin{equation}\label{eq:bound-Gy4delta}
    \begin{aligned}  \|G^{\sigma_t}_y(\tilde\bx_{t+1},\by_t)\|^2 \leq 2\|G^{\sigma_t}_y(\tilde\bx_{t+1},\by_t)-G^{\sigma_t}_y(\bx_{t},\by_t)\|^2 + 2\|G^{\sigma_t}_y(\bx_{t},\by_t)\|^2,
    \end{aligned}
\end{equation}
together with \cref{eq:bound-Gy2delta,eq:bound-Gy3delta} implies that \cref{eq:LS-xt+1yt} holds, which completes the proof.
%\cref{eq:ls-all} is guaranteed to hold.
}
\end{proof}

% In the following part, we will show \cref{eq:ls-all} holding implies \cref{eq:backtrack} holding for two cases: i)\texttt{correction==True} and ii)\texttt{correction==False}.
% \subsubsection{Proof of \cref{thm:ls-hold-main}: case 1 (\texttt{correction==True})}
% In this section, we will show when $t=0$ and checking at \shortalgline{18} of {alg:agda}, \cref{thm:ls-hold-main} holds under certain conditions.

\begin{remark}
    \na{The conditions in \cref{eq:LS-1,eq:LS-2,eq:LS-Lyy,eq:LS-7} are clearly related to the local Lipschitz constants for $\grad_x f(\cdot, \by_t)$, $\grad_y f(\cdot, \by_t)$ and $\grad_y f(\tilde\bx_{t+1}, \cdot)$ in the region 
    %\xtodo{Should the third term be $\grad_y f(\tilde \bx_{t+1},\cdot)$?}
    $[\bx_t, \tilde\bx_{t+1}]\times [\by_t, \tilde\by_{t+1}]$.
    %Indeed, 
    %the term $\Delta_t$ is designed to upper bound the dual optimality error $\delta_t$.
    Moreover, since $G_y^{\sigma_t}(\bx_t,\by^*(\bx_t))=\textbf{0}$, the condition in \cref{eq:LS-xtyt} %can also be interpreted as a form of 
    is also related to a local Lipshitz structure, i.e., \cref{eq:LS-xtyt} holds if and only if $\norm{G^{\sigma_t}_y(\bx_{t},\by_{t})-G^{\sigma_t}_y(\bx_t,\by^*(\bx_t))} \leq \left(\frac{4(1-\sigma_t{\mu_t})}{\sigma^2_{t}}+2l_t^2\right)^{1/2} \norm{\by_t-\by^*(\bx_t)}$.} 
    %Further discussion about the relation between $\delta_t$ and $\Delta_t$ is provided in \tbd.
\end{remark}
\begin{lemma}[Descent lemma on $\bx$]\label{lemma:descent_x-part1}
Suppose {Assumption \ref{ASPT:lipshiz gradient}
and \cref{eq:LS-1}} hold. 
%\xzr{When \cref{alg:agda} is checking the backtracking condition in \eqref{eq:backtrack} at \shortalgline{\ref{algline:check-agda+}} or \cref{alg:sagda+} is checking the backtracking condition in \eqref{eq:backtrack} at \shortalgline{18}} at time step $t\geq 0$. Then
\na{For any iteration $t\in\mathbb{N}$, at the time when the backtracking condition \emph{\cond{t}} in \eqref{eq:backtrack} is checked during the execution of 
\emph{\shortalgline{\ref{algline:check-agda+}}} of \cref{alg:agda}, 
%or \emph{\shortalgline{\ref{algline:check-sagda+}}} of \cref{alg:sagda+}, 
it holds that}
\begin{equation}\label{eq:descent_x_inner-part1}
{%\small
    \begin{aligned}
      \MoveEqLeft\Big(\tau_{t} - \frac{\tau_{t}^2l_{t}}{2}\Big)\|G_x^{\tau_t}(\bx_t,\by_t)\|^2
      \\ &\leq  %g(\bx_t)+f(\bx_t,\by_t)- f(\tilde\bx_{t+1},\tilde\by_{t+1}) -g(\tilde\bx_{t+1})
      \na{\cL(\bx_t,\by_t)-\cL(\tilde\bx_{t+1},\tilde\by_{t+1})+h(\by_t)-h(\tilde\by_{t+1})}
      +f(\tilde\bx_{t+1},\tilde\by_{t+1})- f(\tilde\bx_{t+1},\by_t).
    \end{aligned}}%
\end{equation}
\end{lemma}
% \begin{remark}
% \cref{lemma:descent_x-part1} holds regardless of whether \texttt{correction} is set to \texttt{True} or \texttt{False}.
% \end{remark}
\begin{proof}
\na{From the update rule for $\tilde\bx_{t+1}$, we have $\frac{1}{\tau_t}(\bx_t-\tilde\bx_{t+1}) -\grad_x f(\bx_t,\by_t)\in\partial g(\tilde\bx_{t+1})$. Since $g$ is convex, we get
\begin{equation*}
    g(\tilde\bx_{t+1})+\fprod{\frac{1}{\tau_t}(\bx_t-\tilde\bx_{t+1}) -\grad_x f(\bx_t,\by_t),~\bx_t-\tilde\bx_{t+1}} \leq g(\bx_t).
\end{equation*}
Combining it with \cref{eq:LS-1} yields that
\begin{equation}
     g(\tilde\bx_{t+1})+f(\tilde\bx_{t+1},\by_{t})+\frac{1}{\tau_t}\norm{\bx_t-\tilde\bx_{t+1}}^2\leq g(\bx_t)+f(\bx_{t},\by_{t})+\frac{l_{t}}{2}\| \tilde\bx_{t+1}-\bx_t \|^2.
\end{equation}
Finally, using the fact $\frac{1}{\tau_t}(\bx_t-\tilde\bx_{t+1})=G_x^{\tau_t}(\bx_t,\by_t)$, we get
\begin{equation}\label{eq:use_ls1-temp-part1}
    \begin{aligned}
      \Big(\tau_{t} - \frac{\tau_{t}^2l_{t}}{2}\Big)\|G_x^{\tau_t}(\bx_t,\by_t)\|^2\leq &g(\bx_t)+f(\bx_t,\by_t) - f(\tilde\bx_{t+1},\by_{t}) -g(\tilde\bx_{t+1}),
    \end{aligned}
\end{equation}
and the desired inequality is obtained after adding and subtracting $f(\tilde\bx_{t+1},\tilde\by_{t+1})$ to the rhs of the inequality above.}
\end{proof}

\subsection{Case analysis for the proof}
\na{Our objective is to show that for every $t\in\mathbb{N}$, the backtracking condition \cond{t} stated in \eqref{eq:backtrack} holds after checking \algline{\ref{algline:check-agda+}} finitely many times. Hence, for any $t\in\mathbb{N}$ given, we consider four exhaustive cases depending on the current values of $L_t,{\mu_t}$ and $\hat t_{N(t)}$ at the time we check \algline{\ref{algline:check-agda+}}. %It is important to 
Note that within the same iteration $t$, one may check \algline{\ref{algline:check-agda+}} multiple times with different values of $\hat t_{N(t)}$, $l_t, {\mu_t}$ and $L_t$:~\footnote{\na{Every time \cond{t} in \eqref{eq:backtrack} fails to hold, we increase the value of $l_t$ within the same iteration $t$. Note that one may have $t>\hat t_{N(t)}$ for the first time %\eqref{eq:backtrack}
\cond{t} is checked within iteration $t$; however, if \eqref{eq:backtrack} fails to hold consecutively multiple times, then one may get $l_t>L_t$ eventually after increasing the value of $l_t$ for each time \eqref{eq:backtrack} fails, in which case, according to \agdap{} one \rv{decreases the value of $\mu_t$} and increases the value of $L_t$, leading to the case $t=\hat t_{N(t)}$ for the same iteration $t$.}}
\begin{enumerate}[align=left]
    \item[\textbf{CASE 1:} $t=\hat t_{N(t)}$.] We consider two subcases, i.e.,\\[1mm] 
    (\textbf{CASE 1a}) $L_t<L$ or {$\mu_t>\mu$}, and (\textbf{CASE 1b}) $L_t\geq L$ and {$\mu_t \in (0,\mu]$};
    \item[\textbf{CASE 2:} $t>\hat t_{N(t)}$.] We consider two subcases, i.e.,\\[1mm] 
    (\textbf{CASE 2a}) $L_t<L$  or {$\mu_t>\mu$}, and (\textbf{CASE 2b}) $L_t\geq L$  and {$\mu_t \in (0,\mu]$}. 
\end{enumerate}}%

\subsubsection{\na{Analysis of CASE 1b: $t=\hat t_{N(t)}$ and $L_t\geq L$ and {$\mu_t \in (0,\mu]$}}}
\label{sec:agda-t=tn}
%{\na{Induction proof for \cref{alg:agda} (Base case: $t= \hat{t}_n$ for some $n\in\mathbb{N}$)}}
\na{For given $t\in\mathbb{N}$, if one has $t=\hat t_{N(t)}$, then either $t=0$ %it means that 
or the value of $L_t$ has increased within iteration $t$, i.e., $L_t>L_{t-1}$. Since $\hat t_0=0$, in either case, there exists some $n\in\mathbb{N}$ such that $t=\hat t_n$. %In this section, 
Here we first show that \cond{t} in \eqref{eq:backtrack} \textit{always} holds when it is checked for $t=\hat t_n$ such that $l_t\geq L$ and \rv{$\mu_t\leq \mu$}. Next, building on this result, we argue that \cond{t} will hold after checking \algline{\ref{algline:check-agda+}} finitely many times if CASE 1b holds for the given iteration $t\in\mathbb{N}$. The results in this section depend on the value of \ms{} flag.}
\begin{lemma}\label{lemma:Delta-Bound-tn-true}
 \na{Suppose Assumption \ref{ASPT:lipshiz gradient} holds, and there exists $n\in\mathbb{N}$ such that $t=\hat t_n$ at the time when the backtracking condition \emph{\cond{t}} in \eqref{eq:backtrack} is checked, during the execution of \emph{\shortalgline{\ref{algline:check-agda+}}} in \cref{alg:agda}, for iteration $t$. If \emph{\msf{True}}, then it holds that $\delta_{t}\leq \Delta_{t}=\frac{2\zeta}{{\mu_{t}}}$; otherwise, for the case \emph{\msf{False}}, one has $\delta_{t}\leq \Delta_{t}=\min\Big\{\Big(1+\frac{2 L_{t}}{{\mu_{t}}}\Big) d_t,~\bar \cD_y\Big\}^2$ %for $n\geq 1$ 
 whenever $L_t\geq L$ and $\mu_t \in (0,\mu]$
 %and $\delta_{\hat t_n}\leq \Delta_{\hat t_n}=\frac{2\zeta}{\mu}$ for $n=0$, 
 where $d_t$ is defined in \emph{\algline{\ref{algeq:haty-d}}}, and $\bar \cD_y=\cD_y$ if $\cD_y<+\infty$ is known; otherwise, $\bar\cD_y$ is set to $+\infty$ within \agdap{}.} 
\end{lemma}
\begin{proof}
    \na{For the case \msf{True}, the proof of $\delta_t\leq\Delta_t$ directly follows from \cref{lemma:y0-bound} since according to \shortalgline{\ref{algeq:haty-max}} %and \shortalgline{\ref{algeq:haty-y}} 
    in \cref{alg:agda}, $\by_t$ satisfies $\cL(\bx_t,\by_t)+\zeta\geq \max_{\by\in\cY}\cL(\bx_t,\by)$ for $t=\hat t_n$ and we have $\Delta_{\hat t_n}=\frac{2\zeta}{{\mu_{t}}}$ according to \algline{\ref{algeq:delta_def1}} {and $\mu_t \in (0,\mu]$}.} 
    %Moreover, for $t=0$ when \msf{False}, the same result holds because $\by_0$ is set the same way also for this case.
    %that $\delta_t\leq \Delta_t = \frac{2\zeta}{\mu}$.
    
     \na{For the case \msf{False}, %when $t=\hat{t}_n$, i.e., $n>0$,
     according to \algline{\ref{algeq:delta_def2}},
     we have $\Delta_t=\min\{(1+\frac{2 L_{t}}{{\mu_{t}}}) d_t^2,~\bar \cD_y^2\}$ where $d_t=\norm{\hat\by_{t} - \by_{t}}$ is computed before we set $\by_t\gets\hat\by_t$ in \shortalgline{\ref{algeq:haty-y}}. Since $\hat\by_{t} = \prox{h/L_{t}} (\by_{t} + \frac{1}{L_{t}} \grad_y f(\bx_{t},\by_{t}))$, it follows from \cite[Eq.~(2.2.15)]{nesterov_convex} that  $\frac{{\mu_{t}}}{2}\norm{\by_{t} - \by^*(\bx_t)}^2\leq \langle{L_t(\hat\by_t-\by_t)}, \by_{t} - \by^*(\bx_t)\rangle$ holds whenever $L_t\geq L$ and {$\mu_t \in (0,\mu]$}, which further implies ${\norm{\by_{t} - \by^*(\bx_t)}} \leq {\frac{2 L_{t}}{{\mu_{t}}} \norm{\hat\by_{t} - \by_{t}}}$. Therefore, using triangular inequality, we get $\norm{\hat\by_t-\by^*(\bx_t)}\leq\norm{\hat\by_t-\by_t}+\norm{\by_t-\by^*(\bx_t)}\leq (1+\frac{2L_t}{{\mu_{t}}})d_t$; consequently, after resetting $\by_t$ to $\hat\by_t$ in \shortalgline{\ref{algeq:haty-y}}, we get the following $\delta_t$ bound: $\delta_t=\norm{\by_t-\by^*(\bx_t)}\leq (1+\frac{2L_t}{{\mu_{t}}})d_t$. 
     %On the other hand, recall that when $\cD_y$ is known, we set $\bar \cD_y = \cD_y$; otherwise, $\bar \cD_y=+\infty$. Thus, we also have
     Furthermore, $\delta_t\leq \bar\cD_y^2$ trivially holds from the definition of $\overline{D}_y$; hence, the proof is complete.} 
\end{proof}

\begin{lemma}\label{lemma:complicated-true-tn}
Suppose \cref{ASPT:lipshiz gradient} holds. For iteration $t\in\mathbb{N}$,
\na{if there exists $n\in\mathbb{N}$ such that $t=\hat t_n$ and $l_t\geq L$ and {$\mu_t \in (0,\mu]$} at the time when the backtracking condition \emph{\cond{t}} in \eqref{eq:backtrack} is checked, during the execution of \emph{\shortalgline{\ref{algline:check-agda+}}} in \cref{alg:agda}, then it holds that}
    \begin{equation}
    \label{eq:alternative_ls_appendix-true-part1-tn}
        \begin{aligned}
        \MoveEqLeft\Big(\tau_t - \frac{\tau_t^2 l_t}{2} - \frac{3\sigma_t \tau_t^2l_t^2}{2}\Big)\norm{G^{\tau_t}_x(\bx_t,\by_t)}^2 + {\sigma_t}\norm{G^{\sigma_t}_y(\bx_{t},\by_t)}^2\na{+\sigma_{t}^2\frac{{\mu_t}}{2}\left\|{G^{\sigma_{t}}_y(\tilde\bx_{t+1},\by_{t})}\right\|^2}\\
        \leq& \cL(\bx_t,\by_t) - \na{\cL(\tilde\bx_{t+1},\tilde\by_{t+1})} + 2 \Big(\frac{4(1-\sigma_t{\mu_t})}{\sigma_t}+2\sigma_tl_t^2\Big) \Delta_t + \na{\Lambda_t},
        \end{aligned}
    \end{equation}
    where $\Lambda_t=\zeta + L_t \norm{\by_t-\na{\tilde\by_{t+1}}} \sqrt{\frac{2\zeta}{{\mu_{t}}}}$ if \emph{\msf{True}}, and $\Lambda_t=2d_t L_t \norm{\by_t-\na{\tilde\by_{t+1}}}$ if \emph{\msf{False}}. \na{Therefore, for $\sigma_t=1/l_t$, \eqref{eq:alternative_ls_appendix-true-part1-tn} implies that \eqref{eq:ls-conv} holds with $R_t=0$.}
\end{lemma}
\begin{proof}
First, note that $L_{ \hat{t}_n}\geq L$ \na{(since $L_{\hat t_n}\geq l_{\hat t_n}$ always holds by construction of \agdap{} and $l_{\hat t_n}\geq L$ by the hypothesis).} \na{Suppose that \msf{True}. Under this setting, since $\by_t$ satisfies $\cL(\bx_t,\by_t)+\zeta\geq \max_{\by\in\cY}\cL(\bx_t,\by)$ for $t=\hat t_n$ together with $L_t\geq L$ and {$\mu_t \in (0,\mu]$},} %When $t= \hat{t}_{N(t)}$, 
it follows from \cref{lemma:y0-bound} that for $\Lambda_t=\zeta + L_t\norm{\by_{t}- \tilde\by_{t+1}}\sqrt{\frac{2\zeta}{{\mu_{t}}}}$, one has
\begin{equation}
\label{eq:h-subgradient}
     h(\by_{t})-h(\tilde\by_{t+1}) \leq \langle - \grad_y f(\bx_{t},\by_{t}) ,  \tilde\by_{t+1} - \by_{t} \rangle + \na{\Lambda_t}.
\end{equation}
On the other hand, for the case \msf{False}, the definition of $\hat\by_t$ implies that $\grad_y f(\bx_t,\by_t)-L_t(\hat\by_t-\by_t)\in\partial h(\hat\by_t)$; thus, using the convexity of $h$, we get %for all $\by\in\cY$ 
\begin{equation}
\label{eq:h-subgradient-2}
    \begin{aligned}
        h(\hat\by_t)-h(\by)
        &\leq\fprod{-\grad_y f(\bx_t,\by_t)+L_t(\hat\by_t-\by_t),~\by-\hat\by_t}\\
        &=\fprod{- \grad_y f(\bx_{t},\hat\by_{t}),~\by-\hat\by_{t}}+\fprod{\grad_y f(\bx_t,\hat\by_t)-\grad_y f(\bx_t,\by_t)+L_t(\hat\by_t-\by_t),~\by-\hat\by_t}\\
        &\leq \fprod{- \grad_y f(\bx_{t},\hat\by_{t}),~\by-\hat\by_{t}}+2 d_t L_t\norm{\by-\hat\by_t},\quad\forall~\by\in\cY,
    \end{aligned}
\end{equation}
where in the last inequality we used that $\grad_y f$ is Lipschitz with constant $L\leq L_t$ and $d_t\triangleq \norm{\hat\by_t-\by_t}$. Therefore, after resetting $\by_t$ to $\hat\by_t$ in \shortalgline{\ref{algeq:haty-y}}, substituting $\by=\tilde\by_{t+1}$ in \eqref{eq:h-subgradient-2} implies \eqref{eq:h-subgradient} with $\Lambda_t=2 d_t L_t\norm{\by_t-\tilde\by_{t+1}}$. 
%\xtodo{Does $\Lambda_t$ has a typo here? $\by$ should be $\tilde \by_{t+1}$?}
Thus, the effect of \ms{} shows up in $\Lambda_t$, and the rest of the analysis is common for both cases. 

Since $l_t\geq L$ \rv{and $\mu_t \in(0,\mu]$}, it follows from \cref{lemma:complicated-ls-hold} that \eqref{eq:ls-all} holds for $t=\hat t_n$. Thus, \na{using $f(\bx, \cdot)$ being $\mu$-strongly concave %w.r.t. $\by$ 
for any given $\bx \in \dom g$ {and $\mu_t \in (0,\mu]$},} we can further obtain that
\begin{equation}\label{eq:separate-gap-boundary0-part1-tn}
\begin{aligned}
\MoveEqLeft f\left(\tilde\bx_{t+1}, \tilde\by_{t+1}\right)-h(\tilde\by_{t+1})-f\left(\tilde\bx_{t+1}, \by_{t}\right)+h(\by_{t}) \\
\leq & \left\langle\nabla_y f\left(\tilde\bx_{t+1},\by_{t}\right)- \grad_y f(\bx_{t},\by_{t}),~\tilde\by_{t+1}-\by_{t}\right\rangle-\frac{{\mu_t}}{2}\left\|\tilde\by_{t+1}-\by_{t}\right\|^2 +\na{\Lambda_t}\\
%+\zeta + L\norm{\by_{t}- \tilde\by_{t+1}}\sqrt{\frac{2\zeta}{\xzf{\mu_t}}}\\
\leq & \frac{\sigma_{t}{l}_{t}^2}{2}\|\tilde\bx_{t+1} -\bx_{t}\|^2 + \frac{1}{2\sigma_{{t}}}\|\tilde\by_{t+1}-\by_{t}\|^2 {-\frac{{\mu_t}}{2}\norm{\tilde\by_{t+1}-\by_{t}}^2} +\na{\Lambda_t}
%+\zeta + L\norm{\by_{t}- \tilde\by_{t+1}}\sqrt{\frac{2\zeta}{\xzf{\mu_t}}}
\\
{=}  &\frac{\sigma_{t}{l}_{t}^2\tau^2_{t}}{2}\|G^{\tau_t}_x(\bx_{t},\by_{t})\|^2 + \frac{\sigma_{t}}{2}\|{G^{\sigma_{t}}_y(\tilde\bx_{t+1},\by_{t})}\|^2-\sigma_{t}^2\frac{{\mu_t}}{2}\left\|{G^{\sigma_{t}}_y(\tilde\bx_{t+1},\by_{t})}\right\|^2 +\na{\Lambda_t}
%+\zeta + L\norm{\by_{t}- \tilde\by_{t+1}}\sqrt{\frac{2\zeta}{\xzf{\mu_t}}}
\\
\leq  &\frac{{3}\sigma_{t}{l}_{t}^2\tau^2_{t}}{2}\|G^{\tau_t}_x(\bx_{t},\by_{t})\|^2 + \sigma_{t}\left(\frac{{4}(1-\sigma_t{\mu_t})}{\sigma_{t}^2}+{2}{l}_{t}^2\right)\Delta_{t}-\sigma_{t}^2\frac{{\mu_t}}{2}\left\|{G^{\sigma_{t}}_y(\tilde\bx_{t+1},\by_{t})}\right\|^2 +\na{\Lambda_t},
%+\zeta + \xzr{L_{t}}\norm{\by_{t}- \tilde\by_{t+1}}\sqrt{\frac{2\zeta}{\xzf{\mu_t}}},
\end{aligned}
\end{equation}
\xqs{where the second inequality is due to Young's inequality and \cref{eq:LS-2}, \na{the equality follows from Definition~\ref{def:prox},} and the last inequality is due to \cref{lemma:Delta-Bound-tn-true}, and \cref{eq:LS-xt+1yt}}. Then adding $\sigma_t\|G_y^{\sigma_{t}}(\bx_{t},\by_{t})\|$ to both sides of 
\cref{eq:separate-gap-boundary0-part1-tn}, and using \cref{eq:LS-xtyt,lemma:Delta-Bound-tn-true} yields that
\begin{equation*}
\begin{aligned}
\MoveEqLeft-{\frac{{3}\sigma_t{l}_{t}^2\tau_{t}^2}{2}}\|G_x^{\tau_t}(\bx_{t},\by_{t})\|^2+{\sigma_{t}\left\|G^{\sigma_{t}}_y(\bx_{t},\by_{t})\right\|^2}+\na{\sigma_{t}^2\frac{{\mu_t}}{2}\left\|{G^{\sigma_{t}}_y(\tilde\bx_{t+1},\by_{t})}\right\|^2}\\
\leq &f\left(\tilde\bx_{t+1}, \by_{t}\right)-f\left(\tilde\bx_{t+1}, \tilde\by_{t+1}\right) + h(\tilde\by_{t+1})  -h(\by_{t}) + {\sigma_t}\left(\frac{{8}(1-\sigma_t{\mu_t})}{\sigma_{t}^2}+{4}{l}_{t}^2\right)\Delta_{t}+\na{\Lambda_t}.
%\zeta +\xqs{L_t}\norm{\by_t -\tilde\by_{t+1}}\sqrt{\frac{2\zeta}{\xzf{\mu_t}}}.
\end{aligned}
\end{equation*}
\na{Finally, combining the above inequality with \cref{eq:descent_x_inner-part1} completes the proof.}
\end{proof}
\begin{theorem}\label{thm:ls-hold-main-tn}
    Under \cref{ASPT:lipshiz gradient}, \na{for iteration $t\in\mathbb{N}$, %the backtracking condition 
    \emph{\cond{t}} holds if there exists $n\in\mathbb{N}$ such that $t=\hat t_n$, $l_t\geq L$ and {$\mu_t \in (0,\mu]$} at the time when %\emph{\cond{t}} in 
    \eqref{eq:backtrack} is checked, during the execution of \emph{\shortalgline{\ref{algline:check-agda+}}} in \cref{alg:agda}.}
\end{theorem}
\begin{proof}
    \na{Since $t=\hat t_n$ for some $n\in\mathbb{N}$, $l_{\hat t_n}\geq L$ {and $\mu_t \in (0,\mu]$}, the conditions for {\cref{lemma:complicated-ls-hold,lemma:Delta-Bound-tn-true,lemma:complicated-true-tn}} are satisfied; therefore, \eqref{eq:ls-conv} follows from \cref{lemma:complicated-true-tn}, and \cref{eq:Lyy-1,eq:Lyy-2,eq:final-ls-xt+1yt} directly follow from \cref{lemma:complicated-ls-hold} and \cref{lemma:Delta-Bound-tn-true}.} 
    %and substituting $\sigma_t=\frac{1}{l_t}$ within \cref{eq:alternative_ls_appendix-true-part1-tn} implies \eqref{eq:ls-conv} for $t=\hat t_n$.
\end{proof}

\begin{corollary}
\label{cor:tn-finite}
    Suppose \cref{ASPT:lipshiz gradient} holds. \na{For $t\in\mathbb{N}$, if there exists $n\in\mathbb{N}$ such that $t=\hat t_n$, $L_t\geq L$ {and $\mu_t \in (0,\mu]$} at the time when the backtracking condition \emph{\cond{t}} in \eqref{eq:backtrack} is checked, then %\eqref{eq:backtrack} 
\emph{\cond{t}} will hold after executing \emph{\algline{\ref{algline:check-agda+}}} at most $\log_{1/\gamma}(L_t/l^\circ_t)+1$ many times for the given iteration $t\in\mathbb{N}$.}
\end{corollary}
\begin{proof}
    \na{First, note that every time \eqref{eq:backtrack} does not hold, according to \agdap{} one increases the value of $l_t$ using the update $l_t\gets l_t/\gamma_2$; second, since $\gamma_1=\gamma_2^r$ for some integer $r\geq 1$, and we initialize both $l_0$ and $L_0$ at the same number \rv{$\tilde l$ %={\tilde\mu}/\gamma_2$,
    such that $\tilde l>\tilde \mu$}, it follows that checking \eqref{eq:backtrack} can fail at most $\log_{1/\gamma_2}(L_t/l^\circ_t)$ 
    %many times at which point 
    --if that happens, one would have $l_t=L_t$. Finally, whenever $l_t=L_t$, since $L_t\geq L$ {and $\mu_t \in (0,\mu]$}, \cref{thm:ls-hold-main-tn} implies that \cond{t} must hold.}
\end{proof}

\subsubsection{\na{Analysis of CASE 2b: $t>\hat t_{N(t)}$ and $L_t\geq L$ and {$\mu_t \in (0,\mu]$}}}
%{\na{Induction proof for \cref{alg:agda} (Inductive case)}}
\na{Let $n\in\mathbb{N}$ be such that $N(t)=n$. According to the definition of $N(t)$, one has $L_t=L_s$ {and $\mu_t = \mu_s$} for $s\in [\hat t_n, t]$; hence, $L_t\geq L$ {and $\mu_t \leq \mu$} imply that $L_s\geq L$ {and $\mu_s \leq \mu$} for all $s\in [\hat t_n, t]$. Here we first show that \cond{t} in \eqref{eq:backtrack} always holds when it is checked for $t>\hat t_{N(t)}$ such that $l_t\geq L$ \rv{and $\mu_t \leq \mu$}. Next, building on this result, we argue that for the given iteration $t\in\mathbb{N}$ if CASE 2b is true, then \cond{t} will hold after checking \algline{\ref{algline:check-agda+}} finitely many times. To this aim, we first show that $\delta_t\leq \Delta_t$ whenever $L_t\geq L$ {and $\mu_t \in (0,\mu]$}.} 
%$\delta_s\leq \Delta_s$ for all $s\in[\hat t_n, t]$ using induction.
\begin{lemma}\label{lemma:Delta-Bound-t>tn-agda}
Under \cref{ASPT:lipshiz gradient}, \na{let $t\in\mathbb{N}$ such that $t>\hat t_{N(t)}$ and suppose that %\eqref{eq:backtrack} held for iteration $t-1$ 
\emph{\cond{t-1}} and $\delta_{t-1}\leq \Delta_{t-1}$ hold. If $L_t\geq L$ and {$\mu_t \in (0,\mu]$} at the time when
%the backtracking condition in \eqref{eq:backtrack} 
\emph{\cond{t}} is checked during the execution of \emph{\shortalgline{\ref{algline:check-agda+}}} in \cref{alg:agda} for iteration $t$, then 
%$\delta_{s}\leq \Delta_{s}$ for all $s\in[\hat t_{N(t)}, t]$
$\delta_t\leq \Delta_t$.}
\end{lemma}
\begin{proof}
\na{Note that since $L_t\geq L$, {$\mu_t \in (0,\mu]$} and %\eqref{eq:backtrack} held for iteration $t-1$
\cond{t-1} holds by the hypothesis, we can conclude that
    \begin{equation}
        \delta_{t}\leq B_{t-1}\delta_{t-1} + C_{t-1}\|G_x^{\tau_{t-1}}(\bx_{t-1},\by_{t-1})\|^2\leq  B_{t-1}\Delta_{t-1} + C_{t-1}\|G_x^{\tau_{t-1}}(\bx_{t-1},\by_{t-1})\|^2 \triangleq \Delta_{t},
    \end{equation}
    where the first inequality follows from \cref{lemma:deltat}, in the second inequality we use the assumption $\delta_{t-1}\leq \Delta_{t-1}$, and the equality holds by definition of $\Delta_t$ --see \algline{\ref{algeq:delta-update}}.} 
    %which completes the induction.}
\end{proof}

\begin{lemma}\label{lemma:complicated-t>tn-agda}
Suppose \cref{ASPT:lipshiz gradient} holds, \na{let $t\in\mathbb{N}$ such that $t>\hat t_{N(t)}$. Assuming that %\eqref{eq:backtrack} held for iteration $t-1$ 
\emph{\cond{t-1}} and $\delta_{t-1}\leq \Delta_{t-1}$ hold, if $l_t\geq L$ and {$\mu_t \in (0,\mu]$} at the time when %the backtracking condition in \eqref{eq:backtrack} 
\emph{\cond{t}} is checked during the execution of \emph{\shortalgline{\ref{algline:check-agda+}}} in \cref{alg:agda} for iteration $t$, then the following inequality holds:} 
%\xtodo{$t-1$ terms are ok to me here.}
\na{
\begin{equation}\label{eq:alternative_ls_appendix-true-part1-t>tn}
        \begin{aligned}
            \MoveEqLeft[2] \left(\tau_t - \frac{\tau_t^2 l_t}{2} - \Big(\frac{3}{2}+\frac{\sigma_t}{\sigma_{t-1}}\Big)\sigma_t\tau_t^2l_t^2   
            \right)\norm{G_x^{\tau_t}(\bx_t,\by_t)}^2 
            + \sigma_t\norm{G^{\sigma_t}_y(\bx_{t},\by_t)}^2
            +
            \sigma_{t}^2\frac{{\mu_t}}{2}\left\|{G^{\sigma_{t}}_y(\tilde\bx_{t+1},\by_{t})}\right\|^2 
            \\
            & +\sigma_{t-1}^2{\mu_{t-1}}\left\|G^{\sigma_{t-1}}_y(\bx_{t},\by_{t-1})\right\|^2-\left(1+\sigma^2_{t-1} l_{t-1}^2\right)\sigma_{t-1}\tau_{t-1}^2l_{t-1}^2\left\|G^{\tau_{t-1}}_x(\bx_{t-1},\by_{t-1})\right\|^2
            \\
        \leq & \cL(\bx_t,\by_t) - \na{\cL(\tilde\bx_{t+1},\tilde\by_{t+1})} +2\left(1+\sigma^2_{t-1} l_{t-1}^2\right)\left(\frac{2}{\sigma_{t-1}}+\sigma_{t-1}l^2_{t-1}-2{\mu_{t-1}}\right) \Delta_{t-1}
        \\
        & +2\Big(\Big(2+ \frac{\sigma_{t}}{\sigma_{t-1}} \Big)\Big(\frac{2}{\sigma_{t}}+\sigma_tl^2_{t}\Big)-4{\mu_t}\Big) \Delta_t;
        \end{aligned}
    \end{equation}}%
\na{therefore, for $\sigma_t=1/l_t$, \eqref{eq:alternative_ls_appendix-true-part1-t>tn} implies that \eqref{eq:ls-conv} holds with $\Lambda_t=6l_{t-1}(\Delta_{t}+2\Delta_{t-1})-8{\mu_{t-1}}\Delta_{t-1}$ as in \emph{\algline{\ref{algeq:Lambda-t>tn}}} and $R_t=2\tau_{t-1}^2l_{t-1}\left\|G^{\tau_{t-1}}_x(\bx_{t-1},\by_{t-1})\right\|^2-\sigma_{t-1}^2{\mu_{t-1}}\left\|{G^{\sigma_{t-1}}_y(\bx_{t},\by_{t-1})}\right\|^2$ as in \emph{\algline{\ref{algline:Rt}}}.}
%\na{Moreover, under these assumptions, \emph{\cond{t}} holds for $\sigma_t=1/l_t$ such that $l_t\geq L$.}
%\eqref{eq:alternative_ls_appendix-true-part1-t>tn} implies that \eqref{eq:backtrack} holds for iteration $t$.}
\end{lemma}

\begin{proof}
\na{Fix $t\in\mathbb{N}$ such that $t>\hat t_{N(t)}$ and $l_t\geq L$.} The update rule of $\tilde\by_{t}$ implies that $0\in \partial h(\tilde\by_{t}) + \frac{1}{\sigma_{t-1}}(\tilde\by_{t}-\by_{t-1})-\grad_y f(\tilde\bx_{t},\by_{t-1})$; moreover, the convexity of $h(\cdot)$ implies that 
\begin{equation}\label{eq:opt-condition-y-part1}
    h(\tilde\by_{t}) + \frac{1}{\sigma_{t-1}} \langle \by_{t-1}-\tilde\by_{t}+ \sigma_{t-1} \grad_y f(\tilde\bx_{t},\by_{t-1}), \by-\tilde\by_{t} \rangle \leq h(\by)
\end{equation}
\na{holds for all $\by\in\cY$.} Furthermore, %by replacing $t$ with $t-1$ and choosing $\by=\by_{t+1}$ in \cref{eq:opt-condition-y-part1}, 
\na{since %\eqref{eq:backtrack} held for iteration $t-1$
%\xtodo{typo here, $\tilde \bx_t$}
%\xtodo{also saying $t> \hat t_N(t)$, we did not recompute $\by_t$, so $\tilde \by_t =\by_t$?}
\cond{t-1} holds and $t> \hat t_{N(t)}$, we have $(\tilde\bx_t, \tilde\by_t)=(\bx_t,\by_t)$; hence, choosing $\by=\tilde\by_{t+1}$ in \cref{eq:opt-condition-y-part1},} we obtain
\begin{equation}
    \begin{aligned}
        &h(\by_{t})-h(\tilde\by_{t+1})\leq \langle - \grad_y f(\bx_t, \by_{t-1}) + \frac{1}{\sigma_{t-1}}(\by_t -  \by_{t-1}),~\tilde\by_{t+1}-\by_t \rangle.  \label{eq:opt-y-back}
        % \\
        % &h(\by_{t})-h(\by_{t+1})\leq \langle - \grad_y f(\bx_t, \by_{t-1}) + \frac{1}{\sigma_{t-1}}(\by_t -  \by_{t-1}), \by_{t+1}-\by_t \rangle, \quad \forall\; t = \hat{t}_n+1 \label{eq:opt-y-back-init}
\end{aligned}
\end{equation}
%holds for all  $\xzr{t\geq \hat{t}_{N(t)}+1}$.
%\begin{equation}
%0\leq \langle-\nabla_y f\left(\bx_{t}, \by_{t-1}\right)+\frac{1}{\sigma_{t-1}}\left(\by_t-\by_{t-1}\right), \by_{t+1}-\by_t\rangle,
%\end{equation}
% We will first provide an upper bound of \cref{eq:opt-y-back}, and then a similar bound of \cref{eq:opt-y-back-init} can be induced by the same analysis.
Since $f(\bx, \xzr{\cdot})$ is $\mu$-strongly concave for any given $\bx \in \dom g$ and {$\mu_t \in (0,\mu]$}, using \cref{eq:opt-y-back} we get
\begin{equation}\label{eq:separate-gap-part1}
\begin{aligned}
\MoveEqLeft[2] f\left(\tilde\bx_{t+1}, \tilde\by_{t+1}\right)-h(\tilde\by_{t+1})-f\left(\tilde\bx_{t+1}, \by_t\right)+h(\by_t) \\
\leq & \left\langle\nabla_y f\left(\tilde\bx_{t+1}, \by_t\right), \tilde\by_{t+1}-\by_t\right\rangle-\frac{{\mu_t}}{2}\left\|\tilde\by_{t+1}-\by_t\right\|^2 +\langle - \grad_y f(\bx_t,\by_{t-1}) + \frac{1}{\sigma_{t-1}}(\by_t - \by_{t-1}),~\tilde\by_{t+1}-\by_t \rangle\\
= &
\left\langle\nabla_y f\left(\tilde\bx_{t+1}, \by_t\right)-\nabla_y f\left(\bx_t, \by_{t-1}\right),~\tilde\by_{t+1}-\by_t\right\rangle + \frac{1}{\sigma_{t-1}} \langle \by_t-\by_{t-1},~\tilde\by_{t+1}-\by_t \rangle -\frac{{\mu_t}}{2}\left\|\tilde\by_{t+1}-\by_t\right\|^2.
\end{aligned}
\end{equation}
%Because of our linesearch condition \cref{eq:LS-2} and Young's inequality, the first inner product can be bounded by
%\begin{equation}\label{eq:separate-gap-p1}
    %\begin{aligned}
%& \left\langle\nabla_y f\left(\bx_{t+1}, \by_t\right)-\nabla_y f\left(\bx_{t}, \by_t\right), \by_{t+1}-\by_t\right\rangle \\
%\leq & \frac{\sigma_t L_{t,2}^2 }{2}\left\|\bx_{t+1}-\bx_{t}\right\|^2+\frac{1}{2\sigma_t }\left\|\by_{t+1}-\by_t\right\|^2,
%\end{aligned}
%\end{equation}
%holds for all \xzr{$t\geq \hat{t}_{N(t)}+1$}. 
Denoting $\bv_{t+1}\triangleq\left(\tilde\by_{t+1}-\by_t\right)-\left(\by_t-\by_{t-1}\right)$, we can write the first inner product term on the r.h.s. of \cref{eq:separate-gap-part1} as
\begin{equation}\label{eq:separate-gap-p2-part1}
\begin{aligned}
\MoveEqLeft[2] \left\langle\nabla_y f\left(\tilde\bx_{t+1}, \by_t\right)-\nabla_y f\left(\bx_{t}, \by_{t-1}\right),~\tilde\by_{t+1}-\by_t\right\rangle \\
= & \left\langle\nabla_y f\left(\tilde\bx_{t+1}, \by_t\right)-\nabla_y f\left(\bx_{t}, \by_t\right),~\tilde\by_{t+1}-\by_t\right\rangle +\left\langle\nabla_y f\left(\bx_{t}, \by_t\right)-\nabla_y f\left(\bx_{t}, \by_{t-1}\right), \bv_{t+1}\right\rangle \\
& +\left\langle\nabla_y f\left(\bx_{t}, \by_t\right)-\nabla_y f\left(\bx_{t}, \by_{t-1}\right), \by_t-\by_{t-1}\right\rangle.
\end{aligned}
\end{equation}
Next, we bound the three terms on the r.h.s. of \eqref{eq:separate-gap-p2-part1}. \na{Since $l_t\geq L$, \cref{lemma:complicated-ls-hold} implies that \cref{eq:LS-2} holds; therefore, using Young's inequality,} we get
\begin{equation}\label{eq:separate-gap-p3-part1}
    \begin{aligned}
& \left\langle\nabla_y f\left(\tilde\bx_{t+1}, \by_t\right)-\nabla_y f\left(\bx_{t}, \by_t\right),~\tilde\by_{t+1}-\by_t\right\rangle 
\leq \frac{\sigma_{t} l_t^2 }{2}\left\|\tilde\bx_{t+1}-\bx_{t}\right\|^2+\frac{1}{2\sigma_{t}}\left\|\tilde\by_{t+1}-\by_t\right\|^2.
\end{aligned}
\end{equation}
%\xzr{Because \cref{alg:agda} is checking at \shortalgline{\ref{algline:check-agda+}} at time step $t$,   
\na{Moreover, since %\eqref{eq:backtrack} held for iteration $t-1$
%\xtodo{typo here}
\cond{t-1} holds and $t> \hat t_{N(t)}$, we have $(\tilde\bx_t, \tilde\by_t)=(\bx_t,\by_t)$, and using \cref{eq:Lyy-2} for $t-1$ together with Young's inequality, we get}
\begin{equation}\label{eq:separate-gap-p4-part1}
\left\langle\nabla_y f\left(\bx_{t}, \by_t\right)-\nabla_y f\left(\bx_{t}, \by_{t-1}\right), \bv_{t+1}\right\rangle \leq \frac{\sigma_{t-1} l_{t-1}^2}{2}\left\|\by_t-\by_{t-1}\right\|^2+\frac{1}{2\sigma_{t-1}}\left\|\bv_{t+1}\right\|^2 .
\end{equation}
For the third term, \na{strong-concavity of $f(\bx_t,\cdot)$ and $\rv{\mu_{t-1}=\mu_{t}\in(0,\mu]}$ implies that}
\begin{equation}\label{eq:separate-gap-p5-part1}
\left\langle\nabla_y f\left(\bx_{t}, \by_t\right)-\nabla_y f\left(\bx_{t},~\by_{t-1}\right), \by_t-\by_{t-1}\right\rangle \leq-{\mu_{t-1}}\left\|\by_t-\by_{t-1}\right\|^2 .
\end{equation}
Moreover, it can be easily checked that
\begin{equation}\label{eq:separate-gap-p6-part1}
\left\langle \by_t-\by_{t-1}, \tilde\by_{t+1}-\by_t\right\rangle=\frac{1}{2}\left\|\by_t-\by_{t-1}\right\|^2+\frac{1}{2}\left\|\tilde\by_{t+1}-\by_t\right\|^2-\frac{1}{2}\left\|\bv_{t+1}\right\|^2 .
\end{equation}
Plugging Equations \eqref{eq:separate-gap-p3-part1}~-~\eqref{eq:separate-gap-p6-part1} into \cref{eq:separate-gap-part1} and rearranging the terms, we conclude that
{\small
$$
\begin{aligned}
\MoveEqLeft f\left(\tilde\bx_{t+1}, \tilde\by_{t+1}\right)-h(\tilde\by_{t+1})-f\left(\tilde\bx_{t+1}, \by_t\right) + h(\by_t) \\
\leq &   \frac{\sigma_t l_t^2 }{2}\left\|\tilde\bx_{t+1}-\bx_{t}\right\|^2
-\left({\mu_{t-1}}-\frac{1}{2 \sigma_{t-1}}-\frac{\sigma_{t-1} l_{t-1}^2}{2}\right)\left\|\by_t-\by_{t-1}\right\|^2 -\left(\frac{{\mu_t}}{2}-\frac{1}{2\sigma_{t}}-\frac{1}{2\sigma_{t-1}}\right)\left\|\tilde\by_{t+1}-\by_t\right\|^2
\\
% = & \frac{l_t^2 \tau_{t}^2\sigma_{t}}{2}\left\|G^{\na{\tau_t}}_x(\bx_{t},\by_{t})\right\|^2
% -\sigma_{t-1}^2\left(\xzf{\mu_{t-1}}-\frac{1}{2 \sigma_{t-1}}-\frac{\sigma_{t-1} l_{t-1}^2}{2}\right)\left\|\xzr{G^{\sigma_{t-1}}_y(\bx_{t},\by_{t-1})}\right\|^2 -\sigma_{t}^2\left(\frac{\xzf{\mu_{t-1}}}{2}-\frac{1}{2\sigma_{t}}-\frac{1}{2\sigma_{t-1}}\right)\left\|\xzr{G^{\sigma_{t}}_y(\tilde\bx_{t+1},\by_{t})}\right\|^2
% \\
=
&  \frac{\sigma_{t}l_{t}^2 \tau_{t}^2}{2}\left\|G^{\tau_t}_x(\bx_{t},\by_{t})\right\|^2
 +\left(\frac{\sigma_{t-1}}{2}+\frac{\sigma_{t-1}^3 l_{t-1}^2}{2}\right)\left\|{G^{\sigma_{t-1}}_y(\bx_{t},\by_{t-1})}\right\|^2  +\Big(\frac{\sigma_t}{2}+ \frac{\sigma_{t}^2}{2\sigma_{t-1}} \Big)\left\|{G^{\sigma_{t}}_y(\tilde\bx_{t+1},\by_{t})}\right\|^2
\\
&
 -\sigma_{t-1}^2{\mu_{t-1}}\left\|{G^{\sigma_{t-1}}_y(\bx_{t},\by_{t-1})}\right\|^2  -\sigma_{t}^2\frac{{\mu_t}}{2}\left\|{G^{\sigma_{t}}_y(\tilde\bx_{t+1},\by_{t})}\right\|^2
 +
 \sigma_{t}\left\|{G^{\sigma_{t}}_y(\bx_{t},\by_{t})}\right\|^2
 -
 \sigma_{t}\left\|{G^{\sigma_{t}}_y(\bx_{t},\by_{t})}\right\|^2
 \\
 \leq & \frac{\sigma_{t}l_{t}^2 \tau_{t}^2}{2}\left\|G^{\na{\tau_t}}_x(\bx_{t},\by_{t})\right\|^2
+\left(\sigma_{t-1}+\sigma^3_{t-1} l_{t-1}^2\right)\left(\frac{4(1-\sigma_{t-1}{\mu_{t-1}})}{\sigma^2_{t-1}}+2l^2_{t-1}\right) \Delta_{t-1}  +\Big({2\sigma_t}+ \frac{\sigma_{t}^2}{\sigma_{t-1}} \Big)\Big(\frac{4(1-\sigma_t{\mu_t})}{\sigma_{t}^2}+2l^2_{t}\Big) \Delta_t
\\
&
+ {\left(\sigma_{t-1}+\sigma_{t-1}^3 l_{t-1}^2\right)\cdot \tau_{t-1}^2l_{t-1}^2\left\|G^{\tau_{t-1}}_x(\bx_{t-1},\by_{t-1})\right\|^2}
 + {\Big(\sigma_t+ \frac{\sigma_{t}^2}{\sigma_{t-1}} \Big)\cdot \tau_{t}^2l_{t}^2\left\|G^{\tau_{t}}_x(\bx_{t},\by_{t})\right\|^2}
 \\
 &
  -\sigma_{t-1}^2{\mu_{t-1}}\left\|{G^{\sigma_{t-1}}_y(\bx_{t},\by_{t-1})}\right\|^2  -\sigma_{t}^2\frac{{\mu_t}}{2}\left\|{G^{\sigma_{t}}_y(\tilde\bx_{t+1},\by_{t})}\right\|^2
 {-
\sigma_{t}\left\|{G^{\sigma_{t}}_y(\bx_{t},\by_{t})}\right\|^2},
% \\
% &
%  %-\sigma_{t-1}^2\xzf{\mu_t}\left\|G_y(\bx_{t-1},\by_{t-1})\right\|^2
%  -\sigma_{t}^2\frac{\xzf{\mu_t}}{2}\left\|G_y(\bx_{t},\by_{t})\right\|^2
\end{aligned}
$$}%
%holds for all \xzr{$t\geq  \hat{t}_{N(t)}+1$}, 
where the last inequality follows from \na{(i) $\delta_{t-1}\leq\Delta_{t-1}$ (by hypothesis) and $\delta_{t}\leq\Delta_{t}$ (by \cref{lemma:Delta-Bound-t>tn-agda}), (ii) 
\cref{eq:LS-xt+1yt,eq:LS-xtyt}, and (iii) the fact that \cref{eq:final-ls-xt+1yt} holds for iteration $t-1$ (by hypothesis)}. After the rearranging terms \na{and dropping a non-positive term on the r.h.s., we get}
\begin{equation}
\begin{aligned}
\MoveEqLeft[2]\sigma_{t-1}^2{\mu_{t-1}}\left\|{G^{\sigma_{t-1}}_y(\bx_{t},\by_{t-1})}\right\|^2  
+
\sigma_{t}^2\frac{{\mu_t}}{2}\left\|{G^{\sigma_{t}}_y(\tilde\bx_{t+1},\by_{t})}\right\|^2
+
{\sigma_{t}\left\|G^{\sigma_t}_y(\bx_{t},\by_{t})\right\|^2  }\\
\leq& -f\left(\tilde\bx_{t+1}, \tilde\by_{t+1}\right) + h(\tilde\by_{t+1}) +f\left(\tilde\bx_{t+1}, \by_t\right) - h(\by_t)
\\
& 
+2\left(1+\sigma^2_{t-1} l_{t-1}^2\right)\left(\frac{2(1-\sigma_{t-1}{\mu_{t-1}})}{\sigma_{t-1}}+\sigma_{t-1}l^2_{t-1}\right) \Delta_{t-1}  
+
\na{2\Big(\Big(2+ \frac{\sigma_{t}}{\sigma_{t-1}} \Big)\Big(\frac{2
%(1-\sigma_t\xzf{\mu_t})
}{\sigma_{t}}+\sigma_tl^2_{t}\Big)-4{\mu_t}\Big) \Delta_t}
\\
&
+ {\left(1+\sigma^2_{t-1} l_{t-1}^2\right)\sigma_{t-1}\tau_{t-1}^2l_{t-1}^2\left\|G^{\tau_{t-1}}_x(\bx_{t-1},\by_{t-1})\right\|^2}
+ {\left(\frac{3}{2} + \frac{\sigma_{t}}{\sigma_{t-1}} \right)\sigma_{t}\tau_{t}^2l_{t}^2\left\|G^{\tau_{t}}_x(\bx_{t},\by_{t})\right\|^2}.
\end{aligned}
\end{equation}
%holds for all \xzr{$t\geq  \hat{t}_{N(t)}+1$}. 
Thus, using \cref{eq:descent_x_inner-part1} completes the proof \na{of \cref{eq:alternative_ls_appendix-true-part1-t>tn}}. \na{Furthermore, setting $\sigma_t=1/l_t$ within \eqref{eq:alternative_ls_appendix-true-part1-t>tn} immediately implies \eqref{eq:ls-conv} since $\frac{\sigma_t}{\sigma_{t-1}}=\frac{l_{t-1}}{l_t}\leq \frac{1}{\gamma}$, which follows from \algline{\ref{algeq:l-update}}.}
\end{proof}
\begin{theorem}\label{thm:ls-hold-main-t>tn}
    Suppose \cref{ASPT:lipshiz gradient} holds. \na{For iteration $t\in\mathbb{N}$ such that $t > \hat t_{N(t)}$, $L_t\geq L$ {and $\mu_t \in (0,\mu]$}, both \emph{\cond{t}} and $\delta_t\leq \Delta_t$ will hold after executing \emph{\algline{\ref{algline:check-agda+}}} at most $\log_{1/\gamma}(L_t/l^\circ_t)+1$ many times for the given iteration $t\in\mathbb{N}$.} 
\end{theorem}
%\xtodo{do we need to mention how to get $\log_{1/\gamma}(L_t/l^\circ_t)+1$ from $l_t\leq l$?}
\begin{proof}
    \na{Given $t\in\mathbb{N}$ such that $t>\hat t_{N(t)}$, $L_t\geq L$ {and $\mu_t \in (0,\mu]$}, let $n\in\mathbb{N}$ such that $n=N(t)$. Since $L_t=L_s$ {and $\mu_t = \mu_s$} for all $s\in[\hat t_n, t]$, it follows from $L_t\geq L$ {and $\mu_t \in (0,\mu]$} that $L_s\geq L$ {and $\mu_s \in (0,\mu]$} for all $s\in[\hat t_n, t]$. Using this fact, the rest of the proof directly follows from an induction argument. Indeed, since $L_{\hat t_n}\geq L$ {and $\mu_{\hat t_n } \in (0,\mu]$}, \cref{cor:tn-finite} implies that for iteration $\hat t_n$, \cond{\hat t_n} must hold within $\log_{1/\gamma}(L_{\hat t_n}/l_{\hat t_n}^\circ)+1$ times checking \eqref{eq:backtrack}; %according to \cref{thm:ls-hold-main-tn};
    moreover, according to \cref{lemma:Delta-Bound-tn-true}, we know that $\delta_{\hat t_n}\leq \Delta_{\hat t_n}$ holds. With %the former argument
    these two conclusions, we have established the base case for the induction proof, i.e., for $s=\hat t_n$.} \na{Next, we discuss the inductive step of the proof, i.e., for any iteration $s\in [\hat t_n+1, t]$, given that \cond{s-1} and $\delta_{s-1}\leq\Delta_{s-1}$ holds, we would like to show that \cond{s} and $\delta_s\leq\Delta_s$ will hold after a finite number of times executing {\shortalgline{\ref{algline:check-agda+}}} in \cref{alg:agda}. Since $L_s\geq L$ {and $\mu_s \in (0,\mu]$}, \cref{lemma:Delta-Bound-t>tn-agda} immediately implies that $\delta_s\leq \Delta_s$. We next show that \cond{s} holds. 
    %via using two arguments: (1) 
    Indeed, similar to the proof of \cref{cor:tn-finite}, checking \eqref{eq:backtrack} for iteration $s$ can fail at most $\log_{1/\gamma_2}(L_s/l^\circ_s)$ many times at which point one would have $l_s=L_s$; finally, whenever $l_s=L_s$, since $L_s\geq L$ {and $\mu_s \in (0,\mu]$} and within \agdap{} we set $\sigma_s=1/l_s$ for all $s\in[\hat t_n, t]$, \cref{lemma:complicated-ls-hold,lemma:complicated-t>tn-agda} imply that \cond{s} must hold, which completes the proof of the inductive step.}
    %\cref{lemma:complicated-ls-hold}, and using $\sigma_t=\frac{1}{l_t}$ in within \cref{eq:alternative_ls_appendix-true-part1-t>tn}.
\end{proof}
\begin{corollary}\label{cor:ls-hold-main-t>tn}
    Suppose \cref{ASPT:lipshiz gradient} holds. \na{For iteration $t\in\mathbb{N}$ such that $t > \hat t_{N(t)}$, if $l_t\geq L$ {and $\mu_t \in (0,\mu]$} at the time when the backtracking condition %\emph{\cond{t}} 
    in \eqref{eq:backtrack} is checked, during the execution of \emph{\shortalgline{\ref{algline:check-agda+}}} in \cref{alg:agda}, then both \emph{\cond{t}} and $\delta_t\leq \Delta_t$  hold for the given iteration $t\in\mathbb{N}$.} 
\end{corollary}
\begin{proof}
    \na{$l_t\leq L_t$ and $l_t\geq L$ imply %that 
    $L_t\geq L$; hence, according to the proof of \cref{thm:ls-hold-main-t>tn}, it follows that %both 
    $\delta_{t-1}\leq \Delta_{t-1}$ and {\cond{t-1}} must be true simultaneously 
    %at the time 
    %the backtracking condition 
    when
    {\cond{t}} in \eqref{eq:backtrack} is checked for iteration $t$. Thus, \cref{lemma:complicated-ls-hold,lemma:complicated-t>tn-agda} %together 
    imply that \cond{t} holds, and \cref{lemma:Delta-Bound-t>tn-agda} implies that $\delta_t\leq \Delta_t$ holds.}
\end{proof}
\paragraph{Proof of \cref{thm:backtrack}}
In \cref{cor:tn-finite} and \cref{thm:ls-hold-main-t>tn} we have established that \cond{t} holds in finite time for all $t\in\mathbb{N}$ such that {CASE 1b} 
and {CASE 2b} are true, respectively. We are now ready to complete the proof of \cref{thm:backtrack} by showing that for any iteration $t\in\mathbb{N}$ belonging to either CASE 1a or CASE 2a, the algorithm cannot stay in that case for an infinite amount of time within the given iteration, i.e., for the given iteration $t$ that is in either CASE 1a or CASE 2a, after checking \eqref{eq:backtrack} finitely many times, either \cond{t} holds or one switches to CASE 1b or CASE 2b for the iteration $t$. 

\na{Given $t\in\mathbb{N}$, suppose that CASE 1a holds. Since every time \eqref{eq:backtrack} does not hold, the value of $l_t$ is increased using $l_t\gets l_t/\gamma_2$ according to the update rule in \agdap{}; it follows that checking \eqref{eq:backtrack} can fail at most \rv{$\lceil \max\{\log_{1/\gamma_2}(L/l_t),\log_{1/\gamma_1}(\mu_t/\mu)\}\rceil$} %\nsa{Check if it is ceiling.} 
many times before one switches to CASE 1b, i.e., one would have {$\mu_t \leq \mu$} and $l_t\geq L$ implying that $L_t\geq L$. Now suppose that CASE 2a holds for the given $t\in\mathbb{N}$. We next argue that within $\log_{1/\gamma_2}(L_t/l_t)+1$ many times \cond{t} is checked, either \eqref{eq:backtrack} holds or one moves to either CASE 1 for the given iteration $t$. First, note that every time \eqref{eq:backtrack} does not hold, according to \agdap{} one increases the value of $l_t$ using the update $l_t\gets l_t/\gamma_2$; second, since $\gamma_1=\gamma_2^r$ for some integer $r\geq 1$, and we initialize both $l_0$ and $L_0$ at %the same 
\rv{some given number $\tilde l$ %$={\tilde\mu}/\gamma_2$, 
such that $\tilde l>\tilde \mu$,} it follows that checking \eqref{eq:backtrack} can fail at most $\log_{1/\gamma_2}(L_t/l_t)$ many times at which point one would have $l_t=L_t$, and the next time it fails one switches to CASE 1 as the condition $l_t\leq L_t$ would not hold and one would increase $L_t$ according to the update rule $L_t\gets L_t/\gamma_1$. This result completes the proof since we have just showed that for $t$ such that CASE 1a is true, either \cond{t} holds after checking \eqref{eq:backtrack} finitely many times, or one switches to CASE 1b.}

\section{Proof of \cref{thm:avg--backtrack-time}}
\label{sec:backtrack-bound}
\na{Here we bound the average number of backtracking steps per iteration $t\in\mathbb{N}$.} \rv{Recall that $\tilde\mu$ is the initial estimate for $\mu$ and $\tilde l$ such that $\tilde l>\tilde \mu$ is the initial estimate for the local Lipschitz constant at $(\bx_0,\by_0)$, e.g., one can set $\tilde l=\tilde\mu/\gamma_2$.}
%represents the set containing all $ \hat{t}_n$ and $k_n$, respectively. We use $|\cdot|$ to denote the cardinality of the given set.
%\xtodo{max value of $k_n$ need to be modified accordingly if def of $k_n$ changes.}
\begin{lemma}\label{cor:bound-tn-kn}
     Suppose $\{\mu_t,l_t,L_t\}_{t\geq 0}$ are generated by \agdap{}, stated in \cref{alg:agda}, and
     %or \cref{alg:sagda+}.
     \na{let $\bar{N}\in \N$ and $\cT,\cK\subset\N$ be defined as in Definitions~\ref{def:N-k} and~\ref{def:kn}.}
     Then
    {$l_t\leq \bar L/\gamma_2$, $L_t\leq {\bar{L}/\gamma_1}$ and $\mu_t \geq \gamma_1\bar\mu $ for all $t\in\N$, where $\bar L  \triangleq 
    %\max\{L, \frac{{\tilde{\mu}}}{\mu}\tilde{l}\}
    \rv{\cR\;\tilde l}$ and $\bar \mu \triangleq \max\{\underline{\mu}/\gamma_1,~\rv{\tilde\mu/\cR}
    %\min\{\mu, \frac{\tilde l}{L}{\tilde{\mu}}\}
    \}$} and $\rv{\cR \triangleq \max\{\tilde\mu/\mu,~L/\tilde l,~1\}}$;  moreover, %$|\{ \hat{t}_n\}| 
    $|\cT| = %|\{k_n\}|
    |\cK|=\bar N\leq 
    %\lfloor \log_{1/\gamma}(\kappa)\rfloor+r
    \rv{\left\lceil \frac{1}{r}\log_{1/\gamma}(\rv{\cR})\right\rceil}$, and $\max_{\na{n\in [\bar N]}}\{k_n\}\leq 
    \rv{\left\lceil \frac{1}{r}\log_{1/\gamma}(\rv{\cR})\right\rceil}$, where $|\cdot|$ to denote the cardinality of the given set.
    %\nsa{Previously the bound was $\lceil\log_{1/\gamma_1}(\gamma_2\kappa)\rceil +1$} 
\end{lemma}
\begin{remark}
    {\rv{When $\mu$ is known, after setting $\tilde\mu=\mu$ and $\underline{\mu}=\mu$,} the definition of $\bar\mu$ implies that 
    $\bar{\mu} = \mu/\rv{\gamma_1}$;
    otherwise, we set $\underline{\mu}=0$, which implies that
    $\bar{\mu} = \min\left\{\mu,~\tilde\mu \frac{\tilde l}{L}\right\}$.
    }
\end{remark}

\begin{proof}
    \rv{According to \agdap{} in \cref{alg:agda}, based on the update rule for $L_t$ {and $\mu_t$}, i.e., $L_t\gets L_t/\gamma_1$ {and $\mu_t \gets \mu_t \gamma_1$}, one can bound $|\cT|=\bar N$ by $\lceil \log_{1/\gamma_1}(\cR)\rceil$; indeed, since $L_{\hat t_n}=\tilde l /\gamma_1^{k_n}$ and $\mu_{\hat t_n}=\max\{\tilde \mu \gamma_1^{k_n},~\underline{\mu}\}$, we can conlcude that $L_{\hat t_n}\geq L$ and $\mu_{\hat t_n}\leq \mu$ for $k_n\in\integers_+$ such that $1/\gamma_1^{k_n}\geq \cR$. Thus, $k_n\leq \lceil \log_{1/\gamma_1}(\cR)\rceil$ for all $n\in[\bar N]$, and this also implies that $\bar N\leq \lceil \log_{1/\gamma_1}(\cR)\rceil$.
    %thus, we have $\bar N\leq \left\lceil \frac{\log_{1/\gamma}(R)}{r}\right\rceil$, which also implies that $\max_{\na{n\in [\bar N]}}\{k_n\}\leq  \left\lceil \frac{\log_{1/\gamma}(R)}{r}\right\rceil$.
    Finally, because  $L_t=L_{\hat{t}_{N(t)}}=\tilde l/\gamma_1^{k_{N(t)}}$ and $\mu_t=\mu_{\hat{t}_{N(t)}}=\max\{\gamma_1^{k_{N(t)}}\tilde \mu,~\underline{\mu}\}$, 
    we can conclude that $L_t\leq \tilde l/\gamma_1^{\lceil \log_{1/\gamma_1}(\cR)\rceil}\leq  \cR\; \tilde l/\gamma_1=\bar{L}/\gamma_1$ and $\mu_t\geq \max\{\tilde\mu \gamma_1^{\lceil \log_{1/\gamma_1}(\cR)\rceil}, \underline{\mu}\}\geq \max\{\gamma_1\tilde \mu/\cR,\underline{\mu}\}=\gamma_1\bar{\mu}$ for all $t\geq 0$. {From the definition of $\bar L$ and $\bar\mu$, we have $L\leq \bar L$ and $\mu\geq \bar\mu$; hence,} when {$L_t\in [\bar{L},\bar{L}/\gamma_1]\subset [L, \bar L/\gamma_1]$ (at the same we also have $\mu_t \in [\gamma_1\bar{\mu},\bar{\mu}]\subset [\gamma_1\bar\mu, \mu]$ since $L_t$ and $\mu_t$ change together)}, \agdap{} stated in \cref{alg:agda} %or \cref{alg:sagda+} 
    can always find $l_t$ such that ${l}_t\in [\bar{L},L_t]$ and $l_t\leq {\bar L}/\gamma_2$ because \na{$\gamma_1=\gamma^r\leq \gamma=\gamma_2$ for some $\gamma\in(0,1)$ and $r\in\integers_+$.} At this point, one does not need to increase $l_t$ anymore since we simultaneously have $l_t\geq L$ and $\mu_t\leq \mu$. Indeed, it follows from \cref{thm:ls-hold-main-tn,cor:ls-hold-main-t>tn}
    that for any given $t\in\N$, the condition in \eqref{eq:backtrack} always holds whenever ${l}_t\geq L$ {and $\mu_t \in (0,\mu]$}, and within \agdap{} whenever \cond{t} holds,  one sets $L_{t+1}=L_t$, {$\mu_{t+1} = \mu_t$} and the iteration counter is incremented, i.e., $t\gets t+1$.} 
\end{proof}

%\xtodo{For this part, this bound may be optimized. I mean the $\lceil$ or $\rceil$. }

\begin{corollary}\label{cor:bound-NK}
        $N(t)\leq \rv{\left\lceil \frac{1}{r}\log_{1/\gamma}(\cR)\right\rceil}$ for all $t\geq 0$. 
\end{corollary}
\begin{proof}
    The proof directly follows from \cref{cor:bound-tn-kn}.
\end{proof}

\paragraph{Proof of \cref{thm:avg--backtrack-time}}
\label{pf:thm:avg--backtrack-time}
The bound on the number of \rv{times $L_t$ increases and $\mu_t$ decreases simultaneously} %immediately 
follows from \cref{cor:bound-tn-kn}. Next, we show the bound on the average number of backtracking checks per iteration. \na{For any %iteration 
$t\in\N$, recall that according to \agdap, every time \cond{t} fails to hold one increases the local estimate using $l_t\gets l_t/\gamma_2$; hence, the number of times \cond{t} is checked equals to $1+\log_{1/\gamma_2}(l_t/l^\circ_t)$, where $l_t$ denotes the value of the local constant estimate at the time \cond{t} holds and $l^\circ_t$ denotes the value of the local constant estimate at the beginning of iteration $t$, i.e., before \cond{t} is checked for the first time for iteration $t$. According to \agdap{}, $l^\circ_0=\tilde l$
%\triangleq \qsr{\tilde\mu}/\gamma_2$ 
and for all $t\in\N$, one has $l^\circ_{t+1}=\max\{\gamma_2 l_t, \tilde l\}$. Therefore, for any $T\in\N$, the average number of times per iteration \agdap{} checks the backtracking condition in~\eqref{eq:backtrack} by the end of iteration $T$ can be bounded by 
    \begin{equation}
    \begin{aligned}
                \frac{1}{T+1}\sum_{t=0}^{T} {I_t} = \frac{1}{T+1}\sum_{t=0}^{T}\Big(1+\log_{1/\gamma_2}\Big(\tfrac{l_{t}}{l^\circ_t}\Big)\Big)
                \leq \frac{1}{T+1}\sum_{t=0}^{T}\Big(1+\log_{1/\gamma_2}\Big(\tfrac{l_{t}}{\gamma_2 l_{t-1}}\Big)\Big),
    \end{aligned}
    \end{equation}
    with the convention that $l_{-1}\triangleq \tilde l/\gamma_2$ (since $l^\circ_0=\tilde l$); hence,
 %Indeed, since  $N(T)\leq  \lceil\log_{1/\gamma_1}(\kappa)\rceil$ for any $T$, $L_t$ will only change finite times.
\begin{equation}
 \begin{aligned}
              \frac{1}{T+1}\sum_{t=0}^{T} I_t &\leq
                %&   \frac{1}{T+1}\sum_{t=0}^{T}(1+\log_{1/\gamma_2}(\tfrac{l_{t}}{\gamma_2 l_{t-1}})) \\
               %=
               1 + \frac{1}{T+1} \log_{1/\gamma_2}\Big(\frac{1}{\gamma_2^T} \frac{l_T}{\tilde l}\Big) 
                 \\
                 &=1+ \frac{T}{T+1}+ \frac{1}{T+1} \log_{1/\gamma_2}\Big(\frac{l_T}{\tilde l}\Big) \leq 2+\frac{\log_{1/\gamma}(\bar L/\tilde l)}{T+1},
 \end{aligned}
 \end{equation}
 %where in the last inequality we used 
 which follows from $
 %\qsr{\tilde l = \tilde\mu/\gamma_2}\leq 
 l_t\leq {\bar L}/\gamma_2$ for all $t\geq 0$ and $\gamma_2=\gamma\in(0,1)$. \rv{Finally, the desired result is due to $\cR=\bar L/\tilde l$.}} 
 % Suppose $T \geq 1+ \log_{1/\gamma_2}(\gamma_1\kappa)$, we can obtain that
 % \begin{equation}
 % \begin{aligned}
 %              \frac{1}{T+1}\sum_{t=0}^{T} I_t \leq&   \frac{1}{T+1}\sum_{t=0}^{T}(1+\log_{1/\gamma_2}(\tfrac{l_{t}}{\gamma_2 l_{t-1}})) 
 %               \\
 %               =
 %               & 1 + \frac{1}{T+1} \log_{1/\gamma_2}(\gamma_2^{T+1} \frac{l_T}{l_{-1}}) 
 %                 \\
 %                 =
 %               & 2 + \frac{1}{T+1} \log_{1/\gamma_2}(\frac{l_T}{l_{-1}}) 
 %               \\
 %               \leq  & 3
 % \end{aligned}
 % \end{equation}
 % where the last inequality is by  
 % %\cref{thm:ls-hold-main} 
 % \tbd, i.e., $\mu\leq l_{t}\leq L/\gamma_2$  for all $t$. 
% \end{proof}
% \begin{remark}
%     Indeed, when $T\rightarrow \infty$, we can conclude $\lim_{T\rightarrow \infty}  \frac{1}{T+1}\sum_{t=0}^T I_t = 2$ following the similar analysis above.
% \end{remark}

%\section{Convergence complexity analysis of \agdap{}}
\section{Proof of \cref{thm:main}}
\label{sec:conv-complexity-wcsc-agda}
% In this section, we conduct an analysis to show that when our backtracking conditions are met, a critical inequality involving the sequence $\norm{G(\bx_t,\by_t)}, \cL(\bx_t,\by_t)$ and $\Delta_t$
% holds. The convergence results from the analysis in \cref{thm:convergence-ineq-flag-true} will be used in deriving the ultimate complexity results presented in \cref{cor:complexity-flag-true} for \cref{alg:agda}.
% \xtodo{The equations in this lemma should follow from the line search conditions.}
\na{We first start deriving some intermediate results that will be used in our complexity analysis.}
\begin{lemma}\label{lemma:complicated-imply-short-true-forproof}
   Suppose \Cref{ASPT:lipshiz gradient,aspt:bounded-Y} hold.
%and $\{\bx_{t},\by_{t},\tilde\by_{t},\xzr{L_t,l_t}\}_{t\geq0}$ are generated by \cref{alg:agda}. Then \cref{eq:alternative_ls_appendix} holds.
\na{For \agdap{}, displayed in~\cref{alg:agda}, consider an arbitrary iteration $t\in\N$. At the time \emph{\cond{t}} holds, one has
\begin{equation}
    \label{eq:alternative_ls_appendix}
    \begin{aligned}
       \MoveEqLeft[2]\Big(\tau_t - \Big(2+\frac{1}{\gamma}\Big){\tau_t^2 l_t}\Big)\norm{G_x^{\tau_t}(\bx_t,\by_t)}^2 + \sigma_t\norm{{G^{\sigma_t}_y(\bx_{t},\by_t)}}^2
                    +
            \sigma_{t}^2\frac{{\mu_{t}}}{2}\left\|{G^{\sigma_{t}}_y(\tilde\bx_{t+1},\by_{t})}\right\|^2\\
            &\leq \cL(\bx_t,\by_t) - \cL(\tilde\bx_{t+1},\tilde\by_{t+1})+4(3l_t-2{\mu_{t}})\Delta_t+Q_t,    
    \end{aligned} 
\end{equation}
where \rv{$Q_t$ is defined using $\{\hat{t}_n\}_{n\geq 0}$ and $N(t)$ as in Definitions~\ref{def:tn} and~\ref{def:N-k}, and $\Lambda_t$ as in~\cref{lemma:complicated-true-tn} for $t=\hat{t}_{N(t)}$ such that}
{\small
$$
Q_t\triangleq
\begin{cases}
            %\zeta + L_t \cD_y \sqrt{\frac{2\zeta}{\mu}}
            \Lambda_t, &\text{if $t=\hat{t}_{N(t)}$;}\\
            6l_{t-1}\Delta_t+4(3l_{t-1}-2{\mu_{t-1}})\Delta_{t-1}+2\tau_{t-1}^2l_{t-1}\norm{G_x^{\tau_{t-1}}(\bx_{t-1},\by_{t-1})}^2-\sigma_{t-1}^2{\mu_{t-1}}\left\|G^{\sigma_{t-1}}_y(\bx_{t},\by_{t-1})\right\|^2,  & \text{if $t>\hat{t}_{N(t)}$}.
            \end{cases}
$$}}%
\end{lemma}
\begin{proof}
\na{At the time \cond{t} holds, one also have \eqref{eq:ls-conv}; therefore, the desired result follows from %bounding 
setting $Q_t$ to $\Lambda_t+R_t$ term appearing in \cref{eq:ls-conv}.}
%by $Q_t$, where we use the fact that $\norm{\by_t-\tilde\by_{t+1}}\leq\cD_y$ because of \cref{aspt:bounded-Y}.
\end{proof}
%\qstodo{will this notation $R_t$ be confused by $R$?}

\begin{lemma}\label{lemma:partial-telescop-lemma-flag-true}
Suppose \Cref{ASPT:lipshiz gradient,aspt:bounded-Y} hold. \na{For \agdap{}, displayed in~\cref{alg:agda}, given an arbitrary iteration $T\in\N$, let $n=N(T)$, where $N(\cdot)$ is defined in Definition~\ref{def:N-k}. Then,}   
\begin{equation}
\label{eq5}
\begin{aligned}
\MoveEqLeft\sum_{t=\hat{t}_n}^{T}\gamma_0\tau_t\norm{G_x^{\tau_t}(\bx_t,\by_t)}^2
        + \sum_{t=\hat{t}_n}^{T}{\sigma_t}\norm{G_y^{\sigma_t}(\bx_t,\by_t)}^2 
        + \sum_{t=\hat{t}_n}^{T} \frac{{\mu_{t}}\sigma_t^2}{2}\norm{G_y^{\sigma_t}(\bx_{t+1},\by_t)}^2 
        \\
        & \leq \cL(\bx_{ \hat{t}_n},\by_{ \hat{t}_n})-\cL(\bx_{T+1},\tilde\by_{T+1}) + \Delta
        _{ \hat{t}_n} \sum_{t=\hat{t}_n}^{T} \psi_t \prod_{i=\hat{t}_n}^{t-1}B_i 
        + \na{\Lambda_{\hat t_n}},
        %\zeta + \frac{L}{\na{\gamma^r}}\cD_y \sqrt{\frac{2\zeta}{\mu}},
    \end{aligned}
\end{equation}
where
$\psi_t \triangleq \begin{cases}
   24l_t - 16{\mu_{t}} & \text{if}\; t=  \hat{t}_n, \\
    24l_t - 16{\mu_{t}} + 6l_{t-1}& \text{if}\; \hat{t}_n< \na{t<\hat{t}_{n+1}}. 
\end{cases}$ and $B_t\triangleq1-{\mu_{t}}\sigma_t/2$ for all $t\in\N$.
% \xzr{$\psi_t \triangleq \begin{cases}
%    24l_t - 16\mu & \text{if}\; t=  \hat{t}_n, \\
%     2(2l_t+l_{t-1})(3-\frac{2\mu}{l_t})+12l_{t} -8\mu& \text{if}\; \hat{t}_n< t\leq  \hat{t}_{n+1}-1\; 
% \end{cases}$}
\end{lemma}
\begin{proof}
\na{Given $T\geq 0$, since we set $n=N(T)$, we have $ \hat{t}_n\leq T<\hat{t}_{n+1}$ --note that if $n=\bar N$, then $\hat{t}_{n+1}=+\infty$. \rv{Moreover, since $T<\hat t_{n+1}$ and \cond{t} holds with $(\tilde \bx_{t+1},\tilde \by_{t+1})$ for $\hat t_{n}\leq t\leq T$, we have $(\tilde \bx_{t+1},\tilde \by_{t+1})=(\bx_{t+1},\by_{t+1})$ for $t\in\N$ such that $\hat t_{n}\leq t\leq T-1$ and $\tilde \bx_{T+1}=\bx_{T+1}$; however, it is possible that $\tilde \bx_{T+1}\neq\bx_{T+1}$ in case $T+1=\hat t_{n+1}$}. Thus, with the convention $\prod_{j=i}^{i'}B_j=1$ if $i'<i$,} it follows from \cref{lemma:complicated-imply-short-true-forproof} that
\begin{equation}
\label{eq:crude-grad-bound-1}
        \begin{aligned}
        \MoveEqLeft[2]\sum_{t=\hat{t}_n}^{T}
        \Big(\tau_{t}-{\Big(4+\frac{1}{\gamma}\Big)}{\tau_t^2 l_t}\Big)
        \norm{G_x^{\tau_t}(\bx_{t},\by_{t})}^2 
        +\sum_{t=\hat{t}_n}^{T}{\sigma_t} \norm{G_y^{\sigma_t} (\bx_{t},\by_{t})}^2
        + \sum_{t=\hat{t}_n}^{T} \frac{{\mu_{t}}\sigma_t^2}{2}\norm{G_y^{\sigma_t}(\bx_{t+1},\by_t)}^2 
        \\
        &\leq \cL(\bx_{ \hat{t}_n},\by_{ \hat{t}_n}) - \cL(\tilde\bx_{T+1},\tilde\by_{T+1}) +
       \sum_{t=\hat{t}_n}^{T}\psi_t\Delta_t+ \na{\Lambda_{\hat t_n}}  
       %\zeta + L_t\cD_y \sqrt{\frac{2\zeta}{\mu}}
       \\
      & \leq \cL(\bx_{ \hat{t}_n},\by_{ \hat{t}_n}) - \cL(\tilde\bx_{T+1},\tilde\by_{T+1})+ \Delta_{ \hat{t}_n}\sum_{t=\hat{t}_n}^{T}\psi_t\prod_{i=\hat{t}_n}^{t-1}B_i 
        \\
        & \quad +\sum_{t={\hat{t}_n+1}}^{T}\psi_t\sum_{i= \hat{t}_n+1}^t \Big(C_{i-1}\norm{G_x^{\tau_{i-1}}(\bx_{i-1},\by_{i-1})}^2 \prod_{j=i}^{t-1}B_j \Big)
        + \na{\Lambda_{\hat t_n}},
        %\zeta + \na{L_{\hat t_n}}\cD_y \sqrt{\frac{2\zeta}{\mu}},
    \end{aligned}
\end{equation}
\na{where the second inequality follows from \algline{\ref{algeq:delta-update}}, i.e., $\Delta_{t+1}=B_t\Delta_t+C_t\norm{G_x^{\tau_t}(\bx_t,\by_t)}^2$ for all $t\in\N$ such that $\hat t_n\leq t<T$ --here, $B_t=1-\rv{\mu_t}\sigma_t/2$ and $C_t=\frac{(1-\sigma_{t}{\mu_{t}})(2-\sigma_{t}{\mu_{t}} )}{\sigma_{t}{\mu_{t}}} \frac{L_{t}^2}{\mu_{t}^2} \tau_{t}^2$.}
% and we let
% \begin{equation}
%     \label{eq:psi-t-flag-true}
%     \alpha_t = 
%        \left(\frac{\sigma_{t}}{2}+\frac{\sigma^3_{t} l_{t}^2}{2}\right)\left(\frac{4}{\sigma^2_{t}}+2l^2_{t}\right)=\xzr{3l_t},
%     \quad
%      \hat\psi_t = \begin{cases}
%      \left(\sigma_t
%     + \frac{\sigma^3_{t} l_{t}^2}{2}
%    \right)\left(\frac{4}{\sigma^2_{t}}+2l^2_{t}\right)\xzr{=6l_t}
%     & \text{if}\; t=\hat{t}_n,
%     \\
%     \left(\sigma_t
%     + \frac{\sigma^3_{t} l_{t}^2}{2}
%     +\frac{\sigma_t^2}{2\sigma_{t-1}}\right)\left(\frac{4}{\sigma^2_{t}}+2l^2_{t}\right) \xzr{=6l_t + 3l_{t-1}}
%     & 
%     \text{otherwise.}
%      \end{cases}
% \end{equation}
Indeed, by rearranging terms, we have
\begin{equation*}%\label{eq18-flag-true}
    \begin{aligned}
        \sum_{{t=\hat{t}_n+1}}^{T}\psi_t \sum_{i=\hat{t}_n+1}^{t} \Big( C_{i-1}\|G_x^{\tau_{i-1}}(\bx_{i-1},\by_{i-1})\|^2 \prod_{j=i}^{t-1} B_j\Big)=\sum_{t=\hat{t}_n}^{T-1}C_{t}\|G_x^{\tau_t}(\bx_{t},\by_{t})\|^2 \sum_{i=t}^{T-1}\psi_{i+1} \prod_{j=t+1}^{i}B_{j}.
    \end{aligned}
\end{equation*}
\na{Then using above identity within \eqref{eq:crude-grad-bound-1}, we get}
\begin{equation}\label{eq:long-step-conv-1-flag-true}
    \begin{aligned}
        \MoveEqLeft\tau_{T}\Big(1-{\Big(4+\frac{1}{\gamma}\Big)}{\tau_T l_T}\Big) \norm{G_x^{\tau_T}(\bx_T, \by_T)}^2 +\sum_{t=\hat{t}_n}^{\xqs{T-1}} 
        \Big(\tau_{t}-{\Big(4+\frac{1}{\gamma}\Big)}{\tau_t^2 l_t}
       -
       C_{t} \sum_{i=t}^{T-1}\psi_{i+1} \prod_{j=t+1}^{i}B_{j}\Big)\norm{G_x^{\tau_t}(\bx_t,\by_t)}^2\\
        &\quad + \sum_{t=\hat{t}_n}^{T}{\sigma_t}\norm{G_y^{\sigma_t}(\bx_t,\by_t)}^2 
        + \sum_{t=\hat{t}_n}^{T} \frac{{\mu_{t}}\sigma_t^2}{2}\norm{G_y^{\sigma_t}(\bx_{t+1},\by_t)}^2 
        \\
        &\leq \cL(\bx_{ \hat{t}_n},\by_{ \hat{t}_n})-\cL(\tilde\bx_{T+1},\tilde \by_{T+1}) + \Delta_{ \hat{t}_n}\sum_{t=\hat{t}_n}^{T}\psi_t\prod_{i=\hat{t}_n}^{t-1}B_i 
                + \na{\Lambda_{\hat t_n}}.
                %\zeta + L_{\hat t_n}\cD_y \sqrt{\frac{2\zeta}{\mu}}.
    \end{aligned}
\end{equation}
\rv{Note that $\hat{t}_n\leq T<\hat{t}_{n+1}$; moreover, $L_t=L_{\hat{t}_n}$ and $\mu_t=\mu_{\hat t_n}$ for all $t\in\N$ such that $\hat{t}_n\leq t<\hat{t}_{n+1}$. Therefore,} we have
$1/L_{\hat{t}_n}=1/L_t\leq \sigma_{t} = 1/{l}_t$ for all $ \hat{t}_n\leq t<\hat{t}_{n+1}$, and \rv{from the definition $\psi_t$,}
% $$\psi_t \triangleq \begin{cases}
%    24l_t - 16\xzf{\mu_{t}} & \text{if}\; t=  \hat{t}_n, \\
%     24l_t - 16\xzf{\mu_{t}} + 6l_{t-1}& \text{if}\; \hat{t}_n< t<\hat{t}_{n+1}\; 
% \end{cases},$$
it follows that
\begin{equation}\label{eq:partial-bound-flag-true}
    \begin{aligned}
    & \psi_t  
    %\leq  
    % \begin{cases}
    %      6l_t + 3l_t + 3l_{t-1}\leq 12 L_{ \hat{t}_n} & \text{if} \; t=  \hat{t}_n
    % \end{cases}
     \leq {30}L_{ \hat{t}_n}-{16{\mu_{t}}}
        ,\qquad
        \sum_{i=t}^{T-1}\prod_{j=t+1}^iB_j\leq \sum_{i=t}^{T-1}\prod_{j=t+1}^i (1-\frac{{\mu_{\hat{t}_n}}}{2L_{ \hat{t}_n}})\na{\leq\sum_{k=0}^{\infty}\Big(1-\frac{{\mu_{\hat{t}_n}}}{2L_{ \hat{t}_n}}\Big)^k} = \frac{2L_{ \hat{t}_n}}{{\mu_{\hat{t}_n}}}
    \end{aligned}
\end{equation}
for all $t\in\N$ such that $\hat{t}_n\leq t<\hat{t}_{n+1}$.

Therefore, if we use the two upper bounds provided in \eqref{eq:partial-bound-flag-true} within \eqref{eq:long-step-conv-1-flag-true} and the fact that $L_t=L_{\hat t_n}$ and $\mu_t={\mu_{\hat{t}_n}}$ for all $t\in\N$ such that $\hat{t}_n\leq t<\hat{t}_{n+1}$, it follows that
\begin{equation}
\label{eq:partial-bound-flag-true-2}
    \begin{aligned}
        \MoveEqLeft\sum_{t=\hat{t}_n}^{T} \tau_t\Big(1- \Big(4+\frac{1}{\gamma}\Big)\tau_t l_t
        - {\frac{4\tau_{t}(1-\sigma_{t}{\mu_{t}})(2-\sigma_{t}{\mu_{t}}  )}{\sigma_{t} } \frac{(15L_{t}-8{\mu_{t}})L_{t}^3}{\mu_{t}^4}}\Big)\norm{G_x^{\tau_t}(\bx_t,\by_t)}^2 
        \\
        &\quad + \sum_{t=\hat{t}_n}^{T} {\sigma_t}\norm{G_y^{\sigma_t} (\bx_t,\by_t)}^2
        + \sum_{t=\hat{t}_n}^{T} \frac{{\mu_{t}}\sigma_t^2}{2}\norm{G_y^{\sigma_t}(\bx_{t+1},\by_t)}^2 
        \\
        &\leq \na{\cL(\bx_{ \hat{t}_n},\by_{ \hat{t}_n})}-\cL(\tilde\bx_{T+1},\tilde \by_{T+1}) +  \Delta
        _{ \hat{t}_n} \sum_{t=\hat{t}_n}^{T} 
        %(9{l}_{t}+3{l}_{t-1})
        \psi_t \prod_{i=\hat{t}_n}^{t-1}B_i
       + \na{\Lambda_{\hat t_n}}.    
       %\zeta + \na{L_{\hat t_n}}\cD_y \sqrt{\frac{2\zeta}{\mu}}.
    \end{aligned}
\end{equation}
% \xtodo{Have not update \cref{alg:agda}}
Moreover, using $\tau_t \triangleq \frac{(1-\gamma_0)}{l_t}\left({4+\frac{1}{\gamma}} + {\frac{4(1-\sigma_{t}{\mu_{t}})(2-\sigma_{t}{\mu_{t}}  )(15L_{t}-8{\mu_{t}})L_{t}^3}{{\mu_{t}}^4}} \right)^{-1}$
%and $L_{\hat t_n} \leq \frac{L}{\gamma_1}=\na{\frac{L}{\gamma^r}}$ 
\na{within \eqref{eq:partial-bound-flag-true-2} leads to the desired result.}
%, it follows that
% \begin{equation}
%     \begin{aligned}
%         &\sum_{t=\hat{t}_n}^{T}\gamma_0\tau_t\norm{G_x^{\tau_t}(\bx_t,\by_t)}^2
%         + \sum_{t=\hat{t}_n}^{T}\xzr{\sigma_t}\norm{G_y^{\sigma_t}(\bx_t,\by_t)}^2 
%         + 
%         \sum_{t=\hat{t}_n}^{T} \frac{\xzf{\mu_{t}}\sigma_t^2}{2}\norm{G_y^{\sigma_t}(\bx_{t+1},\by_t)}^2 
%         \\
%         & \leq \cL(\bx_{ \hat{t}_n},\by_{ \hat{t}_n})-\cL(\bx_{T+1},\tilde\by_{T+1}) + \Delta
%         _{ \hat{t}_n} \sum_{t=\hat{t}_n}^{T}   \psi_t\prod_{i=\hat{t}_n}^{t-1}B_i 
%         +     \zeta + \frac{L}{\na{\gamma^r}}\cD_y \sqrt{\frac{2\zeta}{\xzf{\mu_{t}}}}
%     \end{aligned}
% \end{equation}
%which completes the proof.
\end{proof}

\begin{lemma}\label{thm:convergence-ineq-flag-true}
Suppose \Cref{ASPT:lipshiz gradient,aspt:bounded-Y,aspt:primal_lb} hold. \na{For \agdap{}, displayed in~\cref{alg:agda},} 
\begin{equation}
\label{eq:complexity-bound}
\begin{aligned}
        &\sum_{t=0}^{T}\gamma_0\tau_t\norm{G_x^{\tau_t}(\bx_t,\by_t)}^2
        + \sum_{t=0}^{T}{\sigma_t}\norm{G_y^{\sigma_t}(\bx_t,\by_t)}^2 
        + 
        \sum_{t=0}^{T} \frac{{\mu_{t}}\sigma_t^2}{2}\norm{G_y^{\sigma_t}(\bx_{t+1},\by_t)}^2 
        \leq \cL(\bx_0,\by_0)\xqs{-F^*} 
        %+ \xqs{(\frac{1}{\gamma_1} + \frac{3}{2})L\cD_y^2}
        + \Gamma_T+ \Xi_T
    \end{aligned}
\end{equation}
holds for all $T\in\N$, where using $\psi_t$, $B_t$ and $N(T)$ given in \cref{lemma:partial-telescop-lemma-flag-true} and Definition~\ref{def:N-k}, we define $\Gamma_T$ and $\Xi_T$ as
\begin{equation}
\label{eq:Sigma_Gamma}
    \begin{aligned}
        &\Xi_T\triangleq\sum_{n=0}^{N(T)-1}\left(\sum_{t={ \hat{t}_n}}^{ \hat{t}_{n+1}-1}\psi_t \prod_{i=\hat{t}_n}^{t-1}B_i\right) \Delta
        _{ \hat{t}_n} +\left(\sum_{t={\hat t_{N(T)}}}^{T} \psi_t\prod_{i=\hat t_{N(T)}}^{t-1}B_i\right) \Delta
        _{\hat t_{N(T)}},\\
        &\Gamma_T \triangleq \Big(1+N(T)\Big)\cdot{\Big(\bar{L}/\gamma^r+\rv{L}\Big)\cD_y^2} + \na{\sum_{n=0}^{N(T)}\Lambda_{\hat t_n}.}
    \end{aligned}
\end{equation}
%$\psi_t$ are defined in \cref{lemma:partial-telescop-lemma-flag-true}, and $N(T)$ is defined in Definition~\ref{def:N-k}.
\end{lemma}
\begin{proof}
\na{Given $T\geq 0$, it follows from \cref{lemma:partial-telescop-lemma-flag-true} that}
\begin{equation}\label{eq:partial-sum-tail-flag-true}
\begin{aligned}
    \MoveEqLeft\sum_{t=\hat t_{N(T)}}^{T}\gamma_0\tau_t\norm{G_x^{\tau_t}(\bx_t,\by_t)}^2
        + \sum_{t=\hat t_{N(T)}}^{T}{\sigma_t}\norm{G_y^{\sigma_t}(\bx_t,\by_t)}^2 
        +
       \sum_{t=\hat t_{N(T)}}^{T}\frac{{\mu_{t}}\sigma_t^2}{2}\norm{G_y^{\sigma_t}(\bx_{t+1},\by_t)}^2 
        \\
        & \leq \cL(\bx_{t_{N(T)}},\by_{t_{N(T)}})-\cL(\bx_{T+1},\tilde{\by}_{T+1}) + \na{\Delta_{\hat{t}_{T(N)}}} \sum_{t=\hat t_{N(T)}}^{T}  \psi_t\prod_{i=\hat t_{N(T)}}^{t-1}B_i  
       + \na{\Lambda_{\hat t_{N(T)}}},
       %\zeta + \frac{L}{\gamma^r}\cD_y\sqrt{\frac{2\zeta}{\xzf{\mu_{t}}}}
    \end{aligned}
\end{equation}
where $N(T)$ is given in Definition~\ref{def:N-k}. \na{If $N(T)\geq 1$, then for all $n\in\N$ such that $n\leq N(T)-1$,} \cref{lemma:partial-telescop-lemma-flag-true} also implies that
\begin{equation}\label{eq:partial-sum-1-flag-true}
\begin{aligned}
    \MoveEqLeft\sum_{t=\hat{t}_n}^{ \hat{t}_{n+1}-1}\gamma_0\tau_t\norm{G_x^{\tau_t}(\bx_t,\by_t)}^2
        + \sum_{t=\hat{t}_n}^{ \hat{t}_{n+1}-1}{\sigma_t}\norm{G_y^{\sigma_t}(\bx_t,\by_t)}^2 
        + 
        \sum_{t=\hat{t}_n}^{\hat{t}_{n+1}-1} \frac{{\mu_{t}}\sigma_t^2}{2}\norm{G_y^{\sigma_t}(\bx_{t+1},\by_t)}^2 
        \\
        & \leq \cL(\bx_{ \hat{t}_n},\by_{ \hat{t}_n})-\cL(\bx_{ \hat{t}_{n+1}},\tilde{\by}_{ \hat{t}_{n+1}}) +\Delta_{ \hat{t}_n} \sum_{t=\hat{t}_n}^{ \hat{t}_{n+1}-1} \psi_t \prod_{i=\hat{t}_n}^{t-1}B_i  + \na{\Lambda_{\hat t_n}}.
        %\zeta + \frac{L}{\gamma^r}\cD_y\sqrt{\frac{2\zeta}{\xzf{\mu_{t}}}}.
    \end{aligned}
\end{equation}
Moreover, \na{if $N(T)\geq 1$, then for all $n\in\N$ such that $n\leq N(T)-1$, it is crucial to note that $\tilde{\by}_{ \hat{t}_{n+1}}\neq\by_{\hat{t}_{n+1}}$; indeed,} $$\tilde{\by}_{ \hat{t}_{n+1}} = \prox{\sigma_{ \hat{t}_{n+1}-1}h}(\by_{ \hat{t}_{n+1}-1} + \sigma_{\hat{t}_{n+1}-1}\grad_y f(\bx_{ \hat{t}_{n+1}}, \by_{ \hat{t}_{n+1}-1}))$$
\na{while $\by_{\hat{t}_{n+1}}$ is computed 
%such that $\cL(\bx_{\hat{t}_{n+1}},\by_{\hat{t}_{n+1}})+\zeta\geq \max_{y\in\cY}\cL(\bx_{\hat{t}_{n+1}},\by)$.
using \algline{\ref{algeq:haty-max}} for the case \msf{True} and using \algline{\ref{algeq:haty-y}} for the case \msf{False}.
Therefore, setting \rv{$\bar\bx=\bx_{\hat{t}_{n+1}}$, $\bar\by=\by_{ \hat{t}_{n+1}-1}$,} $\by={\by}_{ \hat{t}_{n+1}}$ and $\hat\by=\tilde{\by}_{ \hat{t}_{n+1}}$ within \cref{lemma:F-common-gap}, it implies that}
%\xtodo{Maybe we can remind \cref{cor:bound-tn-kn} for $L\leq L/\gamma^r$. I thought for a while on this trivial condition.}
\begin{equation}
\label{eq:tail_bound_correction-tn}
    \na{\cL(\bx_{\hat{t}_{n+1}}, {\by}_{ \hat{t}_{n+1}}) - \cL(\bx_{ \hat{t}_{n+1}}, \tilde\by_{ \hat{t}_{n+1}})} \leq \Big(\frac{1}{\sigma_{ \hat{t}_{n+1}-1}} +\rv{L}\Big)\cD_y^2 \leq \Big(L_{ \hat{t}_{n}}+\rv{L}\Big)\cD_y^2\leq  {\Big(\bar{L}/\gamma^r+\rv{L}\Big)\cD_y^2},
\end{equation}
\na{where we used $1/\sigma_{\hat{t}_{n+1}-1}={l_{\hat t_n}}\leq L_{\hat t_n}\leq {\bar L/\gamma_1=\bar L/\gamma^r}$ from \cref{cor:bound-tn-kn}. Similarly, using the same argument, i.e., invoking \cref{lemma:F-common-gap} this time by setting $\by=\by^*(\bx_{T+1})$ and $\hat\by=\tilde{\by}_{T+1}$ implies that}
\begin{equation}
\label{eq:tail_bound_correction}
   F(\bx_{T+1}) - \cL(\bx_{T+1}, \tilde{\by}_{T+1}) \leq {\Big(\bar{L}/\gamma^r+\rv{L}\Big)\cD_y^2}.
\end{equation}
\na{Next, after adding \eqref{eq:partial-sum-1-flag-true} and \eqref{eq:tail_bound_correction-tn}, we sum the resulting inequality over $n\in\{0,1,2,...,N(T)-1\}$, and finally we add \eqref{eq:partial-sum-tail-flag-true} and \eqref{eq:tail_bound_correction} to the sum, which leads to the desired inequality in \eqref{eq:complexity-bound}, where we use the fact $F(\bx_{T+1})\geq F^*$.} 
\end{proof}
\na{We are now ready to provide the proof of the main result of this paper.}
\paragraph{Proof of \cref{thm:main}}
For all $t\in\N$, $B_t = 1-\frac{\sigma_t{\mu_{t}}}{2} \leq 1- \frac{{\mu_{t}}}{2L_t}$ since $l_t\leq L_t$. Moreover, for all $n\in\N$ such that $n\leq \bar N$, we also have $L_{\hat{t}_n}= {\frac{\tilde l}{\gamma_1^{k_n}}}$ and ${\mu_{\hat{t}_n}=\max\{\underline{\mu},\tilde{\mu}\gamma_1^{k_n}}\}$ \rv{for some initial estimates $\tilde l>\tilde \mu>0$} %\xzf{$\tilde{l}=\tilde{\mu}/\gamma_2$}
(see \cref{rem:L_tn}),
  %\xtodo{This is still correct, right?$L_{ \hat{t}_n}= \xzr{\frac{\tilde l}{\gamma_1^{k_n}}} =\frac{\mu}{\gamma_2\gamma_1^{k_n}}$.} 
  and $L_t = L_{\hat{t}_n}$ and $\mu_t={\mu_{\hat{t}_n}}$ for all $t\in\N$ such that $\hat{t}_n\leq t \leq  \hat{t}_{n+1}-1$, where $\bar N$ is \na{given} in Definition~\ref{def:N-k}. Therefore, for any given $T\in\N$, 
  %using the definitions of $\{\psi_t\}_{t=0}^T$ and $\Xi_T$, 
  we get $\psi_t\leq 30L_t$ for all $t\in\N$ such that $0\leq t\leq T$. %which helps us 
  Thus, we can bound $\Xi_T$ as follows: %$\varsigma$
        \begin{equation}\label{eq:first-bound-Xi-K}
        \begin{aligned}
               \Xi_T & \leq   \sum_{n=0}^{N(T)-1}\left(\sum_{t={ \hat{t}_n}}^{\hat{t}_{n+1}-1} 
               30L_t\prod_{i=\hat{t}_n}^{\na{t-1}}B_i\right) \Delta
        _{ \hat{t}_n} +\left(\sum_{t=t_{N(T)}}^{T} 
        %(12{l}_t + 18{l}_{t-1})
        30L_t
        \prod_{i=t_{N(T)}}^{t-1}B_i\right) \Delta
        _{t_{N(T)}}  \\
         &\leq
        %\frac{\xzf{\tilde{\mu}}}{\gamma_2}
        \sum_{n=0}^{N(T)-1}\frac{\rv{30\tilde l}}{\gamma_1^{k_n}}\sum_{t={\hat{t}_n}}^{\hat{t}_{n+1}-1} \Big(1-
        %{\gamma_2}
        \rv{\frac{\tilde\mu}{2\tilde l}}\gamma_1^{{(1+\varsigma)}k_n}\Big)^{t- \hat{t}_n} \Delta
        _{ \hat{t}_n} + 
        %{\frac{30\xzf{\tilde{\mu}}}{\gamma_2}}\cdot
        \frac{\rv{30\tilde l}}{\gamma_1^{k_{N(T)}}} \sum_{t=t_{N(T)}}^{T} \Big(1-
        %{\gamma_2}
        \rv{\frac{\tilde\mu}{2\tilde l}}\gamma_1^{{(1+\varsigma)}k_{N(T)}}\Big)^{t-t_{N(T)}}  \Delta
        _{t_{N(T)}}
        \\
        & \leq \rv{60\tilde l\cdot\frac{{\tilde l}}{\tilde\mu}}
 %\frac{\xzf{\tilde{\mu}}{\gamma_2^2}
 \left(\sum_{n=0}^{N(T)-1}\frac{1}{\gamma_1^{{(2+\varsigma)}k_n}} \Delta_{\hat{t}_n} +\frac{1}{\gamma_1^{{(2+\varsigma)}k_{N(T)}}}\Delta_{t_{N(T)}}\right)
        = \rv{60\tilde l\cdot\frac{{\tilde l}}{\tilde\mu}}
        %{\frac{60\xzf{\tilde{\mu}}}{\gamma_2^2}}
        \sum_{n=0}^{N(T)}\na{\frac{1}{\gamma_1^{{(2+\varsigma)}k_n}}\Delta_{\hat{t}_n}}.
        %=\frac{60\xzf{\gamma_2}}{\xzf{\tilde{\mu}^2}}\sum_{n=0}^{N(T)}\xzf{L_{\hat t_n}^3\Delta_{\hat{t}_n}}.
        %\leq \na{{\frac{60\xzf{\tilde{\mu}}}{\gamma_2^2}}\sum_{n=0}^{\lfloor \frac{\log_{1/\gamma}(\kappa)+1}{r}\rfloor-1}\frac{1}{\gamma_1^{2k_n}}\Delta_{\hat{t}_n}},
        \end{aligned}
    \end{equation} 
    \rv{Whenever $\mu>0$ is known, we set $\underline{\mu}=\tilde\mu=\mu$; hence, we get $\varsigma=0$ if $\mu$ is known; otherwise, $\varsigma=1$.}
\na{Next, we show that $\bar\Delta$ satisfies $\max\{\Delta_{\hat t_n}: n=0,\ldots,N(T)\}\leq \bar\Delta$. Indeed, for the case \msf{True}, \rv{according to \algline{\ref{algeq:delta_def1}}, 
% since
% \begin{equation*}
%     h(\by_{ \hat{t}_n}) - f(\bx_{ \hat{t}_n},\by_{ \hat{t}_n}) \leq h(\by^*(\bx_{ \hat{t}_n})) - f(\bx_{ \hat{t}_n},\by^*(\bx_{ \hat{t}_n})) + \zeta,\quad \na{\forall~n\in[\bar N],}
% \end{equation*}
% %holds for all $n\geq 0$,
% %it follows from 
% \cref{lemma:y0-bound} implies that $\Delta_{ \hat{t}_n}\leq \frac{2\zeta}{\xzf{\bar{\mu}}}$ 
we have $\Delta_{ \hat{t}_n}=\frac{2\zeta}{\mu_t}\leq\frac{2\zeta}{\gamma_1\bar \mu}$ 
for all \na{$n\in [\bar N]$} since $\mu_t\geq \gamma_1\bar\mu$.} On the other hand, 
for the case \msf{False}, we have $\Delta_{\hat t_n}= \min\{(1+2\frac{L_{\hat t_n}}{{\mu_{\hat{t}_n}}})d_{\hat t_n},~\overline{\cD}_y\}^2\leq \bar\Delta$, where we use \rv{$L_{\hat t_n}\leq \bar L/\gamma_1$, $\mu_{\hat t_n}\geq \gamma_1\bar\mu$, and} the fact that $d_{\hat t_n}\leq\cD_y$ because of \cref{aspt:bounded-Y}
%, and we also used $L_{\hat t_n} \leq \frac{L}{\gamma_1}=\na{\frac{L}{\gamma^r}}$ for all $t\in\N$
. Therefore, using $N(T)\leq \rv{\lceil \frac{1}{r} \log_{1/\gamma}(\cR)\rceil}$ and $\rv{\cR=\max\{{L}/\tilde{l},~\tilde{\mu}/\mu,~1\}}$, which follows from {\cref{cor:bound-tn-kn}}, we get that
%we get
%\xtodo{This bound may be able to optimized.}
\rv{
\begin{equation*}
    \begin{aligned}
    \Xi_T 
    %\leq\frac{60\xzf{\gamma_2}}{\xzf{\tilde{\mu}^2}}\sum_{n=0}^{N(T)}\xzf{L_{\hat t_n}^3\Delta_{\hat{t}_n}}
    \leq  60\tilde l\cdot\frac{\tilde l}{\tilde \mu} \bar\Delta\sum_{n=0}^{\lceil \frac{1}{r}\log_{1/\gamma}(\cR)\rceil}\frac{1}{\gamma_1^{{(2+\varsigma)}k_n}}
    \leq 60\tilde l\cdot\frac{\tilde l}{\tilde \mu} \bar\Delta\cdot\frac{\left(1/{\gamma_1^{{(2+\varsigma)}}}\right)^{\lceil \log_{1/\gamma_1}(\cR)\rceil}-1}{1/{\gamma_1^{(2+\varsigma)}}-1} \leq 
    60\tilde l\cdot\frac{\tilde l}{\tilde \mu} \frac{\cR^{(2+\varsigma)}}{1-\gamma^{{(2+\varsigma)}r}} \bar\Delta,
    %60\frac{\gamma^{\xzf{(2+\varsigma)r-(4+\varsigma)}}}{1-\gamma^{\xzf{(2+\varsigma)}r}} \xzf{\tilde{\mu}\xzrr{\cR}^{\xzf{(2+\varsigma)}}}\bar\Delta,
    \end{aligned}
\end{equation*}}%
where in the second inequality we used the fact that $\{k_n\}_n\in\integers_+$ is an increasing sequence.} 

\na{Next, we show that $\bar\Lambda$ satisfies $\max\{\Lambda_{\hat t_n}: n=0,\ldots,N(T)\}\leq \bar\Lambda$. Indeed, for the case \msf{True}, according to \agdap{} we have $\Lambda_{\hat t_n}=\zeta + L_{\hat t_n}\norm{\by_{\hat t_n}-\tilde\by_{\hat t_n+1}} \sqrt{\frac{2\zeta}{{\mu_{\hat{t}_n}}}}$; hence $\Lambda_{\hat t_n}\leq \zeta + L_{\hat t_n}\cD_y\sqrt{\frac{2\zeta}{{\mu_{\hat{t}_n}}}}\leq \zeta + \rv{
\frac{\bar{L}}{\sqrt{\bar{\mu}}}
\cD_y\sqrt{\frac{2\zeta}{\gamma^{3r}}}=\bar\Lambda}$, 
where we use the fact that $\norm{\by_t-\tilde\by_{t+1}}\leq\cD_y$ (from \cref{aspt:bounded-Y}), and we also used $L_{\hat t_n} \leq {\bar{L}}/\gamma_1$, {$\mu\geq\gamma_1 {\mu}$} (from \cref{cor:bound-tn-kn}) and $\gamma_1=\gamma^r$. %On the other hand, 
For the case \msf{False}, we have $\Lambda_{\hat t_n}= 2 L_{\hat t_n} d_{\hat t_n}\norm{\by_{\hat t_n}-\tilde\by_{\hat t_n+1}}\leq 2{\bar{L}}/\gamma_1\cD_y^2=\bar\Lambda$, which follows from the same bounds used in the other case in addition to $\norm{d_{\hat t_n}}\leq \cD_y$ due to \cref{aspt:bounded-Y}. Therefore, we can bound $\Gamma_T$, defined in \eqref{eq:Sigma_Gamma}, as $\Gamma_T\leq (1+N(T))\Big({\Big(\bar{L}/\gamma^r+\rv{L}\Big)\cD_y^2}+\bar\Lambda\Big)$. Finally, using the $N(T)$ bound from \cref{cor:bound-NK} completes the proof.}

\section{Additional technical \na{results}}\label{app:technical}
\begin{lemma}
Suppose Assumption \ref{ASPT:lipshiz gradient} holds. For any given $\na{\bar\bx^+}\in\dom g$, let $H:\cY\to\reals\cup\{+\infty\}$ such that $H(\cdot)\triangleq h(\cdot)-f(\bar\bx^+,\cdot)$ and suppose that
   \begin{equation}\label{eq:general-ls1}
   -f(\bar\bx^+,\bar\by)-\fprod{\grad_y f(\bar\bx^+,\bar\by),\rv{\bar\by^+}-\bar\by}+\frac{L'}{2}\norm{\bar\by^+-\bar\by}^2 \geq 
       - f(\bar\bx^+,\bar\by^+)
   \end{equation}
   holds for some $L'>0$ and $\bar\by\in\dom h$, then
\begin{equation}
    \label{eq:str-cvx-H}
    H(\by)\geq H(\bar\by^+)-\fprod{G_y^\sigma(\bar\bx^+,\bar\by),\by-\bar\by}+\frac{\sigma}{2}(2-L'\sigma)\norm{G_y^\sigma(\bar\bx^+,\bar\by)}^2+\frac{\mu}{2}\norm{\by-\bar\by}^2
\end{equation}
holds for all $\by\in\cY$, where $\bar\by^+=\prox{\sigma h}\Big(\bar{\by}+\sigma \grad_y f(\bar\bx^+,\bar \by)\Big)$ and \na{$G_y^\sigma(\bar\bx^+,\bar\by)=(\bar\by^+-\bar\by)/\sigma$.}
\end{lemma}
\begin{proof}
     Define $\tilde H:\cY\to\reals\cup\{+\infty\}$ such that $\widetilde H(\by)\triangleq h(\by)-f(\bar\bx^+,\bar\by)-\fprod{\grad_y f(\bar\bx^+,\bar\by),\by-\bar\by}+\frac{1}{2\sigma}\norm{\by-\bar\by}^2$ for all $\by\in\cY$. Clearly, $\bar\by^+=\argmin_{\by\in\cY}\widetilde H(\by)$. The first-order optimality condition implies that
    \begin{align*}
        \grad_y f(\bar\bx^+,\bar\by)-G_y^\sigma(\bar\bx^+,\bar\by)=\grad_y f(\bar\bx^+,\bar\by) - \frac{1}{\sigma}(\bar\by^+ -\bby)\in\partial h(\bar\by^+);
    \end{align*}
    therefore, using the convexity of $h$, we get
    \begin{align}
    \label{eq:subgradient_h}
        h(\by)\geq h(\bar\by^+)+\fprod{\grad_y f(\bar\bx^+,\bar\by)-G_y^\sigma(\bar\bx^+,\bar\by),\by-\bar\by^+},\quad\forall~\by\in\cY.
    \end{align}
    Moreover, since $f(\bar\bx^+,\cdot)$ is $\mu$-strongly concave, for all $\by\in\dom h$, 
    \begin{equation}
    \label{eq:str-concavity}
        \begin{split}
            -f(\bar\bx^+,\by)-\frac{\mu}{2}\norm{\by-\bar\by}^2 
            &\geq -f(\bar\bx^+,\bar\by)-\fprod{\grad_yf(\bar\bx^+,\bar\by),\by-\bar\by}\\
            & = -f(\bar\bx^+,\bar\by)-\fprod{\grad_yf(\bar\bx^+,\bar\by),\bar\by^+-\bar\by}-\fprod{\grad_y f(\bar\bx^+,\bar\by),~\by-\bar\by^+}.\\
        \end{split}
    \end{equation}
   Next, summing \eqref{eq:subgradient_h} and \eqref{eq:str-concavity} gives
   \begin{align}
   \label{eq:H-bound-1}
    \begin{split}
       H(\by)-\frac{\mu}{2}\norm{\by-\bar\by}^2 
        &\geq \widetilde H(\bar\by^+)-\frac{1}{2\sigma}\norm{\bar\by^+-\bar\by}^2-\fprod{G_y^\sigma(\bar\bx^+,\bar\by),\by-\bar\by^+}\\
        & = \widetilde H(\bar\by^+)-\frac{1}{2\sigma}\norm{\bar\by^+-\bar\by}^2-\fprod{G_y^\sigma(\bar\bx^+,\bar\by),\by-\bar\by}+\sigma\fprod{G_y^\sigma(\bar\bx^+,\bar\by),(\bar\by^+-\bar\by)/\sigma}\\
        & = \widetilde H(\bar\by^+)+\frac{\sigma}{2}\norm{G_y^\sigma(\bar\bx^+,\bar\by)}^2-\fprod{G_y^\sigma(\bar\bx^+,\bar\by),\by-\bar\by}.
    \end{split}
   \end{align}
    Then it follows from \cref{eq:general-ls1} that
    \begin{equation}
        \label{eq:H-bound-2}
        \begin{split}
        \widetilde H(\bar\by^+)
        &=h(\bar\by^+)-f(\bar\bx^+,\bar\by)-\fprod{\grad_y f(\bar\bx^+,\bar\by),\bar\by^+-\bar\by}+\frac{L'}{2}\norm{\bar\by^+-\bar\by}^2+\frac{1}{2}\Big(\frac{1}{\sigma}-L'\Big)\norm{\bar\by^+-\bar\by}^2\\
        &\geq H(\bar\by^+)+\frac{1}{2}\Big(\frac{1}{\sigma}-L'\Big)\norm{\bar\by^+-\bar\by}^2
        = H(\bar\by^+)+\frac{\sigma}{2}(1-L'\sigma)\norm{G_y^\sigma(\bar\bx^+,\bar\by)}^2.\\
        \end{split}
    \end{equation}
    Thus, \eqref{eq:H-bound-1} and \eqref{eq:H-bound-2} together imply the desired inequality in~\eqref{eq:str-cvx-H}.
\end{proof}
\begin{lemma}
\label{cor:str-concavity}
Suppose Assumption \ref{ASPT:lipshiz gradient} holds. For any given $(\bar\bx^+,\bar\by)\in\dom g\times\dom h$ and $\sigma>0$, if \cref{eq:general-ls1} holds for some $L'\geq 0$, then
\begin{align}
\begin{split}
    &2\sigma\fprod{ G_y^\sigma(\bar\bx^+,\na{\bar\by}),\na{\bar\by}-{\by}^{*}(\bar\bx^+)}+\sigma^2\norm{G_y^\sigma(\bar\bx^+,\na{\bar\by})}\leq-\sigma\mu\norm{\na{\bar\by}-{\by}^{*}(\bar\bx^+)}^2-\sigma^2(1-L'\sigma)\norm{G_y^\sigma(\bar\bx^+,\na{\bar\by})}^2.
\end{split}
\end{align}
\end{lemma}
\begin{proof}
From the definition of $\by^*(\cdot)$, we have $\by^*(\cdot)=\argmin_{\by\in\cY} H(\by)$; therefore, substituting $\by=\by^*(\bar\bx^+)$ in \eqref{eq:str-cvx-H}, rearranging the terms and multiplying both sides by $2\sigma$ immediately implies the desired inequality.
\end{proof}
\begin{lemma}\label{lemma:y0-bound}
    Suppose \cref{ASPT:lipshiz gradient} holds. For any given $\bx\in\cX$ and $\zeta\geq 0$, consider \na{$\tilde\by\in\cY$} such that
    \begin{equation}\label{eq:boundary-init}
      h(\tilde\by)-f(\bx,\tilde\by)\leq h(\by^*(\bx)) - f(\bx,\by^*(\bx)) + \zeta.
    \end{equation}
    Then, $\|\tilde\by-\by^*(\bx)\|^2 \leq \frac{2\zeta}{\mu}$, and for all $y\in\cY$ it holds that
        \begin{equation}
        h(\tilde \by)-h(\by) \leq \langle - \grad_y f(\bx,\tilde\by) , \by - \tilde\by \rangle + \zeta + L\|\by -\tilde\by\|\sqrt{\frac{2\zeta}{\mu}},
    \end{equation}
    which trivially implies $h(\tilde \by) - h(\by)\leq \langle - \grad_y f(\bx,\tilde\by) , \by - \tilde\by \rangle + \zeta + L\cD_y\sqrt{\frac{2\zeta}{\mu}}$ if \cref{aspt:bounded-Y} holds as well.
\end{lemma}
\begin{proof}
\na{
    Since $\by^*(\bx) = \argmin_{\by\in\cY} h(\by) - f(\bx,\by)$, it implies that 
    \begin{equation}
         \grad_y f(\bx,\by^*(\bx))\in \partial h(\by^*(\bx));
    \end{equation}
   hence, for any $\by \in \cY$, one gets
    \begin{equation}
        h(\by)\geq h(\by^*(\bx)) + \langle \grad_y f(\bx,\by^*(\bx)), \by -\by^*(\bx) \rangle.
    \end{equation}
    Therefore, \cref{eq:boundary-init} implies that
    \begin{equation*}
     h(\tilde\by) + \langle \grad_y f(\bx,\by^*(\bx)), \by -\by^*(\bx) \rangle \leq h(\by)  + \zeta +   f(\bx,\tilde \by) - f(\bx,\by^*(\bx)).
    \end{equation*}
    By adding $\langle \grad_y f(\bx,\by^*(\bx)),   \by^*(\bx)-\tilde\by   \rangle$ to both sides, one gets
    \begin{equation*}
    \begin{aligned}
             h(\tilde\by) + \langle \grad_y f(\bx,\by^*(\bx)), \by -\tilde\by \rangle & \leq h(\by)  + \zeta +   f(\bx,\tilde\by) - f(\bx,\by^*(\bx)) + \langle \grad_y f(\bx,\by^*(\bx)),   \by^*(\bx)-\tilde\by   \rangle \\
             & \leq h(\by)  + \zeta - \frac{\mu}{2}\|\tilde \by - \by^*(\bx)\|^2,
    \end{aligned}
    \end{equation*}
    where the second inequality %we used 
    follows from the strong concavity of $f(\bx,\cdot)$. By adding $\langle \grad_y f(\bx,\tilde\by), \by -\tilde\by \rangle$ to both sides, %we get
    \begin{equation*}
    \begin{aligned}
             h(\tilde\by) + \langle \grad_y f(\bx,\tilde\by), \by -\tilde\by \rangle & \leq h(\by)  + \zeta - \frac{\mu}{2}\|\tilde \by - \by^*(\bx)\|^2 - \langle \grad_y f(\bx,\by^*(\bx)) -  \grad_y f(\bx,\tilde\by), \by - \tilde\by \rangle \\
             & \leq h(\by)  + \zeta + L\|\by^*(\bx)-\tilde \by\|\|\by -\tilde \by\|.
    \end{aligned}
    \end{equation*}
    Furthermore, using the optimality of $\by^*(\bx)$ and the strong convexity of $h(\cdot)-f(\bx,\cdot)$, \cref{eq:boundary-init} implies that
    $
        \frac{\mu}{2}\|\tilde\by-\by^*(\bx)\|^2 \leq \zeta.
    $
    Therefore, we obtain 
    \begin{equation*}
    \begin{aligned}
             h(\tilde\by) + \langle \grad_y f(\bx,\tilde\by), \by -\tilde\by \rangle \leq h(\by)  + \zeta + L\|\by -\tilde \by\|\sqrt{\frac{2\zeta}{\mu}}\leq h(\by)  + \zeta + L\cD_y\sqrt{\frac{2\zeta}{\mu}},
    \end{aligned}
    \end{equation*}
    which completes the proof.}
\end{proof}
\begin{lemma}\label{lemma:deltat}
 %Suppose Assumption \ref{ASPT:lipshiz gradient} holds  and either \cref{alg:agda} is checking at \shortalgline{\ref{algline:check-agda+}} or \cref{alg:sagda+} is checking at  \shortalgline{18} at time step
 Suppose Assumption \ref{ASPT:lipshiz gradient} holds, and \na{let $t\in\mathbb{N}$ such that $t\geq \hat t_{N(t)}+1$. Assuming %\eqref{eq:backtrack} held for iteration $t-1$
 \emph{\cond{t-1}} holds, if $L_t\geq L$ {and $\mu_t\in(0,\mu]$} at the time when 
 %the backtracking condition in \eqref{eq:backtrack} 
 \emph{\cond{t}} is checked during the execution of \emph{\shortalgline{\ref{algline:check-agda+}}} in \cref{alg:agda} 
 %or \emph{\shortalgline{\ref{algline:check-sagda+}}} of \cref{alg:sagda+} 
 for iteration $t$, then}
    \begin{equation}
    \begin{aligned}
        \delta_{t}\leq 
        B_{t-1}\delta_{t-1} + C_{t-1}\|G_x^{\tau_{t-1}}(\bx_{t-1},\by_{t-1})\|^2,
    \end{aligned}
\end{equation}
where $B_{t}=\left(1-{\mu_t} \sigma_{t}/2\right)$ , $C_{t}= \frac{(1-\sigma_{t} \rv{\mu_t})(2-\sigma_{t} {\mu_t} )}{\sigma_{t} {\mu_t}} \frac{L_t^2}{\mu_t^2} \tau_{t}^2$, $\delta_t = \|\by_t-\by^*(\bx_t)\|^2$ and $N(t)$ is defined in Definition~\ref{def:N-k}.
%and $\tilde \delta_t \triangleq \|\tilde\by_t-\by^*(\bx_t)\|^2$.
\end{lemma}
\begin{proof}
Fix $t\geq \hat{t}_{N(t)}+1$. For notation simplicity, let $\by^*_{t}\triangleq\by^*(\bx_{t})$.
%\xzr{By Definition \ref{def:N-k}, we have $t\geq  \hat{t}_n$, and $L_{t}=L_{ \hat{t}_n}$.}
\na{For iteration $t$, at the time \cref{alg:agda} is checking \eqref{eq:backtrack} in \shortalgline{\ref{algline:check-agda+}},} %or \cref{alg:sagda+} is checking \eqref{eq:backtrack} in \shortalgline{\ref{algline:check-sagda+}}, 
we have $\by_{t}=\prox{\sigma_{t-1} h}\Big(\by_{t-1}+\sigma_{t-1} \grad_y f(\bx_{t},\by_{t-1})\Big)$
and $G_y^{\sigma_{t-1}} (\bx_{t},\by_{t-1})=(\by_{t}-\by_{t-1})/\sigma_{t-1}$. Therefore, %we can obtain 
\begin{equation}\label{delta_primal_AGDA}
    \begin{aligned}
        \delta_{t}&=\|\by_{t}^*-\by_{t}\|^2\\
        &=\|\by_{t}^*-\by_{t-1}-\sigma_{t-1} G_y^{\sigma_{t-1}} (\bx_{t},\by_{t-1})\|^2\\
        &=\|\by_{t}^*-\by_{t-1}\|^2 + 2\sigma_{t-1} \langle  G_y^{\sigma_{t-1}} (\bx_{t},\by_{t-1}),\by_{t-1} - \by_{t}^*\rangle +\sigma_{t-1}^2 \| G_y^{\sigma_{t-1}} (\bx_{t},\by_{t-1})\|^2.
    \end{aligned}
\end{equation}
Then, \na{since \cref{eq:Lyy-1} holds for the iteration $t-1$, it follows from \rv{$\mu_{t-1}=\mu_t\in(0,\mu]$} and \cref{cor:str-concavity} \rv{(with $\bar\by=\by_{t-1}$ and $\bar\bx^+=\bx_t$)} that}
\begin{equation}\label{lemma_use_AGDA}
    \begin{aligned}
        \MoveEqLeft 2\sigma_{t-1} \langle  G_y^{\sigma_{t-1}} (\bx_{t},\by_{t-1}), \by_{t-1}-\by_{t}^*\rangle + \sigma_{t-1}^2\| G_y^{\sigma_{t-1}} (\bx_{t},\by_{t-1})\|^2\\
        &\leq -\sigma_{t-1} {\mu_{t-1}} \|\by_{t-1}-\by_{t}^*\|^2 - \sigma_{t-1}^2(1-\sigma_{t-1} l_{t-1} ) \| G_y^{\sigma_{t-1}} (\bx_{t},\by_{t-1})\|^2.
    \end{aligned}
\end{equation}
By plugging \cref{lemma_use_AGDA} into \cref{delta_primal_AGDA}, we can obtain that 
\begin{equation}\label{delta_ineq1_AGDA}
    \begin{aligned}
        \delta_{t}&\leq (1-\sigma_{t-1} {\mu_{t-1}} )\|\by_{t}^*-\by_{t-1}\|^2 - \sigma_{t-1}^2(1-\sigma_{t-1}l_{t-1} )\| G_y^{\sigma_{t-1}} (\bx_{t},\by_{t-1})\|^2.
    \end{aligned}
\end{equation}
Next, by Young's inequality, for any $\alpha>0$, we can bound $\|\by_{t}^*-\by_{t-1}\|^2$ as follows
\begin{equation*}
    \begin{aligned}
        &\|\by_{t}^*-\by_{t-1}\|^2
        =\|\by_{t}^*-\by_{t-1}^*+\by_{t-1}^*-\by_{t-1}\|^2\leq \Big(1+\frac{1}{\alpha}\Big)\frac{L_{t-1}^2}{\mu_{t-1}^2} \|\bx_{t}-\bx_{t-1}\|^2 + (1+\alpha) \delta_{t-1},
    \end{aligned}
\end{equation*}
where we use the fact that $\by^*(\cdot)$ is $\kappa$-Lipschiz from \cref{lem:primal_lipschitz} and $L_{t-1}=L_{t}\geq L$ \rv{and $\mu_{t-1}=\mu_t \in (0, \mu]$}. Combining it with \eqref{delta_ineq1_AGDA} implies 
\begin{equation*}
    \begin{aligned}
        \delta_{t}
        &\leq (1-\sigma_{t-1} {\mu_{t-1}})(1+\alpha)\delta_{t-1}+ (1-\sigma_{t-1} \rv{\mu_{t-1}})\Big(1+\frac{1}{\alpha}\Big)\frac{L_{t-1}^2}{\mu_{t-1}^2} \tau_{t-1}^2 \| G_{x}^{\tau_{t-1}} (\bx_{t-1},\by_{t-1})\|^2\\
        &\quad - \sigma_{t-1}^2 (1-\sigma_{t-1} l_{t-1} )\| G_y^{\sigma_{t-1}} (\bx_{t},\by_{t-1})\|^2 .
    \end{aligned}
\end{equation*}
Let $\alpha = \frac{1}{2}\frac{{\mu_{t-1}}\sigma_{t-1}}{1-{\mu_{t-1}}\sigma_{t-1}}$. Note that $\alpha> 0$ since \rv{$\sigma_{t}=\frac{1}{l_t}\leq\frac{1}{\tilde l}<\frac{1}{\tilde\mu}$ and $\mu_t\leq\tilde\mu$} %holds 
for all $t\geq 0$ \rv{--here we used that $\tilde l>\tilde \mu$ due to initialization, e.g., $\tilde l=\tilde\mu/\gamma_2$ is a possible choice for $\tilde l$ for the given $\tilde \mu>0$.} 
%This choice implies that
Thus,
\begin{equation*}
    \begin{aligned} %\label{delta_bd_inner} 
        \delta_{t}&\leq \Big(1-\frac{1}{2}\sigma_{t-1} {\mu_{t-1}}\Big)\delta_{t-1} + \frac{(1-\sigma_{t-1} {\mu_{t-1}})(2-\sigma_{t-1} {\mu_{t-1}} )}{\sigma_{t-1} {\mu_{t-1}}} \frac{L_{t-1}^2}{\mu_{t-1}^2} \tau_{t-1}^2 \| G_{x}^{\tau_{t-1}} (\bx_{t-1},\by_{t-1})\|^2.
    \end{aligned}
\end{equation*}
\end{proof}
\begin{lemma}\label{lemma:F-common-gap}
    \xqs{Suppose \cref{ASPT:lipshiz gradient,aspt:bounded-Y} hold.} Given $\bar\bx,\bar\by\in\dom g\times\dom h$, let $\hat\by = \prox{\sigma h}(\bar\by + \sigma \na{\grad_y} f(\bar\bx,\bar\by))$ for some $\sigma>0$. Then,
    \begin{equation}
        \cL(\bar\bx,\by)  - \cL(\bar\bx,\hat\by) \leq \Big(\frac{1}{\sigma}+\rv{L}\Big)\rv{\norm{\bar\by-\hat\by}\norm{\by-\hat\by}},\quad\forall~\by\in\cY;
    \end{equation}
    therefore, $F(\bar\bx)=\max_{y\in\cY}\cL(\bar\bx,\by)\leq \cL(\bar\bx,\hat\by)+\Big(\frac{1}{\sigma}+\rv{L}\Big)\cD_y^2$.
\end{lemma}
\begin{proof}
    Since $\hat\by = \prox{\sigma h}(\bar\by + \sigma \grad_y f(\bar\bx,\bar\by))$, from the first-order optimality condition, we obtain that
    \begin{equation}
       -\frac{1}{\sigma}(\hat\by -\bar\by) + \grad_y f(\bar\bx,\bar\by)\in \partial h(\hat\by).
    \end{equation}
    Moreover, \rv{it follows from \cref{ASPT:lipshiz gradient} %and \cref{aspt:bounded-Y} 
    that}
    \begin{equation}
        \begin{aligned}
            \rv{f(\bar\bx,\by) - f(\bar\bx,\hat\by)\leq \langle \grad_y f(\bar\bx,\hat\by),~\by-\hat\by \rangle.} %+ \frac{L}{2}\cD_y^2.
        \end{aligned}
    \end{equation}
    Using the convexity of $h$, we further obtain that
    \begin{equation}
        \begin{aligned}
             \MoveEqLeft f(\bar\bx,\by) -h(\by) - f(\bar\bx,\hat\by) +  h(\hat\by) 
             \\
             \leq &
              \langle \grad_y f(\bar\bx,\hat\by) + \frac{1}{\sigma}(\hat\by -\bar\by) - \grad_y f(\bar\bx,\bar\by),~\by-\hat\by\rangle 
              %+ \frac{L}{2}\cD_y^2
              \\
              \leq &
              \Big(\| \grad_y f(\bar\bx,\hat\by)- \grad_y f(\bar\bx,\bar\by)\| + \frac{1}{\sigma}\|\hat\by -\bar\by\|\Big)\cdot \| \hat\by - \by\|.
              %+ \frac{L}{2}\cD_y^2
        \end{aligned}
    \end{equation}
    The above inequality together with \cref{ASPT:lipshiz gradient,aspt:bounded-Y} leads to the desired conclusion.
\end{proof}
\section{Derivation of the complexity for other works}
%\nsa{Make corrections according to LL paper}
\label{sec:complexity-others}
\subsection{Derivation of complexity in \cite{yang2022nest}}\label{comp:neada}
The nested adaptive (\neada) algorithm by \cite{yang2022nest} \sa{is a double-loop method \rv{to solve $\min_{\bx\in\cX}\max_{\by\in Y}f(\bx,\by)$ for some closed convex set $Y$}, where the inner loop is to inexactly maximize the coupling function for a given primal iterate 
%with controllable stopping criteria 
and in the outer loop, the primal variable is updated \rv{using the gradient of the coupling function} computed at the current primal iterate and \mg{an} inexact dual maximizer.} For %nonconvex-strongly-concave 
\sa{WCSC} minimax problems, \neada{} with adaptive stepsizes can achieve the near-optimal $\Tilde{\cO}(\epsilon^{-2})$ 
%and $\Tilde{\cO}(\epsilon^{-4})$ 
\rv{gradient complexity 
%respectively 
in the deterministic 
%and stochastic 
gradient setting.} \sa{However, %in the paper 
the dependence of $\cO(1)$ constant on $\kappa$ and $L$ is not explicitly stated in~\cite{yang2022nest}; and in the rest of this section, we compute the important terms %forming
within the $\cO(1)$ constant.} \rv{For all $\bx\in\cX$ and $\by\in \cY$, let $G_y(\bx,\by)\triangleq\norm{\by-\Pi_{Y}(\by+\grad_y f(\by,\by))}$, where $\Pi_Y(\cdot)$ denotes the Euclidean projection onto $Y$.} \sa{According to the proof of \cite[Theorem 3.2]{yang2022nest}, %the total complexity for 
\neada{} in the deterministic setting %can compute
has a theoretical guarantee that
\begin{equation*}
    \rv{\norm{\grad_x f(\bx_T,\by_T)}^2+\kappa^2\norm{G_y(\bx_T,\by_T)}^2}\leq \epsilon^2
\end{equation*}
for some $T\leq \Tilde{\cO}\Big( \Big((A+\cE)^2+\sqrt{v_0}(A+\cE) 
%+ \frac{(L+1)^2}{\mu^2 a_2}
\Big)\epsilon^{-2}\Big)$} with $A+\cE = \tilde{\cO}\Big(\frac{\Phi(x_0)- \min_x \Phi(x)}{\eta} + \kappa L \eta + \frac{\kappa^2 (L+1)^2}{\sqrt{v_0}} \Big)$ for some $v_0,\eta>0$ \sa{--see~\cite[Theorem B.1]{yang2022nest} in the appendix of the paper for the details; moreover, for each outer iteration $t=1,\ldots, T$, subroutine $\cA$ to inexact solve $\max_{\by\in Y} f(\bx_t,\by)$ requires $\cO(\frac{1}{a_2}\log(t))$ gradient calls for some $a_2\in (0,1)$ constant specific to the subroutine $\cA$ that depend on structural properties of $f(\bx,\cdot)$ uniformly in $\bx\in\cX$ such as $L$ and $\mu$. Since $\sigma = 0$, if one uses accelerated \xz{backtracking method for solving strongly convex problems}~\cite{calatroni2019backtracking,rebegoldi2022scaled} 
%\nsa{I sent you papers that can achieve this result using backtracking on $L$, cite them here!} 
as the subroutine $\cA$ \rv{for the case $\mu$ is known}, then $a_2 = \frac{1}{\sqrt{\kappa}}$. %when , then the required computational 
Therefore, the total gradient complexity for \neada{} in the deterministic case} is at least $\sa{\Tilde{\cO}\Big( \Big((A+\cE)^2+\sqrt{v_0}(A+\cE) 
%+ \frac{(L+1)^2}{\mu^2 a_2}
\Big)\epsilon^{-2}/a_2\Big)}=\tilde \cO (\kappa^{4.5} L^4 \epsilon^{-2})$. \rv{On the other hand, for the case $\mu$ is unknown~\cite{bao2023global}, one can only get $a_2=\frac{1}{\kappa}$; thus, the complexity can be bounded by $\tilde \cO (L^4\kappa^{5} \epsilon^{-2})$.}

\subsection{Derivation of complexity in \cite{li2022tiada}} \label{comp:tiada}
%The time-scale adaptive (TiAda) algorithm in \cite{li2022tiada} is a single-loop adaptive GDA algorithm for nonconvex strongly-concave minimax optimization that can achieve near-optimal complexities by finding the time-scale separation between the primal and dual stepsizes.\\ \\
To get the near-optimal complexities, their theoretical analysis %suggests to choose 
\sa{requires} $\alpha>\frac{1}{2}>\beta$. %Then we will use 
\sa{To come up with explicit $\kappa$ and $L$ dependence of their $\cO(1)$ constant, we will use this as well.} %this condition to compute the complexities.
\sa{The complexity for the deterministic case is discussed in \cite[Theorem C.1]{li2022tiada}, and they show that 
\begin{equation}
\label{eq:x-complexity}
\sum_{t=0}^{T-1}\norm{\grad_x f(\bx_t,\by_t)}^2\leq \max\{5C_1,2C_2\}\triangleq C,
\end{equation}
for some constants $C_1$ and $C_2$, which we will analyze next.}  %According to the proof of Theorem C.1,} 
\xz{Indeed, both $C_1$ and $C_2$ are explicitly defined in the statement of Theorem C.1 depending on some other positive constants, i.e., $c_1,c_2,c_3,c_4>0$ and $c_5>0$, which are defined as follows\footnote{\sa{In~\cite{li2022tiada}, there are typos in the definition of $c_1$ and $c_3$; $v_{t_0}^y$ appearing in $c_1$ and $c_3$ should be $v_0^y$.}}:
%from the expression in \cite[Theorem C.1]{li2022tiada}, where
\begin{equation}
\begin{aligned}
        & c_1= \frac{\eta_x \kappa^2}{\eta_y(v_{0}^y)^{\alpha-\beta}},\quad  c_2 =\max\Big\{\frac{4\eta_y \mu L}{\mu+ L},\ \eta_y (\mu+L)\Big\},\quad c_3 = 4(\mu+L) \Big(\frac{1}{\mu^2} + \frac{\eta_y}{(v_{0}^y)^\beta} \Big)c_2^{1/\beta},\\
        & c_4= (\mu+L)\Big( \frac{2\kappa^2}{(v_0^y)^\alpha}+ \frac{(\mu+L)\kappa^2}{\eta_y \mu L}\Big),\quad c_5 = c_3 + \frac{\eta_y v_0^y}{(v_0^y)^\beta}+ \frac{\eta_y c_2^{\frac{1-\beta}{\beta}}}{1-\beta}.
\end{aligned}
\end{equation}
Because $C_1$ and $C_2$ are monotonically increasing in $\{c_i\}_{i=1}^5$ and \sa{they have complicated forms, 
%to compute exact quantities, 
we compute a lower bound for each $c_i$, $i=1,\ldots,5$ and use these bounds to further bound $C_1$ and $C_2$ from below. This will allow us to provide a lower bound on \texttt{TiAda} complexity results in terms of its dependence on $L,\mu$ and $\kappa$. Initialization of \tiada{} requires setting six parameters: $\eta_x,\eta_y,v_0^x,v_0^y>0$ and $\alpha,\beta\in(0,1)$ such that $\alpha>\beta$. Since the problem parameters $L,\mu$ and $\kappa$ are unknown, we treat all these parameters as $O(1)$ constants.} Indeed,
\begin{enumerate}
    \item[(i)] consider $c_1$, treating $\frac{\eta_x}{\eta_y(v_{0}^y)^{\alpha-\beta}}$ as $O(1)$ constant, we have $c_1= \Theta(\kappa^2)$;
    \item[(ii)] for $c_2$, we similarly treat $\eta_y$ as $O(1)$ constant, and get $c_2= \Theta(L)$;
    \item[(iii)] for $c_3$, since $c_2= \Theta(L)$ and we treat $ \frac{\eta_y}{(v_{0}^y)^\beta}$ as $O(1)$ constant, we get $c_3= \Theta(\frac{\kappa L^{1/\beta}}{\mu})$;
    \item[(iv)]  similarly, we have the bounds: $c_4= \Theta(\kappa^3)$ and $c_5= \Theta(\frac{\kappa L^{1/\beta}}{\mu})$. 
    %where we treat $\alpha,\beta,\eta_x,\eta_y,v_{0}^x,v_{0}^y$ as $O(1)$ constants.
\end{enumerate}
Next, we %will find a proper 
\sa{derive} lower bound for $C_1$ and $C_2$ \sa{in terms of $L,\mu$ and $\kappa$}. We first %analyze the components of 
\sa{focus on} $C_1$. Indeed, \sa{it follows from the definition of $C_1$ in \cite[Theorem C.1]{li2022tiada} that} %under the assumption that $\alpha,\beta,\eta_x,\eta_y,v_{0}^x,v_{0}^y$ are $O(1)$ constant, we can obtain that 
\begin{equation}
    C_1 = \Omega\Big( 
    %(c_1c_5)^{\frac{1}{1-\alpha}}, 
    (c_1c_4)^{\frac{1}{\alpha-\beta}}\mathbf{1}_{2\alpha-\beta<1}+ (c_1c_4)^{\frac{2}{1-\alpha}}\mathbf{1}_{2\alpha-\beta\geq 1} \Big)
\end{equation}
where \sa{$\mathbf{1}_{2\alpha-\beta<1}= 1$ if $2\alpha-\beta<1$; otherwise, it is $0$, and we define $\mathbf{1}_{2\alpha-\beta\geq 1}=1-\mathbf{1}_{2\alpha-\beta<1}$.} 
\sa{Although $\alpha,\beta\in (0,1)$, it is necessary to consider their effects when they appear as exponents.}
%calculated at an exponential position. 
Therefore, we consider the following two cases:
\begin{enumerate}
    \item when $2\alpha-\beta \geq 1$, since $\beta>0$, we can conclude that $\alpha > \frac{1}{2}$. \sa{As a result, $(c_1c_4)^{\frac{2}{1-\alpha}}\geq(c_1c_4)^4=\Theta(\kappa^{20})$; therefore,} 
    %we can get the lower bound of $C_1$ as follows
    \begin{equation*}
        C_1 %\geq \Theta\Big(\Big) >\Theta\Big((c_1c_4)^4\Big) \geq   \Theta
        \sa{=\Omega\Big(\kappa^{20}\Big);}
    \end{equation*}
    \item when $2\alpha-\beta  < 1$, it implies $\alpha-\beta<\frac{1-\beta}{2}<\frac{1}{2}$. As a result, we can %get the lower bound of 
    \sa{bound $C_1$ from below as follows: $(c_1c_4)^{\frac{1}{\alpha-\beta}}\geq (c_1c_2)^2=\Theta(\kappa^{10})$; therefore,}
    \begin{equation*}
        C_1 =\Omega\Big(\kappa^{10}\Big).
    \end{equation*}
\end{enumerate}
\sa{Therefore, using \eqref{eq:x-complexity} and $C\geq C_1$, we can conclude that for any $\epsilon>0$, the gradient complexity of \tiada{} for $\sum_{t=0}^{T-1}\norm{\grad_x f(\bx_t,\by_t)}^2\leq \epsilon^2$ to hold is at least $\cO(\kappa^{10}\epsilon^{-2})$. Therefore, for the deterministic WCSC minimax problems, the gradient complexity of \tiada{} to compute an $\epsilon$-stationary point in terms of metric \textbf{(M2)} is at least $\cO(\kappa^{10}\epsilon^{-2})$.}}
\subsection{Derivation of complexity in \cite{lu2020hybrid}}
The hybrid successive approximation~(\texttt{HiBSA}) algorithm is a single-loop proximal alternating algorithm for solving deterministic nonconvex minimax problems $\min_{\cX}\max_{y\in\cY}\cL(\bx,y)$, where $\cL$ has the same form in~\eqref{eq:main-problem}. \sa{\texttt{HiBSA} uses block coordinate updates for the primal variable that has $N\geq 1$ blocks. Moreover, it employs} constant step sizes for primal and dual updates, and to set these step sizes one needs to know the \mg{L}ipschitz constant $L$ and strongly concavity modulus $\mu$. \sa{Convergence results provided in the paper are in terms of \textbf{(M2)} metric.}

\sa{For the WCSC setting, according to \cite[Lemma 3]{lu2020hybrid}, whenever dual step size $\rho$ and primal step size $\frac{1}{\beta}$ are chosen satisfying the following two conditions: $0< \rho < \frac{\mu}{4L^2}$ and $\beta > L^2(\frac{2}{\mu^2 \rho}+\frac{\rho}{2})+\frac{L}{2}-{L}$, it is guaranteed that the \texttt{HiBSA} iterate sequence satisfies the following bound for all $k\geq 0$:
%\todo[inline]{NSA: $\beta > L^2(\frac{2}{\mu^2 \rho}+\frac{\rho}{2})+\frac{L}{2}-\blue{L}$, $c_2=\beta + \blue{L}-\frac{L}{2}-L^2(\frac{2}{\mu^2 \rho}+\frac{\rho}{2})$}
\begin{equation}
    \begin{aligned}
         c_1\norm{y^{k+1}-y^k}^2 + c_2 \sum_{i=1}^N \norm{x_i^{k+1}-x_i^k}^2 \leq \cP^k - \cP^{k+1},
    \end{aligned}
\end{equation}
where $c_1\triangleq 4(\frac{1}{\rho}- \frac{L^2}{2\mu})-\frac{7}{2\rho}>0$ and $c_2\triangleq \beta + {L}-\frac{L}{2}-L^2(\frac{2}{\mu^2 \rho}+\frac{\rho}{2})>0$, and 
$\cP^{k}\triangleq \cL(x^{k},y^{k})+\Big(\frac{2}{\rho^2 \mu}+\frac{1}{2\rho} - 4(\frac{1}{\rho} - \frac{L^2}{2\mu}) \Big)\norm{y^{k}-y^{k-1}}^2$, which is the value of the potential function at iteration $k\geq 0$ --here, $c_1,c_2>0$ follows from the conditions on $\rho$ and $\beta$. Furthermore, from the proof of \cite[Theorem 1]{lu2020hybrid}, for all $k\geq 0$ we get}
\begin{equation}
    \begin{aligned}
       &\norm{G_{x_i}(x^k,y^k)}\leq (\beta+ 2L) \norm{x_i^{k+1}-x_i^k},\quad i=1,\ldots,N,\\
       &\norm{G_y(x^k,y^k)}\leq L\norm{x^{k+1}-x^k}  + \frac{1}{\rho}\norm{y^{k+1}-y^k},
    \end{aligned}
\end{equation}
\sa{which immediately implies that}
\begin{equation}
    \begin{aligned}
        \norm{G(x^k,y^k)}^2 &\leq \Big((\beta +2L)^2 + 2L^2 \Big) \norm{x^{k+1}-x^k}^2 + \frac{2}{\rho^2}\norm{y^{k+1}-y^k}^2 \leq \frac{\cP^{k}-\cP^{k+1}}{d_1},\quad \forall~k\geq 0,
    \end{aligned}
\end{equation}
where \sa{$d_1 \triangleq \min \{ 4(\frac{1}{\rho}-\frac{L^2}{2\mu})-\frac{7}{2\rho},~\beta +{L} - \frac{L}{2}-L^2(\frac{2}{\mu^2 \rho} + \frac{\rho}{2})\}/ \max \{\frac{2}{\rho^2},~(\beta +2L)^2 + 2L^2\}$ --note that $d_1>0$ due to conditions on $\rho$ and $\beta$. Therefore, for any $K\geq 1$, one has 
%$\rho < \frac{\mu}{4L^2}$, $\beta > L^2(\frac{2}{\mu^2 \rho}+\frac{\rho}{2})+\frac{L}{2}-\mu$.
\begin{equation}
\label{eq:hibsa-bound}
    \frac{1}{K}\sum_{k=1}^K\norm{G(x^k,y^k)}^2\leq \frac{1}{K}\cdot\frac{P^1-\underline{\cL}}{d_1}\leq\frac{1}{K}\cdot\frac{\cL(x^1,y^1)-\underline{\cL}+c_3\cD_y^2}{d_1},
\end{equation}
where $c_3\triangleq \frac{2}{\rho^2 \mu}+\frac{1}{2\rho} - 4(\frac{1}{\rho} - \frac{L^2}{2\mu})$, $\cD_y$ is the diameter of the dual domain, and it is assumed that there exists $\underline{\cL}>-\infty$ such that $\cL(x,y)\geq \underline{\cL}$ for all $x\in\cX$ and $y\in\cY$. Note that $c_3>\frac{2}{\rho^2\mu}>32 L\kappa^3$, where we used $0<\rho<\frac{\mu}{4L^2}$; moreover, we also have $\beta>\frac{2L^2}{\mu^2\rho}-%\mu
{\frac{L}{2}}>\frac{2\kappa^2}{\rho}-%\mu
{L}$; therefore, 
{\begin{eqnarray*}
   \frac{1}{d_1}\geq \frac{(\beta +2L)^2 + 2L^2}{4(\frac{1}{\rho}-\frac{L^2}{2\mu})}&\geq& \frac{\kappa^2}{2}\cdot\frac{\beta^2}{\frac{2\kappa^2}{\rho}-L\kappa^3}\\
   &\geq& \frac{\kappa^2\beta}{2}\cdot\frac{\frac{2\kappa^2}{\rho}-%\mu
   {L}}{\frac{2\kappa^2}{\rho}-L\kappa^3}\geq \frac{\kappa^2\beta}{2}\geq \frac{\kappa^2}{2}\Big(\frac{2\kappa^2}{\rho}-%\mu
 {L}\Big)\geq{\kappa^2}\Big(4\kappa^3 L-\frac{%\mu
   {L}}{2}\Big)=\Omega(L\kappa^5).
\end{eqnarray*}%
}
Thus, combining this bound with \eqref{eq:hibsa-bound} and using $c_3>32 L\kappa^3$, we get $\frac{\cL(x^1,y^1)-\underline{\cL}+c_3\cD_y^2}{d_1}=\Omega(L^2\kappa^8\cD_y^2)$.
%Then by using the largest possible step sizes $\beta^{-1} = \cO(L^{-1}\kappa^{-3})$ and $\rho = \cO(L^{-1}\kappa^{-1})$ that satisfies $\rho < \frac{\mu}{4L^2}$, $\beta > L^2(\frac{2}{\mu^2 \rho}+\frac{\rho}{2})+\frac{L}{2}-\mu$, we know that $\frac{1}{d_1}\geq \Omega(L^{1}\kappa^{5})$, and $\cP^1- \underline{\cL} =\cL(x^1,y^1) + \cO(L\kappa^3 \cD_y^2)-\underline{\cL}$. Therefore,
This result implies that the total complexity of HiBSA to compute an $\epsilon$-stationary point for WCSC minimax problems is at least $\cO(\kappa^8 L^2\epsilon^{-2})$.}
\subsection{Derivation of complexity in~\cite{xu2023unified}}
The authors of~\cite{xu2023unified} consider $\min_{x\in \cX}\max_{y\in \cY}\cL(x,y)=\mathds{1}_{X}(x)+f(x,y)-\mathds{1}_{Y}(y)$, where $X$ and $Y$ are closed convex sets, and $f$ is a smooth function such that $f$ is strongly concave in $y$. It is shown that the \agda~method shown in~\eqref{eq:gda} with $g$ and $h$ being the indicator functions of $X$ and $Y$, respectively, is guaranteed to generate an $\epsilon$-stationary point in terms of metric \textbf{(M2)} within $\frac{F_0-\underline{F}}{d_1\epsilon^2}$ iterations whenever the primal ($\tau$) and dual ($\sigma$) step sizes satisfy $\tau<\min\{\frac{1}{L},~\frac{1}{L^2\sigma+4\kappa^2/\sigma}\}$ and $\sigma\leq \frac{1}{4L\kappa}$, where $d_1=\min\{\frac{1}{2\tau}-\frac{\sigma L^2}{2}-\frac{2\kappa^2}{\sigma},~\frac{3\mu-\sigma L^2}{2}+\frac{\mu-4\sigma L^2}{2\sigma\mu}\}/\max\{\frac{1}{\tau^2}+2L^2,~\frac{2}{\sigma^2}\}$, $F_0=\cO(\cL(x_0,y_0))$, $\underline{F}=\cL^*-(\mu+\frac{7}{2\sigma}-\frac{\sigma L^2}{2}-2L\kappa)\cD_y^2$ and $\cL^*=\inf_{x\in X, y\in Y}\cL(x,y)$. Thus, choosing $\tau=\cO(\frac{1}{L\kappa^3})$ and $\sigma=\Theta(\frac{1}{L\kappa})$ implies that $\frac{1}{d_1}=\Omega(L\kappa^5)$ and $F_0-\underline{F}=\Omega(\cL(x_0,y_0)-\cL^*+L\kappa\cD_y^2)$; therefore, the upper complexity provided by the authors can be bounded from below as follows:
\begin{align*}
    \frac{F_0-\underline{F}}{d_1\epsilon^2}=\Omega\Big(L\kappa^5(\cL(x_0,y_0)-\cL^*+L\kappa\cD_y^2)\epsilon^{-2}\Big).
\end{align*}

\end{document}